\documentclass[]{imsart}
\RequirePackage[OT1]{fontenc}
\RequirePackage{amsthm,amsmath,amssymb,mathrsfs,amsfonts,extarrows, xcolor}
\usepackage[round,authoryear]{natbib}
\RequirePackage[colorlinks,citecolor=blue,urlcolor=blue]{hyperref}
\usepackage{txfonts}
\startlocaldefs
\numberwithin{equation}{section}
\theoremstyle{plain}
\newtheorem{thm}{Theorem}[section]
\theoremstyle{remark}

\endlocaldefs
\numberwithin{equation}{section}

\newtheorem{lemma}[thm]{Lemma}
\newtheorem{remark}[thm]{Remark}

\newtheorem{definition}[thm]{Definition}

\newcommand{\bss}{\boldsymbol}
\newcommand{\bs}{{\bf S}}
\newcommand{\br}{{\bf R}}
\newcommand{\bi}{{\bf I}}
\newcommand{\bx}{{\bf X}}
\newcommand{\by}{{\bf Y}}
\newcommand{\bh}{{\bf H}}
\newcommand{\bt}{{\bf T}}
\newcommand{\cC}{{\mathcal C}}
\newcommand{\ba}{{\bf A}}
\newcommand{\bb}{{\bf B}}
\newcommand{\bw}{{\bf W}}
\newcommand{\bm}{{\bf M}}
\newcommand{\bS}{{\boldsymbol\Sigma}}

\newcommand{\bd}{{\bf D}}

\newcommand{\be}{{\bf e}}

\newcommand{\bbx}{{\bf x}}
\newcommand{\bq}{{\bf q}}
\newcommand{\E}{{\rm E}}
\newcommand{\tr}{{\rm tr}}
\newcommand{\mb}{\mathbf}
\newcommand{\re}{{\rm E}}
\newcommand{\rtr}{{\rm tr}}
\renewcommand{\(}{\left(}
\renewcommand{\)}{\right)}

\allowdisplaybreaks
\begin{document}

\begin{frontmatter}
\title{Central limit theorem for linear spectral statistics of general separable sample covariance matrices with applications}
\runtitle{CLT of LSS for separable sample covariance matrices}

\begin{aug}
\author{\snm{Huiqin Li}\thanksref{a}\ead[label=e1]{huiqinli@jsnu.edu.cn}}
\author{\snm{Yanqing Yin}\thanksref{b}\ead[label=e2]{yinyq@jsnu.edu.cn}}
\author{\snm{Shurong Zheng}\thanksref{c}\ead[label=e3]{zhengsr@nenu.edu.cn}}

\thanks{{Huiqin Li was partially supported by NSFC 11701234 and the Priority Academic Program Development of Jiangsu Higher Education Institutions. Yanqing Yin was partially supported by a project of a Grant NSFC11801234, a
Project of Natural Science Foundation of Jiangsu Province (BK20181000), a project of The Natural Science
Foundation of the Jiangsu Higher Education Institutions of China (18KJB110008) and the Priority Academic
Program Development of Jiangsu Higher Education Institutions. Shurong Zheng is supported by NSFC 11522105 and 11690012.}}

\address[a]{{School of Mathematics and Statistics, Jiangsu Normal University, Xuzhou, China, 221116.}\printead{e1}}
\address[b]{{School of Mathematics and Statistics, Jiangsu Normal University, Xuzhou, China, 221116.}\printead{e2}}
\address[c]{{KLASMOE, Northeast Normal University, China, 130024.}\printead{e3}}

\runauthor{Huiqin Li, Yanqing Yin and Shurong Zheng}


\end{aug}

\begin{abstract}
In this paper, we consider the separable covariance model, which plays an important role in wireless communications and spatio-temporal statistics and describes a process where the time correlation does not depend on the spatial location and the spatial correlation does not depend on time. We established a central limit theorem for linear spectral statistics of general separable sample covariance matrices in the form of $\bs_n=\frac1n\bt_{1n}\bx_n\bt_{2n}\bx_n^*\bt_{1n}^*$ where $\bx_n=(x_{jk})$ is of $m_1\times m_2$ dimension, the entries $\{x_{jk}, j=1,...,m_1, k=1,...,m_2\}$ are independent and identically distributed complex variables with zero means and unit variances, $\bt_{1n}$ is a $p\times m_1 $ complex matrix and $\bt_{2n}$ is an $m_2\times m_2$ Hermitian matrix. We then apply this general central limit theorem to the problem of testing white noise in time series.
\end{abstract}


\end{frontmatter}

\section{Introduction}
\subsection{Background and Motivation}
 Covariance matrices play an important role in modern multivariate analysis \citep{Anderson1983An}. In the framework of the classical statistical theory, suppose that $\mb y_1,\mb y_2,\cdots,\mb y_n$ are the samples with the sample size $n$ drawn from a centered $p$ dimensional population $\mathbb{Y}\in \mathbb{R}^n$. Then if $n$ tends to infinity while the dimension $p$ is fixed, the so-called sample covariance matrix, defined as $$\mb S_n=n^{-1}\sum_{i=1}^n\mb y_i\mb y_i^*,$$ is a good estimator of the population covariance matrix $\Sigma=\E\mathbb{Y}\mathbb{Y}^*$ where $*$ denotes the transpose and conjugate. However, in the biological and genetic study, millions of genes are measured for individuals where the number of individuals is small compared with the number of genes \citep{Patterson2006Population}. That is to say, in many situations in modern statistics, we need to deal with data sets where the dimension $p$ is comparable or even large compared with the sample size $n$. In this setting, we face the ``curse of dimensionality" \citep{donoho2000high}, which drives the drastic changes for modern statistical theory and promotes the development of high-dimensional statistical inference \citep{meinshausen2006high,B2014High,Goia2016An}. The following fact gives us a look into this phenomenon. By the strong law of large number, for any $1\leq j, k\leq p$, the $(j, k)$ entry $s_{j,k,n}$ of $\mb S_n$ is a consistent estimator of the corresponding element $\sigma_{j,k}$ being the $(j,k)$ entry of $\mb \Sigma$. Then applying the eigenvalue perturbation theorem, we know that for any $1\leq j\leq p$, the distance between the $j$-th largest eigenvalues of $\mb S_n$ and $\mb \Sigma$ is $o(p)$, which tends to 0 as $n\to\infty$ when $p$ is of constant order. However, when $p$ is of the same order with $n$ or of a larger order than $n$, the bounds of the distances between the eigenvalues of $\mb S_n$ and $\mb \Sigma$ will blow up. Then, $\Sigma$ cannot be estimated through $\mb S_n$ directly \citep{Fan2008High,Chen2013Covariance,Cai2010OPTIMAL}.

Although $\mb S_n$ is no longer a good estimator for the population covariance matrix in high dimensional framework, some properties of $\Sigma$ can still be obtained through the eigenvalue statistics of $\mb S_n$, such as in the standard technique of multivariate statistics, principal components analysis
(PCA) \citep{Johnstone2001On}. This leads to the high-dimensional statistical inference and the random matrix theory, which mainly focuses on the properties of the eigenvalue and eigenvectors of random matrices. In the random matrix theory, the sample covariance type matrices are one of the different types of random matrices that have been investigated by many authors. We firstly introduce the following definitions. Let $\ba$ be any $n \times n$ square matrix having real eigenvalues and denote its eigenvalues  by ${\lambda_j}, j = 1,2, \cdots, n$. Then the $\mb {Empirical\ Spectral\ Distribution} $ (ESD) of $\ba$ is defined by
$$F^{\ba}\left(x\right) =\frac{1}{n}\sum\limits_{j = 1}^n {I\left({\lambda _j} \le x\right)},$$
where $I_{(\cdot)}$ is the indicator function and the Stieltjes transform of ${F^{\mathbf A}}\left(x\right)$ is given by
$${ m}_{F^{\ba}}\left(z\right)=\int_{-\infty}^{+\infty}(x-z)^{-1}d{F^{\mathbf A}}\left(x\right),$$
where $z=u+ iv\in\mathbb{C}^+$.

In this paper, we consider a more general covariance matrix, whose spectrum properties, to our best of knowledge, have not been considered before. To be specific, we consider the general separable covariance matrices $${\mb T_{1n}\mb X_n\mb T_{2n}\mb X_n^*\mb T_{1n}^*}/{n},$$ where $\mb X_n$ is an $m_1\times m_2$ random matrix whose entries are i.i.d. with zero means and unit variances while $\mb T_{1n}$ a $p\times m_1$ matrix and $\mb T_{2n}$ an $m_2\times m_2$ Hermitian matrix. In fact, the sample covariance matrix has applications in many fields such as wireless communications \citep{Verdu2002Spectral} and spatio-temporal statistics \citep{Li2008Testing,Mitchell2003Spatio}. Indeed, if $\mb T_{2n}=\mb T_{3n}\mb T_{3n}^*$ with $\mb T_{3n}$ being of $m_2\times n$ dimension, then the joint covariance of $\mb Y_n=\mb T_{1n}\mb X_n\mb T_{3n}$, viewed as a $pn\times 1$ data vector, is given by $$(\mb T_{1n}\mb T_{1n}^*)\otimes(\mb T_{3n}^*\mb T_{3n}).$$  Thus $\mb Y_n$ is a data matrix whose rows correspond to indices of spatial locations and columns correspond to indices of point in time. In particular, when the entries of $\mb X_n$ are Gaussian, the joint distribution of $\mb Y_n$ is $N_{pn}\(0,\(\mb T_{1n}\mb T_{1n}^*\)\otimes\(\mb T_{3n}^*\mb T_{3n}\)\)$. Note that the separable model describes a process where the time correlation does not depend on the spatial location and the spatial correlation does not depend on time, i.e. there is no space-time interaction. The introduced general separable covariance matrix model covers many covariance type matrices that have been well studied in random matrix theory as special cases.

\subsection{Some Primary Results}
Under the condition that $m_1=p, m_2=n$, $\mb T_{1n}=\mb I_p$ and $\mb T_{2n}=\mb I_n$ where $\mb I_p$ and $\mb I_n$ are the $p\times p$ and $n\times n$ identity matrix, our models reduce to the well studied ordinary sample covariance matrices. On the framework of high-dimensional setting, i.e., $p/n\to c\in(0,\infty)$ as $n\to \infty$, the ESDs of this kind of matrices converge to the famous M-P law \citep{marchenko1967distribution} almost surely as $n\to \infty$. The strong convergence of the extreme eigenvalues of ordinary sample covariance matrices was considered in \citep{geman1980limit, bai1988necessary, BaiYin1993, Tikhomirov2015The}. Suppose that the entries of $\mb X_n$ have finite fourth moment, then almost surely, the largest eigenvalue (the spectrum norm) of $\mb S_n=\mb X_n\mb X_n^*/{n}$, denoted as $\lambda_{max}(\mb S_n)$, tends to $(1+\sqrt{c})^2$ (the right edge of the support of standard M-P law) while the smallest non-zero eigenvaleu $\lambda_{min}(\mb S_n)$ tends to $(1-\sqrt{c})^2$ (the left edge of the support of standard M-P law). It is also known that the finite fourth moment is a necessary condition for the strong convergence of $\lambda_{max}(\mb S_n)$ while surprisingly, the existence of variance is enough for the almost sure convergence of $\lambda_{min}(\mb S_n)$. The fluctuation of the extreme eigenvalues was considered in \citep{tracy2002distribution,Johnstone2001On} and it was proved that the standardized largest eigenvalue follows the famous T-W law as $n\to\infty$. It is worth to note that the strong convergence and fluctuations of extreme eigenvalues are two independent results and one can not deduce the strong convergence result from T-W law.

When $m_1=p, m_2=n$, $\mb T_{2n}=\mb I_n$ while $\mb T_{1n}$ is non-negative definite with its limiting spectral distribution exists, then M-P law still valid. This extensive model relaxes the condition of uncorrelation between entries of population $\mathbb{Y}$. At this stage, if (1): $p/n\to c\in(0,\infty),$ (2):$F^{\mb T_{1n}^2}\xrightarrow{D}H$, where $F^{\mb T_{1n}^2}$ is the ESD of ${\mb T_{1n}^2}$, $H$ is a c.d.f and (3): $\mb T_{1n}$ is bounded in spectral norm, then almost surely, the ESD $F^{\bs_n}$ of $\mb S_n=\frac{1}{n}\mb T_{1n}\mb X_n\mb X_n^*\mb T_{1n}$, tends weakly to a nonrandom p.d.f. $F$ as $n\to \infty$. And for each $z\in\mathbb{C}^+$, $m(z)=m_F(z)$ is a solution to the equation
\begin{align}\label{al3}
m(z)=\int\frac1{t(1-c-czm(z))-z}dH(t),
\end{align}
which is unique in the set $\left\{m(z)\in\mathbb{C}^+: -(1-c)/z+cm(z)\in\mathbb{C}^+\right\}$. Notice that $$\mb S_n=n^{-1}\sum_{i=1}^n \mb T_{1n}\mb x_i\mb x_i^*\mb T_{1n},$$ where $\mb x_i$ is the i-th column of $\mb X_n$. $\mb S_n$ can be viewed as the sample covariance matrix of the samples $\mb y_i=\mb T_{1n}\mb x_i, i=1,...,n$ drawn from the $p$ dimensional population $\mathbb{Y}=\mb T_{1n}\mathbb{X}$ where $\mathbb{X}$ is a random vector with standard i.i.d entries (mean 0 and variance 1). This ``linear" model, combined with conditions (1)-(3) ensures the convergence of ESD of $n^{-1}\sum_{i=1}^n \mb y_i\mb y_i^*$ to M-P distribution. Then the fluctuations of linear spectral statistics (LSS) of sample covariance matrices $\mb S_n$ was considered in \cite{bai2004clt}. For more results concerning the CLT of LSS of random matrices, we refer the readers to \citep{AndersonZ06C,LytovaP09C,LytovaP09Ca,ZhengB15S,ZhengB17C} and reference therein.

The case that $m_2=n$, $\mb T_{2n}=\mb I_n$ and $m_1>p$ can be arbitrary, has been considered by some authors recently. The LSD and CLT of LSS was established in \cite{zheng2017clt} while \cite{yin2018no} shows that no eigenvalues outside the limiting support for large sample size.
As for the separable covariance matrices, i.e., $m_1=p, m_2=n$ and $\mb T_{1n}$ being non-negative definite, \cite{zhang} firstly obtained the LSD of $\mb S_n$ under some conditions. Then \cite{paul2009no} proved the no eigenvalue outside result by assuming that $\mb T_{2n}$ is diagonal with nonnegative entries. Recently, \cite{liclt} established the CLT for LSS of $\mb S_n$ under the condition that the fourth moments of the variables in $\mb X_n$ equals 3.

\subsection{Model Definition and Main Results} We consider a more general model which could be useful in many statistical problems. The general separable sample covariance matrices is defined as follows.
\begin{definition}\label{def}
The sample covariance matrices $\bs_n=\frac1n\bt_{1n}\bx_n\bt_{2n}\bx_n^*\bt_{1n}^*$ is defined as the general separated sample covariance matrices if the following conditions are satisfied:
\begin{itemize}
\item[(a)] $\bx_n=(x_{jk})$ is of $m_1\times m_2$ dimension where $\{x_{jk},j=1,2,\cdots,m_1,k=1,2,\cdots,m_2\}$ are independent and identically distributed complex variables with mean zero and variance 1;
\item[(b)] $\bt_{1n}$ is a $p\times m_1 $ nonrandom complex matrix and $\bt_{2n}$ is a nonrandom $m_2\times m_2$ Hermitian matrix;
\item[(c)] With probability $1$, as $n\to\infty$, the empirical spectral distributions of $\boldsymbol\Sigma_1=\bt_{1n}\bt_{1n}^*$ and $\bt_{2n}$, denoted by $H_{1n}$ and $H_{2n}$, converge weakly to two probability functions $H_1$ and $H_2$, respectively;
\item[(d)] $c_n=p/n\to c\in(0,\infty)$ as $n\to\infty$ and $m_1\geq p, m_2\geq n$;
\end{itemize}
\end{definition}

Let us investigate the LSD of $\bs_n$ first.  It is well known that there exists unitary matrices $p\times p$ dimensional ${\bf U}_1$, $m_2\times m_2$ dimensional ${\bf U}_2$, $m_1\times m_1$ dimensional ${\bf V}_1$ and diagonal matrices $p\times p$ dimensional $\Lambda_1$ and $n\times n$ dimensional $\Lambda_2$ such that
  $$\bt_{1n}={\bf U}_1\left(\Lambda_1, {\bf 0}_{p\times(m_1-p)}\right){\bf V}_1^*,~\bt_{2n}={\bf U}_2{\rm diag}\left(\Lambda_2,{\bf 0}_{(m_2-n)\times(m_2-n)}\right){\bf U}_2^*.$$
Let
\begin{align*}
  \widetilde\bx_n\triangleq{\bf V}_1^*\bx_n{\bf U}_2=\begin{pmatrix}
  \widetilde\bx_{11}&\widetilde\bx_{12}\\
  \widetilde\bx_{21}&\widetilde\bx_{22}
  \end{pmatrix}.
\end{align*}
If we suppose that the entries of $\bx_n$ are standard complex normal random variables, then $\bx_n$ has the same distribution as $\widetilde\bx_n$. Note that the sample covariance matrix
$\bs_n=\frac1n{\bf U}_1\Lambda_1\widetilde\bx_{11}\Lambda_2\widetilde\bx_{11}^*\Lambda_1{\bf U}_1^*$
has the same eigenvalues with
$\widetilde\bs_n=\frac1n\Lambda_1\widetilde\bx_{11}\Lambda_2\widetilde\bx_{11}^*\Lambda_1.$
Then from \cite{zhang}, with probability $1$, as $n\to\infty$, the empirical spectral distribution function of $\bs_n$ converges weakly to a non-random probability distribution function $ F$ for which if $H_1=1_{[0,\infty)}$ or $H_{2}=1_{[0,\infty)}$, $ F=1_{[0,\infty)}$; otherwise if for each $z\in \mathbb{C}^+$,
\begin{align}\label{gal1}
\begin{cases}
m(z)=-z^{-1}(1-c^{-1})-z^{-1}c^{-1}\int\frac1{1+q_1(z)y}dH_2(y),\\
m(z)=-z^{-1}\int\frac1{1+q_2(z)x}dH_1(x),\\
m(z)=-z^{-1}-c^{-1}q_1(z)q_2(z),
\end{cases}
\end{align}
is viewed as a system of equations for the complex vector $\left(m(z),q_1(z),q_2(z)\right)$, then the Stieltjes transform of $F$, denoted by $m_F(z)$, together with the two other functions, denoted by $g_1(z)$ and $g_2(z)$, both of which are analytic on $\mathbb{C}^+$, will satisfy that $\left(m_F(z),g_1(z),g_2(z)\right)$ is the unique solution to (\ref{gal1}) in the set
$$U=\left\{\left(m(z),q_1(z),q_2(z)\right):\Im m(z)>0,\Im(zq_1(z))>0,\Im q_2(z)>0 \right\}.$$
On the other hand, if  \ $\bt_{1n}$ is real, $m_1=p$ and $\bt_{2n}$ is diagonal, then by \cite{paul2009no}, with probability 1, $F^{\bs_n}$ converges weakly to a probability distribution function $F$ whose Stieltjes transform $m(z)$, for $z\in\mathbb{C}^+$, is given by
\begin{align*}
  m(z)=\int\frac1{x\int\frac{y}{1+cye}dH_2(y)-z}dH_1(x)
\end{align*}
where $e=e(z)$ is the unique solution in $\mathbb{C}^+$ of the equation
\begin{align*}
  e=\int\frac{x}{x\int\frac{y}{1+cye}dH_2(y)-z}dH_1(x).
\end{align*}
Let $g_1(z)=ce(z)$, and $g_2(z)=-z^{-1}\int\frac{y}{1+cye(z)}dH_2(y).$
Then $\left(m(z),g_1(z),g_2(z)\right)$ also satisfies the equations (\ref{gal1}). Furthermore, we have
\begin{align}\label{cl2}
\begin{cases}
  zg_1(z)=&-c\int\frac x{1+g_2(z)x}dH_1(x)\\
  zg_2(z)=&-\int\frac y{1+g_1(z)y}dH_2(y)
  \end{cases}.
\end{align}

If we let $F^{c,H_{1},H_{2}}$ denote $F$, then $F^{c_n,H_{1n},H_{2n}}$ is obtained from $F^{c,H_{1},H_{2}}$ with ${c,H_{1},H_{2}}$ replaced by $c_n,H_{1n},H_{2n}$ respectively. Define
$$
G_n(x)=p\Big(F^{\bs_n}(x)-F^{c_n,H_{1n},H_{2n}}(x)\Big).
$$
The main result of the present paper, which establishes the CLT of LSS of general separable sample covariance matrices, is stated in the following theorem.

\begin{thm}\label{th1}
Denote by $s_1\ge\cdots\ge s_n$ ($s_1> 0$) the eigenvalues of $\bt_{2n}$. Let $f_1,\cdots,f_{\kappa}$
be functions on $\mathbb{R}$ analytic on an open interval containing
\begin{align}\label{int}
\Bigg[&\liminf_ns_n\left(\lambda_{\min}\({\bS_1}\)I_{(0,1)}(c)\left(1-\sqrt c\right)^2I(s_n\ge0)+\lambda_{\max}\({\bS_1}\)\left(1+\sqrt c\right)^2I(s_n<0)\right),\notag\\
&\limsup_ns_1\left(\lambda_{\max}\({\bS_1}\)\left(1+\sqrt c\right)^2\right)\Bigg].
\end{align}
 Suppose that $\bt_{1n}$ and $\bt_{2n}$ are nonrandom matrices, and their spectral norms are both bounded in n. What is more, we assume that ${\rm Rank}\(\bt_{2n}\)=O(n)$. Let $\alpha_x=|\re x_{jk}^2|^2$, $\kappa_x=\re|x_{jk}|^4-|\re x_{jk}^2|^2-2$.
Then
\begin{itemize}
\item[(i)] If $\bx_n=(x_{jk})$, $\bt_{1n}$, $\bt_{2n}$ are real and $\re{x_{jk}^4}=3,j=1,\cdots,m_1,k=1,\cdots,m_2$, there exists $\delta>0$ such that $\sup_{j,k}\re |x_{jk}|^{6+\delta}\le M<\infty$, then
\begin{align}\label{a3}
  \left(\int f_1(x)dG_n(x),\cdots,\int f_{\kappa}(x)dG_n(x)\right)
\end{align}
converges weakly to a Gaussian vector $\left(X_{f_1},\cdots,X_{f_{\kappa}}\right)$ with mean
\begin{align}\label{ee}
\re X_{f}=&\frac1{2\pi i}\oint_{\cC} \frac{f(z)}{1-cz^{-2}d_3(z)d_4(z)}\left(\frac{\alpha_x}{1-{{\alpha_x}c}{z^{-2}}d_{3}(z)d_{4}(z)}+{\kappa_x}\right)\\
&\times\left[\frac{c d_3(z)d_{4}(z)}{z^3}-\frac{c^2 d_{3}^2(z)d_{4}^2(z)}{z^5}+\frac{c d_{5}(z)}{z^4}+\frac{c^2d_{6}(z)}{z^4}\right]dz\notag
\end{align}
and covariance function
\begin{align}\label{cc}
{\rm Cov}\bigg( X_{f},X_g\bigg)=&-\frac1{4\pi^2}\oint_{\cC_1}\oint_{\cC_2}f(z_1)g(z_2)\frac{\partial^2}{\partial z_2\partial z_1}\Bigg\{\int_0^{d(z_1,z_2)}\frac{1}{1-z}dz\notag\\
&+\int_0^{\alpha_xd(z_1,z_2)}\frac{1}{1-z}dz+\kappa_xd(z_1,z_2)\Bigg\}dz_1dz_2
\end{align}
where $f,g\in\left\{f_1,\cdots,f_{\kappa}\right\}$. Here
\begin{align*}
d_3(z)=\int\frac{x^2}{\left({1+xg_{2}(z)}\right)^2}dH_{1}(x),\quad d_4(z)=\int\frac{y^2}
{\left(g_{1}(z)y+1\right)^2}dH_{2}(y),
\end{align*}
\begin{align*}
  d_5(z)=d_4(z)\int \frac{x^3}{\left(1+g_{2}(z)x\right)^3}dH_{1}(x){\int\frac{y}{(1+g_{1}(z)y)^2}dH_{2}(y)},
\end{align*}
\begin{align*}
  d_6(z)=d_3(z)\int\frac{x}
{\left({1+g_{2}(z)x}\right)^2}dH_{1}(x) \int\frac{y^3}{(1+g_{1}(z)y)^3}dH_{2}(y),
\end{align*}

and
\begin{align*}
d(z_1,z_2)=&\frac1{z_1z_2}\frac{z_1g_1(z_1)-z_2g_1(z_2)}
{g_2(z_1)-g_2(z_2)}\frac{z_1g_2(z_1)-z_2g_2(z_2)}{g_1(z_1)-g_1(z_2)}.
\end{align*} The contours in (\ref{ee}) and (\ref{cc}) (two contours in (\ref{cc}), which we may assume to be nonoverlapping) are closed and are taken in the positive direction in the complex plane, each enclosing the support of $F^{c,H_1,H_2}$.
\item[(ii)] If \ $\bx_n=(x_{jk})=(u_{jk}+i v_{jk}),u_{jk},v_{jk}\in\mathbb{R}$, is a complex matrix with $\re(u_{jk})=\re(v_{jk})=0$, $\re(u_{jk}^2)=\re(v_{jk}^2)=\frac12$, $\re(u_{jk}^4)=\re(v_{jk}^4)=\frac34$,  $u_{jk}$ and $v_{jk}$ are independent, there exists $\delta>0$ such that $\sup_{j,k}\re |x_{jk}|^{6+\delta}\le M<\infty$, then (\ref{a3})-(\ref{cc}) also hold.
\item[(iii)] If $\bt_{1n}$ is real and $\bt_{1n}^*\bt_{1n}, \bt_{2n}$ are diagonal, then (\ref{a3})-(\ref{cc}) still hold.
\end{itemize}
\end{thm}

\begin{remark}
The existence of the $(6+\delta)$th moment in $(i)$ and $(ii)$ and the condition that ${\rm Rank}\(\bt_{2n}\)=O(n)$ are combined to ensure the a.s. bound of the spectral norm of $\bs_n$ as indicated in \cite{yin2018no}. We note that for the cases where $m_1$ and $m_2$ are both of the same order as $n$, then the moment condition could be relaxed and the existence of the fourth moment is enough. When $\kappa_x\neq 0$, then if the conditions in $(iii)$ are not satisfied, the CLT of LSS may not hold. Such
counter examples with $m_1=p$ (or $m_2=n$) can be found in \cite{ZhengB15S}.
\end{remark}

\subsection{Organization of this paper and Contributions} The rest of the present paper is organized as follows. In Section \ref{ap}, an application of our main theorem to high dimensional white noise test is showed. The main theorem is proved in Section \ref{prm}. Lemmas and some technical details are postponed to Appendix.

\section{An application of the main theorem on high dimensional white noise test}\label{ap}
\subsection{Describe of the test problem}
Consider $\{\varepsilon_i\}_{i=1}^n$ be a p-dimensional linear process of the form
$$\bss\varepsilon_i=\sum_{k\geq 0}\mb B_k\mb x_{i-k},$$
where $\mb B_k$ are $p\times m$ coefficient matrices and $\mb x_t$ is a sequence of $m$ dimensional random vectors satisfy that their entries are i.i.d. with 0 means and unit variances.
Define $\mb \Sigma_{(\tau)}={\rm Cov}\(\bss\varepsilon_{i+\tau},\bss\varepsilon_{i}\)$, the so called autocovariance matrix at lag $\tau$ and  $\widehat \Sigma_{(\tau)}=\frac{1}{n}\sum_{i=1}^{n-\tau}\bss\varepsilon_{i+\tau}\bss\varepsilon_{i}^*$, the sample autocovariance matrix at lag $\tau$.
The goal is to test whether $\{\bss\varepsilon_i\}_{i=1}^n$ is a white noise. The hypothesis testing problem is then
$$H_0: {\rm Cov}\(\bss\varepsilon_{i+\tau},\bss\varepsilon_{i}\)=0, \quad \tau=1,\cdots,q$$
where $q$ is a prescribed integer, against the alternate
$$H_{1,\tau}: {\mbox{For given}} 1\leq \tau\leq q \ {\mbox{ we \ have \ }}{\rm Cov}\(\bss\varepsilon_{i+\tau},\bss\varepsilon_{i}\)\neq 0.$$

\subsection{The proposed test procedure}
Define $\Xi_{(\tau)}=\(\Xi_{(\tau),j,k}\)=\frac{1}{2}\(\widehat \Sigma_{(\tau)}+\widehat \Sigma_{(\tau)}^*\)$, $\mb X_n=\(\mb x_1,\cdots,\mb x_n\)$. It is easy to see that under null
\begin{align}
&\Lambda_{(\tau)}\triangleq \sum_{j,k=1}^p\Xi_{(\tau),j,k}^2=\tr\Xi_{(\tau)}\Xi_{(\tau)}^*\\\notag
=&\frac{1}{n^2}\tr\Gamma_n\mb X_n\(\left(
                                                                                           \begin{array}{cc}
                                                                                             \mb 0 & \frac12\mb I_{n-\tau} \\
                                                                                             \frac12\mb I_{\tau} & \mb 0 \\
                                                                                           \end{array}
                                                                                         \right)+\left(
                                                                                           \begin{array}{cc}
                                                                                             \mb 0 & \frac12\mb I_{\tau}\\
                                                                                            \frac12\mb I_{n-\tau} & \mb 0 \\
                                                                                           \end{array}
                                                                                         \right)\)
\mb X_n^*\Gamma_n^*\(\Gamma_n\mb X_n\(\left(
                                                                                           \begin{array}{cc}
                                                                                             \mb 0 & \frac12\mb I_{n-\tau} \\
                                                                                             \frac12\mb I_{\tau} & \mb 0 \\
                                                                                           \end{array}
                                                                                         \right)+\left(
                                                                                           \begin{array}{cc}
                                                                                             \mb 0 & \frac12\mb I_{\tau}\\
                                                                                            \frac12\mb I_{n-\tau} & \mb 0 \\
                                                                                           \end{array}
                                                                                         \right)\)
\mb X_n^*\Gamma_n^*\)^*,
\end{align}
where $\Gamma_n\Gamma_n^*=\mb \Sigma_{(0)}.$
Then our test statistics can be written as

\begin{align}
\widehat \Lambda_{(\tau)} =&\frac{1}{n^2}\tr\Gamma_n\mb X_n\(\left(
                                                                                           \begin{array}{cc}
                                                                                             \mb 0 & \frac12\mb I_{n-\tau} \\
                                                                                             \mb 0 & \mb 0 \\
                                                                                           \end{array}
                                                                                         \right)+\left(
                                                                                           \begin{array}{cc}
                                                                                             \mb 0 & \mb 0\\
                                                                                            \frac12\mb I_{n-\tau} & \mb 0 \\
                                                                                           \end{array}
                                                                                         \right)\)
\mb X_n^*\Gamma_n^*\(\Gamma_n\mb X_n\(\left(
                                                                                           \begin{array}{cc}
                                                                                             \mb 0 & \frac12\mb I_{n-\tau} \\
                                                                                             \mb 0 & \mb 0 \\
                                                                                           \end{array}
                                                                                         \right)+\left(
                                                                                           \begin{array}{cc}
                                                                                             \mb 0 & \mb 0\\
                                                                                            \frac12\mb I_{n-\tau} & \mb 0 \\
                                                                                           \end{array}
                                                                                         \right)\)
\mb X_n^*\Gamma_n^*\)^*\\\notag
&\triangleq \frac{1}{n^2}\tr \mb T_{1n}\mb X_n\mb T_{2n}\mb X_n^*\mb T_{1n}^*\(\mb T_{1n}\mb X_n\mb T_{2n}\mb X_n^*\mb T_{1n}^*\)^*.
\end{align}
We reject the null hypothesis for large $\widehat \Lambda_{(\tau)}.$
\begin{thm}\label{tha}
Assume that the conditions (a)-(e) and those in Theorem \ref{th1} are satisfied. Then under null for given $\tau$,
$$\widehat \Lambda_{(\tau)}-\(\frac{pc}{2}-\frac{\tau c^2}{2}\)\left(\int x dH_{1}(x)\right)^2\stackrel{d}{\rightarrow}{N}(\mu,\sigma^2),$$
where $\mu=\frac{{c }\left({\alpha_x}+{\kappa_x}\right)}{2}\int x^2dH_1(x),$ and $$\sigma^2=\frac{ c^2(1+\alpha_x^2)} {2}\left(\int x^2dH_1(x)\right)^2
+\frac32{c^3}(\kappa_x+2)\left(\int xdH_1(x)\right)^2\int x^2dH_1(x).$$
\end{thm}
The proof of Theorem \ref{tha} is postponed to Section \ref{pth2}, which is also an example on how to determine the mean and variance functions in our main Theorem \ref{th1}.

\subsection{Simulation Study}
In this subsection, we conduct some simulations to investigate the performance our proposed test procedure. Suppose that the entries in $\mb X_n$ follow the gaussian distribution. We consider the following two models
\begin{description}
  \item[Model 1:] $\bss\varepsilon_i=\mb \Sigma_{(0)}^{1/2}\mb x_{i},$ where $\mb \Sigma_{(0)}=\(\sigma_{0,i,j}\)_{p\times p}$ with
\begin{align}
\sigma_{0,i,j}=
\begin{cases}
2+(-1)^i \quad & {\mbox{for \ }} i=j, \\
0 \quad & {\mbox{for \ }} i\neq j.
  \end{cases}
\end{align}
  \item[Model 2:] $\bss\varepsilon_i=\mb \Sigma_{(0)}^{1/2}\(\mb x_{i}+0.3\mb x_{i-1}+0.1\mb x_{i-2}\),$ where $\mb \Sigma_{(0)}$ is defined the same way as Model 1.
\end{description}
Table \ref{table1} and \ref{table2} below show the empirical size and empirical power of our test under Model 1 and 2 respectively with different pairs of $n$ and $p$.
  \begin{table}[!hbp]
  \center
    \begin{tabular}{|c c|cc||cc|cc|cc|}
       \hline
        \hline
        $p$ & $n$ & $q=1$ & $q=3$ &  $p$ & $n$ & $q=1$ & $q=3$ \\
       \hline
       5 &  50 &0.094 & 0.086 & 20 & 40 & 0.078 & 0.092 \\
       \hline
       25 & 250 &0.063 & 0.071 & 50 & 100 & 0.065 & 0.074 \\
       \hline
       50 & 500 & 0.062 & 0.072 & 100 & 200 & 0.054 & 0.061 \\
       \hline
       100 & 1000 &0.063 & 0.058 & 300 & 600 & 0.050 & 0.056 \\
       \hline
       10 & 50 & 0.084 & 0.081 & 50 & 25 & 0.066 & 0.061  \\
       \hline
       50 & 250 & 0.056 & 0.075 & 100 & 50 & 0.054 & 0.053  \\
       \hline
       100 & 500 & 0.060 & 0.057 & 200 & 100 & 0.055 & 0.047  \\
       \hline
       200 & 1000 & 0.055 & 0.056 & 500 & 250 & 0.047 & 0.045  \\
       \hline
     \end{tabular}
  \caption{Empirical Size under model 1}\label{table1}
  \end{table}

    \begin{table}[!hbp]
  \center
    \begin{tabular}{|c c|cc||cc|cc|cc|}
       \hline
        \hline
        $p$ & $n$ & $q=1$ & $q=3$ &  $p$ & $n$ & $q=1$ & $q=3$ \\
       \hline
       5 &  50 &0.857 & 0.741 & 20 & 40 & 0.946 & 0.908 \\
       \hline
       25 & 250 &1.000 & 1.000 & 50 & 100 & 1.000 & 1.000 \\
       \hline
       50 & 500 & 1.000 & 1.000 & 100 & 200 & 1.000 & 1.000 \\
       \hline
       100 & 1000 &1.000 & 1.000 & 300 & 600 & 1.000 & 1.000 \\
       \hline
       10 & 50 & 0.863 & 0.872 & 50 & 25 & 0.968 & 0.946  \\
       \hline
       50 & 250 & 1.000 & 1.000 & 100 & 50 & 1.000 & 1.000  \\
       \hline
       100 & 500 & 1.000 & 1.000 & 200 & 100 & 1.000 & 1.000  \\
       \hline
       200 & 1000 & 1.000 & 1.000 & 500 & 250 & 1.000 & 1.000 \\
       \hline
     \end{tabular}
  \caption{Empirical Power under model 2}\label{table2}
  \end{table}

\section{Proof of the main theorem}\label{prm}
This section is to prove our main theorem. For the reader's convenience, we give here the outline. The proof of Theorem \ref{th1} is split into the following steps.
\begin{description}
\item[Step 1:] In subsection \ref{re2}, we will show that the results of Theorem \ref{th1} are true under gaussian case.
\item[Step 2:] We shall prove Theorem \ref{th1} when the underline distribution is non-gaussian but share the same kurtosis with gaussian distribution. Here the strategy is to compare the characteristic functions of the linear spectral statistics. This part is showed in subsection \ref{ng}.
\item[Step 3:] Then in subsection \ref{sec3} we give the proof of the theorem for the general underline distribution under condition $(iii)$.
\end{description}
We now proceed our proof step by step.

\subsection{The Proof Under Gaussian Case}\label{re2}
Assume now the entries in $\bx_n$ follow normal distribution. By the arguments presented below Definition \ref{def}, it follows that
  $F^{\bs_n}(x)=F^{\widetilde\bs_n}(x).$
Let $\widetilde{\underline\bs}_n=\frac1n\widetilde\bx_{11}^*\Lambda_1^2\widetilde\bx_{11}\Lambda_2$, then we know that $\widetilde{\underline\bs}_n$ and $\widetilde\bs_n$ have the same non-zero eigenvalues. It is obvious that
\begin{align*}
F^{\underline\bs_n}(x)=c_nF^{\bs_n}(x)+(1-c_n)1_{[0,\infty)}(x).
\end{align*}
Denote by $\underline F^{c,H_1,H_2}$ the limiting spectral distribution of $F^{\widetilde{\underline\bs}_n}$, then one finds
\begin{align*}
\underline F^{c,H_1,H_2}(x)=cF(x)+(1-c)1_{[0,\infty)}(x).
\end{align*}
Likewise, $\underline F^{c_n,H_{1n},H_{2n}}$ is obtained from $\underline F^{c,H_{1},H_{2}}$ with ${c,H_{1},H_{2}}$ replaced by $c_n,H_{1n},H_{2n}$ respectively.
Hence, $G_n$ can be rewritten as
\begin{align*}
  n[F^{\widetilde{\underline\bs}_n}-\underline F^{c_n,H_{1n},H_{2n}}].
\end{align*}
Applying Theorem 2.1 in \cite{liclt}, some elementary calculations shows that Theorem \ref{th1} holds under the Gaussian case with mean
\begin{align*}
\re X_{f}=&\frac1{2\pi i}\oint_{\cC} \frac{f(z)}{1-cz^{-2}d_3(z)d_4(z)}\frac{\alpha_x}{1-{{\alpha_x}c}{z^{-2}}d_{3}(z)d_{4}(z)}\\
&\times\left[\frac{c d_3(z)d_{4}(z)}{z^3}-\frac{c^2 d_{3}^2(z)d_{4}^2(z)}{z^5}+\frac{c d_{5}(z)}{z^4}+\frac{c^2d_{6}(z)}{z^4}\right]dz\notag
\end{align*}
and covariance function
\begin{align*}
{\rm Cov}\left( X_{f},X_g\right)=&-\frac1{4\pi^2}\oint_{\cC_1}\oint_{\cC_2}f(z_1)g(z_2)\frac{\partial^2}{\partial z_2\partial z_1}\bigg\{\int_0^{d(z_1,z_2)}\frac{1}{1-z}dz\\
&\qquad+\int_0^{\alpha_xd(z_1,z_2)}\frac{1}{1-z}dz\bigg\}dz_1dz_2.
\end{align*}

\subsection{The proof of Theorem \ref{th1} under $(i)$ and $(ii)$}\label{ng}

From last section, it suffices to show the conclusion holds under the non-Gaussian case. The strategy is to compare the characteristic functions of the linear spectral statistics under the Gaussian case and the non-Gaussian case.

It is worth mentioning that for the complex case, $u_{jk}$, the real part of $x_{jk}$ and $v_{jk}$, the imaginary part of $x_{jk}$ are independent. Thus it is enough to consider the real case only.
The proof of this part are split into two steps. Firstly, we truncate and recentralize the entries in $\mb X_n$. Then complete our proof by handling  the truncated one.
\subsubsection{Step 1: Truncation and Recentralization}

Denoting $\bt_{1n}=(t_{1jk})$, $\bt_{2n}=(t_{2jk})$ and
\begin{align*}
  t_{\cdot k}=\sqrt{\sum_{j=1}^pt_{1jk}^2},\quad t_{ j\cdot}=\sqrt{t_{2jj}},
\end{align*}
then one has $|t_{\cdot k}|<K_1$, $|t_{ j\cdot}|<K_2$ and
\begin{align}\label{al1}
\sum_{k=1}^{m_1}t_{\cdot k}^2\le p\|\boldsymbol\Sigma_1\|_2\le K_3 n,\quad \sum_{j=1}^{m_2}t_{j\cdot }^2\le \rtr\left(\bt_{2n}\right)\le K_4 n.
\end{align}

At first, we can select a arbitrarily slowly decreasing sequence of constants $\eta_n\to0$ such that $\eta_n\sqrt n\to\infty$ and truncate the variables $x_{jk}$ at ${\eta_n\sqrt n}/({t_{\cdot j}t_{k\cdot}}),j=1,\cdots,m_1,k=1,\cdots,m_2$. Define $\hat x_{jk}=x_{jk}I\left(|x_{jk}|\le {\eta_n\sqrt n}/({t_{\cdot j}t_{k\cdot}}) \right),\widehat\bx_n=(\hat x_{jk})$ and
$$ \widehat\bs_n=\frac1n\bt_{1n}\widehat\bx_n\bt_{2n}\widehat\bx_n'\bt_{1n}',$$
then it yields from (\ref{al1})
\begin{align*}
&P(\bs_n\neq \widehat\bs_n,i.o.)
\le\lim_{N\to\infty}\sum_{n=N}^{\infty}{\rm P}\left(\bx_n\neq \widehat\bx_n\right)\\
\le&\lim_{N\to\infty}\sum_{n=N}^{\infty}{\rm P}\left(\bigcup_{j=1}^{m_1}\bigcup_{k=1}^{m_2}\{|x_{jk}|\ge\eta_n\sqrt n/({t_{\cdot j}t_{k\cdot}})\}\right)\\
\le&\lim_{N\to\infty}\sum_{n=N}^{\infty}\sum_{j=1}^{m_1}\sum_{k=1}^{m_2}{t_{\cdot j}^{6+\delta}t_{k\cdot}^{6+\delta}}\frac{\re x_{jk}^{6+\delta}I\left(|x_{jk}|\ge\eta_n\sqrt n/({t_{\cdot j}t_{k\cdot}})\right)}{\eta_n^{6+\delta}n^{3+\delta/2}}\\
\le&K_1^{4+\delta}K_2^{4+\delta}K_3K_4M\lim_{N\to\infty}\sum_{n=N}^{\infty}\frac1{\eta_n^{6+\delta}n^{1+\delta/2}}\to0.
\end{align*}

Next, define $\ddot{\bs}_n=\frac1n\bt_{1n}\left(\widehat\bx_n-\re\widehat\bx_n\right)\bt_{2n}\left(\widehat\bx_n-\re\widehat\bx_n\right)'\bt_{1n}'$. We use $\hat G_n(x)$ and $\ddot G_n(x)$ to denote the analogues of $G_n(x)$ with the matrix $\bs_n$ replaced by $\widehat\bs_n$ and $\ddot\bs_n$, respectively. Let $\lambda_j({\ba})$ denote the $j$-th smallest eigenvalue of Hermitian $\ba$. Using Lemma \ref {lel1}, it implies that
\begin{align*}
  &\re\left|\int f_j(x)d\hat G_n(x)-\int f_j(x)d\ddot G_n(x)\right|\le C_j\sum_{k=1}^p\re\left|\lambda_k({\widehat\bs_n})-\lambda_k({\ddot\bs_n})\right|\\
\le&C_j\left(\re\sum_{k=1}^p\left|\lambda_k^{1/2}({\widehat\bs_n})-\lambda_k^{1/2}({\ddot\bs_n})\right|^2\right)^{1/2}
\left(\re\sum_{k=1}^p\left|\lambda_k^{1/2}({\widehat\bs_n})+\lambda_k^{1/2}({\ddot\bs_n})\right|^2\right)^{1/2}\\
\le&\frac{\sqrt 2C_j}{\sqrt n}\left(\rtr\bt_{1n}\re{\widehat\bx_n}\bt_{2n}\re{\widehat\bx_n}'\bt_{1n}'\right)^{1/2}
\left(\re\sum_{k=1}^p\left(\lambda_k({\widehat\bs_n})+\lambda_k({\ddot\bs_n})\right)\right)^{1/2}\\
=&\frac{\sqrt 2C_j}{\sqrt n}\left(\rtr\bt_{1n}\re{\widehat\bx_n}\bt_{2n}\re{\widehat\bx_n}'\bt_{1n}'\right)^{1/2}
\left(\re\rtr({\widehat\bs_n})+\re\rtr({\ddot\bs_n})\right)^{1/2}
\end{align*}
where $C_j$ is a bound on $f'_j(z)$. Note that
\begin{align}\label{al2}
  \re|\hat x_{jk}|\le\frac{Mt_{\cdot j}^{5+\delta}t_{k\cdot}^{5+\delta}}{\eta_n^{5+\delta}n^{5/2+\delta/2}}\le\frac{MK_1^{4+\delta}K_2^{4+\delta}
  t_{\cdot j}t_{k\cdot}}{\eta_n^{5+\delta}n^{5/2+\delta/2}}.
\end{align}
Using (\ref{al2}) and Cauchy-Schwarz inequality, one finds
\begin{align*}
  &\rtr\left(\bt_{1n}\re{\widehat\bx_n}\bt_{2n}\re{\widehat\bx_n}'\bt_{1n}'\right)=\sum_{l=1}^p\sum_{j_1,j_2=1}^{m_1}\sum_{k_1,k_2=1}^{m_2} t_{1lj_1}\re\hat x_{j_1k_1}t_{2k_1k_2}\re\hat x_{j_2k_2}t_{1lj_2}\\
\le&\sum_{j_1,j_2=1}^{m_1}\sum_{k_1,k_2=1}^{m_2}|t_{2k_1k_2}|\left|\re\hat x_{j_1k_1}\re\hat x_{j_2k_2}\right|\left(\sum_{l=1}^p t_{1lj_1}^2 \sum_{l=1}^p
t_{1lj_2}^2\right)^{1/2}\\
\le&\sum_{j_1,j_2=1}^{m_1}\sum_{k_1,k_2=1}^{m_2}t_{\cdot j_1}t_{\cdot j_2}|t_{2k_1k_2}|\left|\re\hat x_{j_1k_1}\re\hat x_{j_2k_2}\right|\\
\le&M^2 K_1^{8+2\delta}K_2^{6+2\delta}\sum_{j_1,j_2=1}^{m_1}\sum_{k_1,k_2=1}^{m_2}t_{\cdot j_1}t_{\cdot j_2}|t_{2k_1k_2}|
\frac{t_{\cdot j_1}t_{k_1\cdot}^2t_{\cdot j_2}t_{k_2\cdot}^2}{\eta_n^{10+2\delta}n^{5+\delta}}
\le\frac C{\eta_n^{10+2\delta}n^{1+\delta}}=o(n^{-1})
\end{align*}
and
\begin{align*}
&\re\rtr({\widehat\bs_n})+\re\rtr({\ddot\bs_n})=\sum_{l=1}^p\sum_{j_1,j_2=1}^{m_1}\sum_{k_1,k_2=1}^{m_2}\bigg(\re t_{1lj_1}\hat x_{j_1k_1}t_{2k_1k_2}\hat x_{j_2k_2}t_{1lj_2}\\
&+\re t_{1lj_1}(\hat x_{j_1k_1}-\re\hat x_{j_1k_1})t_{2k_1k_2}(\hat x_{j_2k_2}-\re\hat x_{j_2k_2})t_{1lj_2}\bigg)\\
=&\sum_{l=1}^p\sum_{j_1\neq j_2=1}^{m_1}\sum_{k_1,k_2=1}^{m_2} t_{1lj_1}t_{2k_1k_2}t_{1lj_2}\re\hat x_{j_1k_1}\re\hat x_{j_2k_2}+\sum_{l=1}^p\sum_{j=1}^{m_1}\sum_{k_1\neq k_2=1}^{m_2}t_{1lj}^2t_{2k_1k_2}\re \hat x_{jk_1}\re\hat x_{jk_2}\\
&+\sum_{l=1}^p\sum_{j=1}^{m_1}\sum_{k=1}^{m_2} t_{1lj}^2t_{2kk}\re\hat x_{jk}^2+\sum_{l=1}^p\sum_{j=1}^{m_1}\sum_{k=1}^{m_2}
t_{1lj}^2t_{2kk}\re(\hat x_{jk}-\re\hat x_{jk})^2\\
\le&\sum_{l=1}^p\sum_{j_1,j_2=1}^{m_1}\sum_{k_1,k_2=1}^{m_2}{| t_{1lj_1}t_{2k_1k_2}t_{1lj_2}|}|\re\hat x_{j_1k_1}\re\hat x_{j_2k_2}|+2\sum_{j=1}^{m_1}\sum_{k=1}^{m_2} t_{\cdot j}^2t_{k\cdot}^2\\
\le&\sum_{j_1,j_2=1}^{m_1}\sum_{k_1,k_2=1}^{m_2}t_{\cdot j_1}t_{\cdot j_2}|t_{2k_1k_2}||\re\hat x_{j_1k_1}\re\hat x_{j_2k_2}|+2\sum_{j=1}^{m_1}\sum_{k=1}^{m_2} t_{\cdot j}^2t_{k\cdot}^2\\
\le&M^2 K_1^{8+2\delta}K_2^{6+2\delta}\sum_{j_1,j_2=1}^{m_1}\sum_{k_1,k_2=1}^{m_2}|t_{2k_1k_2}|
\frac{t_{\cdot j_1}^2t_{k_1\cdot}^2t_{\cdot j_2}^2t_{k_2\cdot}^2}{\eta_n^{10+2\delta}n^{5+\delta}}+2\sum_{j=1}^{m_1}\sum_{k=1}^{m_2} t_{\cdot j}^2t_{k\cdot}^2
\le C{n^{2}}.
\end{align*}
Hence, we see
\begin{align*}
  \re\left|\int f_j(x)d\hat G_n(x)-\int f_j(x)d\ddot G_n(x)\right|=o(1).
\end{align*}

Thirdly, let $\breve{\bs}_n=\frac1n\bt_{1n}\breve\bx_n\bt_{2n}\breve\bx_n'\bt_{1n}'$ with $\breve\bx_n$ $m_1\times m_2$ having $(j,k)$-th entry $\breve x_{jk}=(\hat x_{jk}-\re \hat x_{jk})/\sigma_{jk}$, where $\sigma_{jk}=\re\left(\hat x_{jk}-\re\hat x_{jk}\right)^2$. Likewise, we still use $\breve G_n(x)$ to denote the analogues of $G_n(x)$ with the matrix $\bs_n$ replaced by $\breve\bs_n$. Due to Lemma \ref {lel1}, one has
\begin{align*}
  &\re\left|\int f_j(x)d\ddot G_n(x)-\int f_j(x)d\breve G_n(x)\right|\le C_j\sum_{k=1}^p\re\left|\lambda_k({\ddot\bs_n})-\lambda_k({\breve\bs_n})\right|\\
\le&C_j\left(\re\sum_{k=1}^p\left|\lambda_k^{1/2}({\ddot\bs_n})-\lambda_k^{1/2}({\breve\bs_n})\right|^2\right)^{1/2}
\left(\re\sum_{k=1}^p\left|\lambda_k^{1/2}({\ddot\bs_n})+\lambda_k^{1/2}({\breve\bs_n})\right|^2\right)^{1/2}\\
\le&\frac{\sqrt 2C_j}{\sqrt n}\left(\re\rtr\bt_{1n}\left({\widehat\bx_n}-\re{\widehat\bx_n}-{\breve\bx_n}\right)\bt_{2n}\left({\widehat\bx_n}-\re{\widehat\bx_n}-{\breve\bx_n}\right)'\bt_{1n}'\right)^{1/2}\\
&\times\left(\re\sum_{k=1}^p\left(\lambda_k({\ddot\bs_n})+\lambda_k({\breve\bs_n})\right)\right)^{1/2}.
\end{align*}
From the fact that
\begin{align*}
 1- \sigma_{jk}\le &1-\sigma_{jk}^2\le2C\re x_{jk}^2I\left(|x_{jk}|\ge\eta_n\sqrt n/(t_{\cdot j}t_{k\cdot})\right)\\
  \le&\frac{Mt_{\cdot j}^{4+\delta}t_{k\cdot}^{4+\delta}}{\eta_n^{4+\delta}n^{2+\delta/2}}\le\frac{MK_1^{4+\delta}K_2^{4+\delta}}{\eta_n^{4+\delta}n^{2+\delta/2}}=o(n^{-2}),\notag
\end{align*}
 we  obtain
\begin{align*}
 &\re\rtr\bt_{1n}\left({\widehat\bx_n}-\re{\widehat\bx_n}-{\breve\bx_n}\right)\bt_{2n}\left({\widehat\bx_n}-\re{\widehat\bx_n}-{\breve\bx_n}\right)'\bt_{1n}'\\
 =&\sum_{l=1}^p\sum_{j_1,j_2=1}^{m_1}\sum_{k_1,k_2=1}^{m_2}t_{1lj_1}t_{2k_1k_2}t_{1lj_2}(1-\sigma_{j_1k_1}^{-1})(1-\sigma_{j_2k_2}^{-1})\\
 &\qquad\qquad\qquad\quad\times\re(\hat x_{j_1k_1}-\re \hat x_{j_1k_1})(\hat x_{j_2k_2}-\re\hat x_{j_2k_2})\\
 =&\sum_{l=1}^p\sum_{j=1}^{m_1}\sum_{k=1}^{m_2}t_{1lj}^2t_{2kk}(1-\sigma_{jk}^{-1})^2\re(\hat x_{jk}-\re \hat x_{jk})^2\sum_{j=1}^{m_1}\sum_{k=1}^{m_2}t_{\cdot j}^2t_{k\cdot}^2(1-\sigma_{jk})^2=o(n^{-2}).
\end{align*}
Moreover,
\begin{align*}
\re\sum_{k=1}^p&\left(\lambda_k({\ddot\bs_n})+\lambda_k({\breve\bs_n})\right)=\re\left(\rtr{\ddot\bs_n}+\rtr{\breve\bs_n}\right)=
\sum_{l=1}^p\sum_{j_1,j_2=1}^{m_1}\sum_{k_1,k_2=1}^{m_2} t_{1lj_1}t_{1lj_2}\\
&\qquad\qquad\times t_{2k_1k_2}(1+\sigma_{j_1k_1}^{-1}\sigma_{j_2k_2}^{-1})\re(\hat x_{j_1k_1}-\re\hat x_{j_1k_1})(\hat x_{j_2k_2}-\re\hat x_{j_2k_2})\\
=&\sum_{l=1}^p\sum_{j=1}^{m_1}\sum_{k=1}^{m_2}t_{1lj}^2t_{2kk}(1+\sigma_{jk}^{-2})\re(\hat x_{jk}-\re\hat x_{jk})^2
=\sum_{j=1}^{m_1}\sum_{k=1}^{m_2} t_{\cdot j}^2t_{k\cdot}^2(1+\sigma_{jk}^{2})\le C{n^{2}}.
\end{align*}
Therefore, it yields that
\begin{align*}
  &\re\left|\int f_j(x)d\ddot G_n(x)-\int f_j(x)d\breve G_n(x)\right|\le\frac{C}{\sqrt n}.
\end{align*}

 We below assume that $x_{jk},j=1,\cdots,m_1,k=1,\cdots,m_2$ are truncated at $\eta_n\sqrt n/(t_{\cdot j}t_{k\cdot})$, centralized and renormalized. That is to say,
\begin{align*}
|x_{jk}|\le\eta_n\sqrt n/(t_{\cdot j}t_{k\cdot}), \ \re x_{jk}=0, \ \re x_{jk}^2=1, \ \re x_{jk}^4=3+o(1).
\end{align*}

\subsubsection{Step 2: Complete the proof of Theorem \ref{th1} under $(i)$ and $(ii)$}\label{see2}

Denote $$\bb_n=\frac1n\bt_{1n}\by_n\bt_{2n}\by_n'\bt_{1n}'$$ where the entries of $\by_n=(y_{jk})$ are independent real Gaussian random variables such that
\begin{align*}
\re y_{jk}=0,\quad \re y_{jk}^2=1,\quad {\rm for} \ j=1\cdots m_1,k=1,\cdots,m_2.
\end{align*}
Moreover, suppose that $\bx_n$ and $\by_n$ be independent random matrices. As in \cite{liclt}, for any $\theta\in[0,\pi/2]$, we introduce the following matrices
\begin{align}\label{eq31}
{\bf W}_n(\theta)=\bx_n\sin\theta+\by_n\cos\theta, \quad{\rm and}\quad{\bf G}_n(\theta)=\frac1n\bt_{1n}{\bf W}_n(\theta)\bt_{2n}{\bf W}_n'(\theta)\bt_{1n}'
\end{align}
where
$$\left({\bf W}_n(\theta)\right)_{jk}=w_{jk}=x_{jk}\sin\theta+a_{jk}\cos\theta.$$
Furthermore, let
\begin{align}\label{eq32}
&{\bf H}_n(t,\theta)=e^{it{\bf G}_n(\theta)},\ S(\theta)=\rtr f({\bf G}_n(\theta)),\\
&S^0(\theta)=S(\theta)-p\int f(x)dF^{c_n,H_{1n},H_{2n}}(x), \ Z_{n}(x,\theta)=\re e^{ixS^0(\theta)}\notag.
\end{align}
For simplicity, we omit the argument $\theta$ from the notations of ${\bf W}_n(\theta),{\bf G}_n(\theta),{\bf H}_n(t,\theta)$ and denote them by ${\bf W}_n,{\bf G}_n,{\bf H}_n(t)$ respectively.

Note that
\begin{align}\label{eq29}
Z_n(x,\pi/2)-Z_n(x,0)=\int_0^{\pi/2}\frac{\partial Z_n(x,\theta)}{\partial \theta}d\theta.
\end{align}
The aim is to prove that $\frac{\partial Z_n(x,\theta)}{\partial \theta}$ converges to zero uniformly in $\theta$ over the interval $[0,\pi/2]$, which ensures Theorem \ref{th1} under condition (i).

To this end, let $f(\lambda)$ be a smooth function with the Fourier transform given by
$\widehat f(t)=\frac1{2\pi}\int_{-\infty}^{\infty}f(\lambda)e^{-it\lambda}d\lambda.$
Then, $f(\lambda)=\int_{-\infty}^{\infty}\widehat f(t)e^{it\lambda}dt.$
From Lemma \ref{le3}, we have
\begin{align*}
\frac{\partial Z_{n}(x,\theta)}{\partial \theta}=\frac {2x i} n\sum_{j=1}^{m_1}\sum_{k=1}^{m_2}\re w_{jk}'\left[\bt_{1n}'\widetilde f({\bf G}_n)\bt_{1n}\bw_n\bt_{2n}\right]_{jk}e^{ixS^0(\theta)}
\end{align*}
where
$w_{jk}'=\frac{d w_{jk}}{d\theta}=x_{jk}\cos\theta-{y}_{jk}\sin\theta$
and
\begin{align}\label{eq30}
\widetilde f({\bf G}_n)=i\int_{-\infty}^{\infty}u\widehat f(u){\bf H}_n(u)du.
\end{align}
Let ${\bf W}_{n jk}(w,\theta)$ denote the corresponding matrix ${\bf W}_{n}$ with $w_{jk}$ replaced by $w$. And let
\begin{align}
&{\bf G}_{n jk}(w,\theta)=\frac1n\bt_{1n}{\bf W}_{n jk}(w,\theta)\bt_{2n}{\bf W}_{n jk}'(w,\theta)\bt_{1n}',\notag\\
& {\bf H}_{njk}(w,t,\theta)=e^{it{\bf G}_{njk}(w,\theta)},\quad S(w,\theta)=\rtr f({\bf G}_{njk}(w,\theta)),\notag
\end{align}
and
\begin{align}
 &S^0(w,\theta)=S(w,\theta)-p\int f(x)dF^{c_n,H_{1n},H_{2n}}(x),\notag\\
 &\widetilde f({\bf G}_{njk}(w,\theta))=i\int_{-\infty}^{\infty}u\widehat f(u){\bf H}_{njk}(w,u,\theta)du,\notag\\
&\varphi_{jk}(w)=\left[\bt_{1n}'\widetilde f({\bf G}_{njk}(w,\theta))\bt_{1n}{\bf W}_{njk}(w,\theta)\bt_{2n}\right]_{jk}e^{ixS^0(w,\theta)}.\notag
\end{align}
For simplicity, we still omit the argument $\theta$ from the notations of ${\bf W}_{njk}(w,\theta)$, ${\bf G}_{njk}(w,\theta)$, ${\bf H}_{njk}(w,t,\theta)$ and denote them by ${\bf W}_{njk w},{\bf G}_{njk w},{\bf H}_{njk w}(t)$ respectively.
By Taylor's formula, one finds
\begin{align*}
\varphi_{jk}(w_{jk})=\sum_{l=0}^3\frac1{l!}w_{jk}^l\varphi_{jk}^{(l)}(0)+\frac1{4!}w_{jk}^4\varphi_{jk}^{(4)}(\varrhoup w_{jk})\quad\varrhoup\in(0,1)
\end{align*}
which implies that
\begin{align*}
\frac{\partial Z_{n}(x,\theta)}{\partial \theta}=\frac {2x i} n\sum_{l=0}^3\frac1{l!}\sum_{j=1}^{m_1}\sum_{k=1}^{m_2}\re w_{jk}'w_{jk}^l\re\varphi_{jk}^{(l)}(0)+\frac {2x i}{4!n} \sum_{j=1}^{m_1}\sum_{k=1}^{m_2}\re w_{jk}' w_{jk}^4\varphi_{jk}^{(4)}(\varrhoup w_{jk}).
\end{align*}
It is easy to obtain
\begin{align*}
&\re w_{jk}'w_{jk}^0=0,\quad\quad\quad\qquad\qquad\re w_{jk}'w_{jk}^1=0,\\
&\re w_{jk}'w_{jk}^2=\re w_{jk}^3\sin^2\theta\cos\theta,\quad\re w_{jk}'w_{jk}^3=o(1)\sin^3\theta\cos\theta.
\end{align*}
It follows that
\begin{align*}
\frac{\partial Z_{n}(x,\theta)}{\partial \theta}=&\frac {x i} {n}\sum_{j=1}^{m_1}\sum_{k=1}^{m_2}\re w_{jk}^3\sin^2\theta\cos\theta\re\varphi_{jk}^{(2)}(0)
+\frac {x i}{12n} \sum_{j=1}^{m_1}\sum_{k=1}^{m_2} \re w_{jk}' w_{jk}^4\varphi_{jk}^{(4)}(\varrhoup w_{jk})\\
\triangleq&\mathcal{I}_1+\mathcal{I}_2.
\end{align*}
We shall prove that both $\mathcal{I}_1$ and $\mathcal{I}_2$ convergence to 0 as $n\to \infty$. The proof is left to Section \ref{appes}.

Thus we complete the proof of our Theorem under $(i)$ and $(ii)$.

\subsection{The proof of Theorem \ref{th1} under $(iii)$}\label{sec3}
This section is to prove our main theorem under $(iii)$. To begin with, we need to derive the LSD of $\bs_n$ under this case.
\subsubsection{ LSD when {$\bt_{1n}$ is real and $\bt_{2n}$ is diagonal}}\label{see1}
We below give a more general results than we need. We note that the result present in this subsection is also new and may have its own interest. Denote $\bt_{2n}={\rm diag}\left(s_1,\cdots,s_{m_2}\right).$
Since the rank of $\bt_{2n}$ is $O(n)$, we know that the nonzero entries of ${\rm diag}\left(s_1,\cdots,s_{m_2}\right)$ is $O(n)$. So, we shall replace $m_2$ with $n$ in the following without lose of generality.

 Introduce $\bx_n=\left(\bbx_1,\cdots,\bbx_n\right), \ \bq_k=\bt_{1n}\bbx_k$,
\begin{align*}
&\bd(z)=\bs_n-z\bi_p, \ \bd_k(z)=\bd(z)-\frac1ns_k\bq_k\bq_k^*, \notag\\
 &\bd_{jk}(z)=\bd_k(z)-\frac1ns_j\bq_j\bq_j^*, \ \bs_{nk}=\bs_n-\frac1ns_k\bq_k\bq_k^*,
\end{align*} and
\begin{align*}
&\varepsilon_k(z)=\bq_k^*\bd_k^{-1}(z)\bq_k-\rtr(\bd_k^{-1}(z)\bS_1), \ \gamma_k(z)={\bq}_k^*\bd_k^{-2}(z){\bq}_k-\rtr(\bd_k^{-2}(z)\bS_1)\notag\\
&\beta_k(z)=\frac1{1+n^{-1}s_k\bq_k^*\bd_k^{-1}(z)\bq_k}, \ \widetilde{\beta}_k(z)=\frac1{1+n^{-1}s_k\rtr(\bd_k^{-1}(z)\bS_1)},\\
&\beta_{jk}(z)=\frac1{1+n^{-1}s_j\bq_j^*\bd_{jk}^{-1}(z)\bq_j}, \ \psi_{jk}(z)=\frac1{{1+n^{-1}s_j\rtr(\bd_k^{-1}(z)\bS_1)}}.
\end{align*}
From the proof of Lemma 0.1 in \cite{liclt} and \cite{yin2018no}, it can be verified that
\begin{align}\label{cal11}
  |\psi_{jk}|\le C.
\end{align}

Let ${\bf R}_{k}(z)=z\bi_p-\frac{1}n\sum_{j\neq k}s_j\psi_{jk}(z)\bS_{1}$,
Write
\begin{align*}
\bd_k(z)+{\bf R}_{k}(z)=\frac1n\sum_{j\neq k}s_j\bq_j\bq_j^*-\frac{1}n\sum_{j\neq k}s_j\psi_{jk}(z_1)\bS_1
\end{align*}
which implies that
\begin{align*}
&{\bf R}_{k}^{-1}(z)+\bd_k^{-1}(z)=\frac1n\sum_{j\neq k}s_j{\bf R}_{k}^{-1}(z)\bq_j\bq_j^*\bd_k^{-1}(z)-\frac{1}n\sum_{j\neq k}s_j\psi_{jk}(z){\bf R}_{k}^{-1}(z)\bS_1\bd_k^{-1}(z).
\end{align*}
Using the formula
\begin{align}\label{al:2}
\left(\boldsymbol\Sigma+q\boldsymbol{\alpha\beta}^*\right)^{-1}\boldsymbol{\alpha}=\frac{\boldsymbol\Sigma^{-1}\boldsymbol{\alpha}}
{1+q\boldsymbol{\beta}^*
\boldsymbol\Sigma^{-1}\boldsymbol\alpha},
\end{align}
we have for a $p\times p$ matrix $\bm$
\begin{align}
&\frac1p\rtr\left(\bm{\bf R}_{k}^{-1}(z)\right)+\frac1p\rtr\left(\bm\bd_k^{-1}(z)\right)\label{beq14}\\
=&\frac1{pn}\sum_{j\neq k}s_j\psi_{jk}(z)\left[\bq_j^*\bd_{jk}^{-1}(z)\bm{\bf R}_{k}^{-1}(z)\bq_j-\rtr\left(\bm{\bf R}_{k}^{-1}(z)\bS_1\bd_{jk}^{-1}(z)\right)\right]\notag\\
&+\frac1{pn}\sum_{j\neq k}s_j\left(\beta_{jk}(z)-\psi_{jk}(z)\right)\bq_j^*\bd_{jk}^{-1}(z)\bm{\bf R}_{k}^{-1}(z)\bq_j\notag\\
&+\frac1{pn}\sum_{j\neq k}s_j\psi_{jk}(z)\rtr\left[\bm{\bf R}_{k}^{-1}(z)\bS_1\left(\bd_{jk}^{-1}(z)-\bd_k^{-1}(z)\right)\right]\notag\\
\triangleq&{r_1}(z)+{r_2}(z)+{r_3}(z)\notag.
\end{align}

By a direct calculation, we have for any positive number $t\ge0$
\begin{align*}
\Im\bigg(z-&\frac{1}n\sum_{j\neq k}s_j\psi_{jk}(z)t\bigg)=v_0-\frac{1}{n^2}\sum_{j\neq k}\frac{s_j^2 t}{|1+n^{-1}s_j\rtr(\bd_k^{-1}(z)\bS_1)|^2}\Im\rtr(\bd_k^{-1}(\bar z)\bS_1)\\
=&v_0\left(1+\frac{1}{n^2}\sum_{j\neq k}\frac{s_j^2 t}{|1+n^{-1}s_j\rtr(\bd_k^{-1}(z)\bS_1)|^2} \rtr\left(\bd_k^{-1}(z)\bd_k^{-1}(\bar z)\bS_1\right)\right)\ge v_0
\end{align*}
which yields
$\left\|{\bf R}_{k}^{-1}(z)\right\|\le\frac1{v_0}.$
 By (\ref{cal11}) and (\ref{mal10}), one gets
\begin{align}\label{bal4}
\re\left|r_1(z)\right|^3
\le &\frac{C}{n^3}\re\left|\bq_j^*\bd_{jk}^{-1}(z)\bm{\bf R}_{k}^{-1}(z)\bq_j
-\rtr\left(
{\bf R}_{k}^{-1}(z)\bS_1\bd_{jk}^{-1}(z)\bm\right)\right|^3\\
\le&\frac{C}{n^3}n^{3/2}\re\left\|\bd_{jk}^{-1}(z)\bm{\bf R}_{k}^{-1}(z)\right\|^3\le\frac C{n^{3/2}}.\notag
\end{align}
Let $\bd_{ljk}=\bd_{jk}-\frac1ns_l\bq_l\bq_l^*$, and $$\widetilde\beta_{jk}(z)=\frac1{1+n^{-1}s_j\rtr(\bd_{jk}^{-1}(z)\bS_1)}, \ \beta_{ljk}(z)=\frac1{1+n^{-1}s_l\bq_l^*\bd_{ljk}^{-1}(z)\bS_1\bq_l}.$$
Using the above inequality and (\ref{mal10}), we obtain
\begin{align}\label{bal8}
\re|\beta_{jk}(z)-{\psi_{jk}(z)}|^3
\le&C\left[\re|\beta_{jk}(z)-\widetilde\beta_{jk}(z)|^3+|\widetilde\beta_{jk}(z)-\psi_{jk}(z)|^3\right]\\
\le&\frac C{n^3}\re|\bq_j^*\bd_{jk}^{-1}(z)\bq_j-\rtr(\bd_{jk}^{-1}(z)\bS_1)|^3\notag\\
&+\frac C{n^3}\re|\rtr(\bd_{jk}^{-1}(z)\bS_1)-\rtr(\bd_k^{-1}(z)\bS_1)|^3\notag\\
\le&\frac C{n^{3/2}}+\frac C{n^3}=O(n^{-3/2})\notag
\end{align}
which implies that
\begin{align}\label{bal5}
\re\left|r_2(z)\right|^{2}
\le&\frac C{n^4}\sum_{j\neq k}\re^{2/3}\left|\beta_{jk}(z)-{\psi_{jk}(z)}\right|^3\re^{1/3}\left|\bq_j^*\bd_{jk}^{-1}(z)\bm{\bf R}_{k}^{-1}(z)\bq_j\right|^{6}\\
\le&\frac C{n^3}\left[O(n^{-3/2})\right]^{2/3}\left[O(n^6)\right]^{1/3}=O(n^{-2})\notag.
\end{align}
Note that from (\ref{cal11})
\begin{align*}
|r_3(z)|=&\left|\frac1{pn^2}\sum_{j\neq k}s_j^2\psi_{jk}(z)\beta_{jk}(z)\bq_j^*\bd_{jk}^{-1}(z)\bm{\bf R}_{k}^{-1}(z)\bS_1\bd_{jk}^{-1}(z)\bq_j\right|\\
\le&\frac C{n^3}\sum_{j\neq k}\left|\beta_{jk}(z)-\psi_{jk}(z)\right|\left|\bq_j^*\bq_j\right|+\frac C{n^3}\sum_{j\neq k}\left|\bq_j^*\bq_j\right|.
\end{align*}
By (\ref{mal10}), it follows that
\begin{align*}
  {\rm P}\left(\left|\bq_j^*\bq_j-\rtr\left(\bS_1\right)\right|\ge n\right)\le \frac1{n^3}\re\left|\bq_j^*\bq_j-\rtr\left(\bS_1\right)\right|^3\le\frac C{n^{3/2}}
\end{align*}
which is summable. Hence, by (\ref{bal8}), one has
\begin{align}\label{bal6}
\left|{r_3}(z)\right|\le&\frac C{n^3}\sum_{j\neq k}\left|\bq_j^*\bq_j-\rtr\left(\bS_1\right)\right|+\frac Cn
\le \frac Cn\quad{\rm a.s}.
\end{align}
 Using (\ref{bal4}), (\ref{bal5}), and (\ref{bal6}), (\ref{beq14}) can be represented as
\begin{align}\label{beq15}
&\frac1p\rtr\left(\bm\bd_k^{-1}(z)\right)=-\frac1p\rtr\left(\bm{\bf R}_{k}^{-1}(z)\right)+o_{{\rm a.s.}}(1).
\end{align}
Especially,
\begin{align*}
\frac1p\rtr\left(\bd_k^{-1}(z)\right)=&-\frac1p\rtr\left({\bf R}_{k}^{-1}(z)\right)+o_{{\rm a.s.}}(1)\\
\frac1p\rtr\left(\bS_1\bd_k^{-1}(z)\right)=&-\frac1p\rtr\left(\bS_1{\bf R}_{k}^{-1}(z)\right)+o_{{\rm a.s.}}(1).
\end{align*}
From the above equality and subsection 2.2 in \cite{paul2009no}, we see that
 with probability 1, $F^{\bs_n}$ converges weakly to a probability distribution function $F$ whose Stieltjes transform $m(z)$, for $z\in\mathbb{C}^+$, is given by
\begin{align*}
  m(z)=\int\frac1{x\int\frac{y}{1+cye}dH_2(y)-z}dH_1(x)
\end{align*}
where $e=e(z)$ is the unique solution in $\mathbb{C}^+$ of the equation
\begin{align*}
  e=\int\frac{x}{x\int\frac{y}{1+cye}dH_2(y)-z}dH_1(x).
\end{align*}
Combining the above arguments, it also follows that
\begin{align}\label{cl3}
\frac1n\rtr\left(\bS_1\bd_k^{-1}(z)\right)\to ce(z)=g_1(z).
\end{align}

\subsubsection{Complete the proof of Theorem \ref{th1} under $(iii)$}\label{proofiii}
Under the conditions in $(iii)$,
denote
\begin{align*}
  \bt_{1n}^*\bt_{1n}={\rm diag}\left(\lambda_1,\cdots,\lambda_{m_1}\right).
\end{align*}
Since the rank of $\bt_{1n}^*\bt_{1n}$ is the same as $\bt_{1n}\bt_{1n}^*$ , which is not larger than $p$. Hence, we know that the nonzero entries of ${\rm diag}\left(\lambda_1,\cdots,\lambda_{m_1}\right)$ is not larger than $p$. So, we shall replace respectively $m_1,m_2$ with $p,n$ in the following. In this case, we only need to truncate $x_{jk}$ at $\eta_n\sqrt n$ and then recentralize. This step is same with the truncation step in \cite{liclt} thus omitted.

Rewrite for $z\in\mathcal{C}_n$
\begin{align*}
M_n(z)=p[m_n(z)-\re m_n(z)]+p[\re m_n(z)-m_n^0(z)]\triangleq M_{n1}(z)+M_{n2}(z)
\end{align*}
where $m_n^0(z)=m_{F^{c_n,H_{1n},H_{2n}}}(z)$. Moreover $g_{1n}^0(z)$ and $g_{2n}^0(z)$ are similarly obtained from $g_1(z)$ and $g_2(z)$ respectively. Then $\left(m_n^0(z),g_{1n}^0(z),g_{2n}^0(z)\right)$ satisfies the equations (\ref{gal1}). In other words
\begin{align}
\underline m_n^0(z)=&-z^{-1}\int\frac1{1+g_{1n}^0(z)y}dH_{2n}(y)\label{i1}\\
m_n^0(z)=&-z^{-1}\int\frac1{1+g_{2n}^0(z)x}dH_{1n}(x)\label{i2}\\
m_n^0(z)=&-z^{-1}-c_n^{-1}g_{1n}^0(z)g_{2n}^0(z).\label{i3}
\end{align}
Furthermore,
\begin{align}
&zg_{1n}^0(z)=-c_n\int\frac x{1+g_{2n}^0(z)x}dH_{1n}(x)\label{i4}\\
&zg_{2n}^0(z)=-\int\frac y{1+g_{1n}^0(z)y}dH_{2n}(y)\label{i5}.
\end{align}

By the assumption of Theorem \ref{th1}, we may suppose $\max\left\{\left\|\bS_1\right\|,\left\|\bt_{2n}\right\|\right\}\le\tau$.
Let $v_0$ be any positive number. Let $x_r$ be any positive number if the right end point of interval (\ref{int}) is zero.  Otherwise choose
\begin{align*}
x_r\in
(\limsup_ns_1\lambda_{\max}^{\bS_1}\left(1+\sqrt c\right)^2,\infty).
\end{align*}
Let $x_l$ be any negative number if the left end point of interval (\ref{int}) is zero. Otherwise choose
\begin{align*}
x_l\in
\begin{cases}
(0,\liminf_ns_n\lambda_{\min}^{\bS_1}I_{(0,1)}(c)\left(1-\sqrt c\right)^2),&{\rm if} \ \liminf_ns_n\lambda_{\min}^{\bS_1}I_{(0,1)}(c)>0,\\
(-\infty,\liminf_ns_n\lambda_{\max}^{\bS_1}\left(1+\sqrt c\right)^2),&{\rm if} \ \liminf_ns_n\lambda_{\min}^{\bS_1}I_{(0,1)}(c)\le 0.
\end{cases}
\end{align*}
Let $\mathcal{C}_u=\left\{x+iv_0:x\in[x_l,x_r]\right\}.$
Define the contour $\mathcal{C}$
\begin{align*}
\mathcal{C}=\left\{x_l+iv:v\in[0,v_0]\right\}\cup\mathcal{C}_u\cup\left\{x_r+iv:v\in[0,v_0]\right\}.
\end{align*}

To avoid dealing with the small $\Im z$, we truncate $M_n(z)$ on a contour $\mathcal{C}$ of the complex plane. We define now the subsets $\mathcal{C}_n$ of $\mathcal{C}$ on which $M_n(\cdot)$ agrees with $\widehat M_n(\cdot)$. Choose sequence $\{\varepsilon_n\}$ decreasing to zero satisfying for some $\alpha\in(0,1),$
$\varepsilon_n\ge n^{-\alpha}.$

Let $\mathcal{C}_l=\left\{x_l+iv:v\in[n^{-1}\varepsilon_n,v_0]\right\}$ and $
\mathcal{C}_r=\left\{x_r+iv:v\in[n^{-1}\varepsilon_n,v_0]\right\}.$

Then $\mathcal{C}_n=\mathcal{C}_l\cup\mathcal{C}_u\cup\mathcal{C}_r$. For $z=x+iv$, the process $\widehat M_n(\cdot)$ can now be defined as
\begin{align}\label{aa}
\widehat M_n(\cdot)=
\begin{cases}
M_n(z),&{\rm for} \ z\in\mathcal{C}_n,\\
M_n(x_l+in^{-1}\varepsilon_n),& {\rm for} \ x=x_l,v\in[0,n^{-1}\varepsilon_n],\\
M_n(x_r+in^{-1}\varepsilon_n),& {\rm for} \ x=x_r,v\in[0,n^{-1}\varepsilon_n].
\end{cases}
\end{align}

The central limit theorem of $\widehat M_n(z)$ is specified below.
\begin{lemma}\label{th2}
Under the condition (iii) of Theorem \ref{th1}, $\widehat M_n(z)$ converges weakly to a two-dimensional Gaussian process $M(\cdot)$ satisfying for $z\in\mathcal{C}$ under the assumptions in $(i)$
\begin{align*}
\re M(z)=&-\frac1{1-cz^{-2}d_3(z)d_4(z)}\left(\frac{\alpha_x}{1-{{\alpha_x}c}{z^{-2}}d_{3}(z)d_{4}(z)}+{\kappa_x}\right)\\
&\times\left[\frac{c d_3(z)d_{4}(z)}{z^3}-\frac{c^2 d_{3}^2(z)d_{4}^2(z)}{z^5}+\frac{c d_{5}(z)}{z^4}+\frac{c^2d_{6}(z)}{z^4}\right]\notag
\end{align*}
and for $z_1,z_2\in \mathcal{C}\cup\overline{\mathcal{C}}$ with $\overline{\mathcal{C}}=\{\bar z:z\in\mathcal{C}\}$,
\begin{align*}
{\rm Cov}\bigg( M(z_1),&M(z_2)\bigg)=\frac{\partial^2}{\partial z_2\partial z_1}\Bigg\{\int_0^{d(z_1,z_2)+{\alpha_xd(z_1,z_2)}}\frac{1}{1-z}dz\notag\\
&+\frac{c\kappa_x(z_1g_2(z_1)-z_2g_2(z_2))}{z_1z_2(g_1(z_1)-g_1(z_2))} \int\frac{x^2}{(1+g_2(z_1)x)(1+g_2(z_2)x)}dH_1(x)\Bigg\}.
\end{align*}
\end{lemma}

From \cite{BaiYin1993} and \cite{Yin1988}, we conclude that
\begin{align*}
 \lambda_{\max}\left(\frac1n\bx_n\bx_n^*\right)\to \left(1+\sqrt c\right)^2\quad{\rm a.s.}
\end{align*}
and
\begin{align*}
 \lambda_{\min}\left(\frac1n\bx_n\bx_n^*\right)\to \left(1-\sqrt c\right)^2\quad{\rm a.s.}
\end{align*}
The upper and lower bounds of the extreme eigenvalues of $\bs_n$ depends largely on the signs of $s_1$ and $s_n$.  Since $s_1>0$, we have
\begin{align*}
\lambda_{\max}(\bs_n)\le s_1\lambda_{\max}^{\bS_1}  \lambda_{\max}\left(\frac1n\bx_n\bx_n^*\right)\le s_1\lambda_{\max}^{\bS_1}  \left(1+\sqrt c\right)^2\quad{\rm a.s.}
\end{align*}
If $s_n>0$, then we have
\begin{align*}
\lambda_{\min}(\bs_n)\ge s_n\lambda_{\min}^{\bS_1}  I_{(0,1)}(c)\lambda_{\min}\left(\frac1n\bx_n\bx_n^*\right)\ge s_n\lambda_{\min}^{\bS_1} I_{(0,1)}(c)\left(1-\sqrt c\right)^2\quad{\rm a.s.}
\end{align*}
Otherwise, we get
\begin{align*}
\lambda_{\min}(\bs_n)\ge s_n \lambda_{\max}^{\bS_1}  \lambda_{\max}\left(\frac1n\bx_n\bx_n^*\right)\ge s_n \lambda_{\max}^{\bS_1}  \left(1+\sqrt c\right)^2\quad{\rm a.s.}
\end{align*}
Combining the definitions of $x_l,x_r$, we find with probability $1$
\begin{align*}
\liminf_{n\to\infty}\min\left(x_r-\lambda_{\max}(\bs_n),\lambda_{\min}(\bs_n)-x_l\right)>0.
\end{align*}
Since $F^{\bs_n}\to F^{c,H_1.H_2}$ with probability $1$, the support of $F^{c_n,H_{1n},H_{2n}}$ is contained in interval (\ref{int}) with probability 1. Thus, by Cauchy integral formula, for $f\in\{f_1,\cdots,f_{\kappa}\}$ and large $n$, with probability $1$,
\begin{align*}
\int f(x)dG_n(x)=-\frac1{2\pi i}\oint f(z)M_{n}(z)dz
\end{align*}
where the complex integral is over ${\mathcal{C}\cup\overline{\mathcal{C}}}$. For $v\in[0,n^{-1}\varepsilon_n]$, note that
\begin{align*}
\left|M_n(x_r+iv)-M_n(x_r+in^{-1}\varepsilon_n)\right|
\le 4n\left|\max\left(\lambda_{\max}(\bs_n), p_r\right)-x_r\right|^{-1}
\end{align*}
and
\begin{align*}
\left|M_n(x_l+iv)-M_n(x_l+in^{-1}\varepsilon_n)\right|
\le 4n\left|\min\left(\lambda_{\min}(\bs_n), p_l\right)-x_l\right|^{-1}.
\end{align*}
It follows that for large $n$, with probability $1$,
\begin{align*}
&\left|\oint f(z)\left(M_{n}(z)-\widehat M_n(z)\right)dz\right|\\
\le &8K\varepsilon_n\left[\left|\max\left(\lambda_{\max}(\bs_n), p_r\right)-x_r\right|^{-1}+\left|\min\left(\lambda_{\min}(\bs_n), p_l\right)-x_l\right|^{-1}\right]\to 0
\end{align*}
where $p_l \ (p_r)$ is the left endpoint (right endpoint) of interval (\ref{int}) and $K$ is the bound on $f$ over $\mathcal{C}$.

Note that the mapping
\begin{align*}
\widehat M_n(\cdot)\rightarrow \left(-\frac1{2\pi i}\oint f_1(z)\widehat M_n(z)dz,\cdots,-\frac1{2\pi i}\oint f_{\kappa}(z)\widehat M_n(z)dz\right)
\end{align*}
is continuous. Using Lemma \ref{th2}, the proof of Theorem \ref{th1} is completed.

\subsection{The proof of Lemma \ref{th2}}\label{pth2}
According to Lemma 3.1 of \cite{Bai2015}, we know that the ESD of $\mb T_{2n}$ tends to $H_2$, which is an Arcsine distribution with density function $$H_2'(t)=\frac{1}{\pi\sqrt{1-t^2}},\quad t\in(-1,1).$$ And for given $n$, the certain $n-\tau+1$ non zero eigenvalues of  $\mb T_{2n}$ are $\lambda_k=\cos\frac{k\pi}{n-\tau+2}, k=1,\cdots,n-\tau+1.$

Applying Theorem \ref{th1} and mapping into our case, we have $f(x)=x^2$ and
\begin{align*}
\sigma^2={\rm Cov}\bigg( X_{f},X_f\bigg)=&-\frac1{4\pi^2}\oint_{\cC_1}\oint_{\cC_2}z_1^2z_2^2\frac{\partial}{\partial z_2}\Bigg\{\frac{1}{1-d(z_1,z_2)}\frac{\partial d(z_1,z_2)}{\partial z_1}\notag\\
&+\frac{\alpha_x}{1-\alpha_xd(z_1,z_2)}\frac{\partial d(z_1,z_2)}{\partial z_1}+\kappa_x\frac{\partial d(z_1,z_2)}{\partial z_1}\Bigg\}dz_1dz_2\\
=&-\frac1{4\pi^2}\oint_{\cC_1}\oint_{\cC_2}z_1^2z_2^2\frac{1}{(1-d(z_1,z_2))^2}\frac{\partial d(z_1,z_2)}{\partial z_1}\frac{\partial d(z_1,z_2)}{\partial z_2}dz_1dz_2\\
&-\frac1{4\pi^2}\oint_{\cC_1}\oint_{\cC_2}z_1^2z_2^2\frac{1}{1-d(z_1,z_2)}\frac{\partial ^2d(z_1,z_2)}{\partial z_1\partial z_2}dz_1dz_2\\
&-\frac{\alpha_x^2}{4\pi^2}\oint_{\cC_1}\oint_{\cC_2}z_1^2z_2^2\frac1{(1-\alpha_xd(z_1,z_2))^2}\frac{\partial d(z_1,z_2)}{\partial z_1}\frac{\partial d(z_1,z_2)}{\partial z_2}dz_1dz_2\\
&-\frac{\alpha_x}{4\pi^2}\oint_{\cC_1}\oint_{\cC_2}z_1^2z_2^2\frac 1{1-\alpha_xd(z_1,z_2)}\frac{\partial ^2d(z_1,z_2)}{\partial z_1\partial z_2}dz_1dz_2\\
&-\frac{\kappa_x}{4\pi^2}\oint_{\cC_1}\oint_{\cC_2}z_1^2z_2^2\frac{\partial^2d(z_1,z_2)}{\partial z_1\partial z_2}dz_1dz_2\\
\triangleq& \mathcal{K}_1(1)+\mathcal{K}_2(1)+\alpha_x^2\mathcal{K}_1(\alpha_x)+\alpha_x\mathcal{K}_2(\alpha_x)+{\kappa_x}\mathcal{K}_3.
\end{align*}
where
\begin{align*}
 \mathcal{K}_1(\alpha)=&-\frac1{4\pi^2}\oint_{\cC_1}\oint_{\cC_2}z_1^2z_2^2\frac{1}{(1-\alpha d(z_1,z_2))^2}\frac{\partial d(z_1,z_2)}{\partial z_1}\frac{\partial d(z_1,z_2)}{\partial z_2}dz_1dz_2\\
 \mathcal{K}_2(\alpha)=&-\frac1{4\pi^2}\oint_{\cC_1}\oint_{\cC_2}z_1^2z_2^2\frac 1{1-\alpha d(z_1,z_2)}\frac{\partial ^2d(z_1,z_2)}{\partial z_1\partial z_2}dz_1dz_2.
\end{align*}

To begin with, we present more relations about $g_1(z)$ and $g_2(z)$. From (\ref{cl2}), we deduce that
\begin{align}\label{aa1}
  z=-\frac{\int\frac{y}{1+g_1(z)y}dH_2(y)}{g_2(z)}=-\frac{c\int\frac{x}{1+g_2(z)x}dH_1(x)}{g_1(z)}
\end{align}
which implies that
\begin{align*}
1-{\int\frac1{1+g_1(z)y}dH_2(y)}=c-{c\int\frac1{1+g_2(z)x}dH_1(x)}.
\end{align*}
It follows that from the above equality
\begin{align}\label{aa2}
  g_1(z)=0 \Leftrightarrow g_2(z)=0.
\end{align}
It can be verified that from (\ref{cl2})
\begin{align}\label{aa5}
  \frac{dg_1(z)}{dz}=&\frac{c\int\frac x{(1+xg_2(z))^2}dH_1(x)}{z^2-cd_3(z)d_4(z)},\quad{\rm and}\quad
  \frac{dg_2(z)}{dz}=\frac{\int\frac y{(1+yg_1(z))^2}dH_2(y)}{z^2-cd_3(z)d_4(z)}.
\end{align}
This yields
\begin{align}\label{aa4}
  \frac{\partial g_2(z)}{\partial g_1(z)}=\frac{\int\frac y{(1+yg_1(z))^2}dH_2(y)}{c\int\frac x{(1+xg_2(z))^2}dH_1(x)}.
\end{align}

Let $h_1(z)={y}/({1+g_1(z)y}),\ h_2(z)={x}/({1+g_2(z)x}),$
$$h_3(z)={\int\frac y{(1+yg_1(z))^2}dH_2(y)}/{\int\frac x{(1+xg_2(z))^2}dH_1(x)},$$
$$ h_{11}=h_1(z_1), \ h_{12}=h_1(z_2), \ h_{21}=h_2(z_1),\ h_{22}=h_2(z_2),$$
$$ g_{11}=g_1(z_1), \ g_{12}=g_1(z_2), \ g_{21}=g_2(z_1),\ g_{22}=g_2(z_2).$$
Rewrite
\begin{align}\label{aa3}
d(z_1,z_2)=&\frac c{z_1z_2}\int h_{21}h_{22}dH_1(x)\int h_{11}h_{12}dH_2(y)
\end{align}
Thus, we get from (\ref{aa1}) and (\ref{aa4})
\begin{align*}
\frac{\partial d(z_1,z_2)}{\partial z_1}=&-\frac{\partial }{\partial z_1}\frac {g_{11}}{z_2\int h_{21}dH_1(x)}\int h_{21}h_{22}dH_1(x)\int h_{11}h_{12}dH_2(y)\\
  =&\frac c{z_1z_2g_{11}}\int h_{21}h_{22}dH_1(x)\int\frac{ h_{11}^2h_{12}}ydH_2(y)\frac{\partial g_{11} }{\partial z_1}\\
    &-\frac {ch_3(z_1)}{z_1^2z_2g_{11}}\int h_{21}^2dH_1(x)\int h_{21}h_{22}dH_1(x)\int h_{11}h_{12}dH_2(y)\frac{\partial g_{11} }{\partial z_1}\\
  &-\frac {h_3(z_1)}{z_1z_2}\int h_{21}^2h_{22}dH_1(x)\int h_{11}h_{12}dH_2(y)\frac{\partial g_{11} }{\partial z_1},
\end{align*}
\begin{align*}
\frac{\partial d(z_1,z_2)}{\partial z_2}=&\frac c{z_1z_2 g_{12}}\int h_{21}h_{22}dH_1(x)\int \frac{h_{11}h_{12}^2}ydH_2(y)\frac{\partial g_{12}}{\partial z_2}\\
   &-\frac {ch_3(z_2)}{z_1z_2^2g_{12}}\int h_{22}^2dH_1(x)\int h_{21}h_{22}dH_1(x)\int h_{11}h_{12}dH_2(y)\frac{\partial g_{12} }{\partial z_1}\\
  &-\frac {h_3(z_2)}{z_1z_2}\int h_{21}h_{22}^2dH_1(x)\int h_{11}h_{12}dH_2(y)\frac{\partial g_{12} }{\partial z_1},
\end{align*}
and
\begin{align*}
&\frac{\partial^2 d(z_1,z_2)}{\partial z_1\partial z_2}\\
=&\frac c{z_1z_2 g_{11}g_{12}}\int h_{21}h_{22}dH_1(x)\int \frac{h_{11}^2h_{12}^2}{y^2}dH_2(y)\frac{\partial g_{11}}{\partial z_1}\frac{\partial g_{12}}{\partial z_2}\\
&\ -\frac {c h_3(z_2)}{z_1 z_2^2 g_{11}g_{12}}\int h_{22}^2dH_1(x)\int h_{21}h_{22}dH_1(x)\int \frac{h_{11}^2h_{12}}{y}dH_2(y)\frac{\partial g_{11}}{\partial z_1}\frac{\partial g_{12}}{\partial z_2}\\
&\ -\frac { h_3(z_2)}{z_1 z_2g_{11}}\int h_{21}h_{22}^2dH_1(x)\int \frac{h_{11}^2h_{12}}{y}dH_2(y)\frac{\partial g_{11}}{\partial z_1}\frac{\partial g_{12}}{\partial z_2}\\
&\ -\frac {c h_3(z_1)}{z_1^2z_2 g_{11}g_{12}}\int h_{21}^2dH_1(x)\int h_{21}h_{22}dH_1(x)\int \frac{h_{11}h_{12}^2}{y}dH_2(y)\frac{\partial g_{11}}{\partial z_1}\frac{\partial g_{12}}{\partial z_2}\\
&\ +\frac {c h_3(z_1)
h_3(z_2)}{z_1^2 z_2^2g_{11}g_{12}}\int h_{21}^2dH_1(x)\int h_{22}^2dH_1(x)\int h_{21}h_{22}dH_1(x)\int h_{11}h_{12}dH_2(y)\frac{\partial g_{11}}{\partial z_1}\frac{\partial g_{12}}{\partial z_2}\\
&\ +\frac {h_3(z_1)
h_3(z_2)}{z_1^2 z_2 g_{11}}\int h_{21}^2dH_1(x)\int h_{21}h_{22}^2dH_1(x)\int h_{11}h_{12}dH_2(y)\frac{\partial g_{11}}{\partial z_1}\frac{\partial g_{12}}{\partial z_2}\\
&\ -\frac {h_3(z_1)}{z_1z_2g_{12}}\int h_{21}^2h_{22}dH_1(x)\int \frac{h_{11}h_{12}^2}{y}dH_2(y)\frac{\partial g_{11}}{\partial z_1}\frac{\partial g_{12}}{\partial z_2}\\
&\ +\frac {h_3(z_1)
h_3(z_2)}{z_1z_2^2g_{12}}\int h_{22}^2dH_1(x)\int h_{21}^2h_{22}dH_1(x)\int h_{11}h_{12}dH_2(y)\frac{\partial g_{11}}{\partial z_1}\frac{\partial g_{12}}{\partial z_2}\\
&\ +\frac { h_3(z_1)
h_3(z_2)}{cz_1z_2}\int h_{21}^2h_{22}^2dH_1(x)\int h_{11}h_{12}dH_2(y)\frac{\partial g_{11}}{\partial z_1}\frac{\partial g_{12}}{\partial z_2}.
\end{align*}

For $\mathcal{K}_1(\alpha)$, we obtain
\begin{align*}
&\mathcal{K}_1(\alpha)\\
=&-\frac{c^2}{4\pi^2}\oint_{\cC_1}\oint_{\cC_2}\frac{1}{g_{11}g_{12}(1-\alpha d(z_1,z_2))^2}
\left(\int h_{21}h_{22}dH_1(x)\right)^2\\
&\qquad\times\int \frac{h_{11}^2h_{12}}ydH_2(y)\int \frac{h_{11}h_{12}^2}ydH_2(y)d g_{11}d g_{12}\\
=&\frac{c^2}{2\pi i}\oint_{\cC_1}\frac{1}{g_{11}}
\left(\int xh_{21}dH_1(x)\right)^2\int {h_{11}^2}dH_2(y)\int {yh_{11}}dH_2(y)d g_{11}d g_{12}\\
=&{c^2}\left(\int x^2dH_1(x)\right)^2\left(\int {y^2}dH_2(y)\right)^2=\frac{c^2}4\left(\int x^2dH_1(x)\right)^2.
\end{align*}
Here
\begin{align*}
\int {y^2}dH_2(y)&=\int_{-1}^1{y^2}\frac1{\pi\sqrt{1-y^2}}dy\\
&\xlongequal{y=\cos\theta}\frac1{\pi}\int_{0}^{\pi}{\cos^2\theta}d\theta=\frac1{2\pi}\int_{0}^{2\pi}{\cos^2\theta}
d\theta\\
&\xlongequal{s=e^{i\theta}}\frac1{2\pi i}\oint_{|s|=1}\frac{\left(s^2+1\right)^2}
{4s^3}ds=\frac12.
\end{align*}
By (\ref{aa3}), it follows that
\begin{align*}
  \left.\frac{\partial d(z_1,z_2)}{\partial g_{11}}\right|_{g_{11}=0}=-\frac{1}{z_2\int xdH_1(x)}\int xh_{22}dH_1(x)\int yh_{12}dH_2(y).
\end{align*}
and
\begin{align*}
  \left.\frac{\partial d(z_1,z_2)}{\partial g_{12}}\right|_{g_{12}=0}=-\frac{1}{z_1\int xdH_1(x)}\int xh_{21}dH_1(x)\int yh_{11}dH_2(y)
\end{align*}
Using (\ref{aa1}) and (\ref{aa2}), one finds that
\begin{align*}
\mathcal{K}_2(\alpha)=&-\frac c{4\pi^2}\oint_{\cC_1}\oint_{\cC_2}\frac {z_1z_2 g_{11}^{-1}g_{12}^{-1}}{1-\alpha d(z_1,z_2)}\int h_{21}h_{22}dH_1(x)\int \frac{h_{11}^2h_{12}^2}{y^2}dH_2(y){dg_{11}}{dg_{12}}\\
=&\frac {c^2}{4\pi^2}\oint_{\cC_1}\oint_{\cC_2}\frac {z_1 g_{11}^{-1}g_{12}^{-2}}{1-\alpha d(z_1,z_2)}\int h_{22}dH_1(x)\int h_{21}h_{22}dH_1(x)\\
&\qquad\times\int \frac{h_{11}^2h_{12}^2}{y^2}dH_2(y)d g_{11}d g_{12}\\
=& -\frac{c^2}{2\pi i}\oint_{\cC_1}\bigg[\frac {z_1 g_{11}^{-1}}{1-\alpha d(z_1,z_2)}\int h_{22}dH_1(x)\int h_{21}h_{22}dH_1(x)\\
&\qquad\times\left.\int \frac{h_{11}^2h_{12}^2}{y^2}dH_2(y)\bigg]'\right|_{g_{12}=0}d g_{11}\\
=& -\frac{c^2}{2\pi i}\oint_{\cC_1}\frac {\alpha z_1}{g_{11}}\left.\frac{\partial d(z_1,z_2)}{\partial g_{12}}\right|_{g_{12}=0}\int xdH_1(x)\int xh_{21}dH_1(x)\int {h_{11}^2}dH_2(y)d g_{11}\\
&+\frac{c^2}{\pi i}\oint_{\cC_1}\frac {z_1}{g_{11}}\int xdH_1(x)\int xh_{21}dH_1(x)\int {yh_{11}^2}dH_2(y)d g_{11}\\
=& \frac{c^2}{2\pi i}\oint_{\cC_1}\frac {\alpha }{g_{11}}\left(\int xh_{21}dH_1(x)\right)^2\int {yh_{11}}dH_2(y)\int {h_{11}^2}dH_2(y)d g_{11}\\
&-\frac{c^3}{\pi i}\oint_{\cC_1}\frac {\int h_{21}dH_1(x)}{g_{11}^2}\int xdH_1(x)\int xh_{21}dH_1(x)\int {yh_{11}^2}dH_2(y)d g_{11}\\
=& {c^2} {\alpha }\left(\int x^2dH_1(x)\right)^2\left(\int {y^2}dH_2(y)\right)^2\\
&-2{c^3}\int xdH_1(x)\left.\left[{\int h_{21}dH_1(x)}\int xh_{21}dH_1(x)\int {yh_{11}^2}dH_2(y)\right]'\right|_{g_{11}=0}\\
=& \frac{\alpha c^2} {4}\left(\int x^2dH_1(x)\right)^2+4{c^3}\left(\int xdH_1(x)\right)^2\int x^2dH_1(x)\int {y^4}dH_2(y)\\
=& \frac{\alpha c^2} {4}\left(\int x^2dH_1(x)\right)^2+\frac{3c^3}2\left(\int xdH_1(x)\right)^2\int x^2dH_1(x)
\end{align*}
where
\begin{align*}
\int {y}dH_2(y)&=\int {y^3}dH_2(y)=0
\end{align*}
and
\begin{align*}
\int {y^4}dH_2(y)&=\int_{-1}^1{y^4}\frac1{\pi\sqrt{1-y^2}}dy\\
&\xlongequal{y=\cos\theta}\frac1{\pi}\int_{0}^{\pi}{\cos^4\theta}d\theta=\frac1{2\pi}\int_{0}^{2\pi}{\cos^4\theta}
d\theta\\
&\xlongequal{s=e^{i\theta}}\frac1{2\pi i}\oint_{|s|=1}\frac{\left(s^2+1\right)^4}
{16s^5}ds=3/8.
\end{align*}

Using the same methods, we have
\begin{align*}
\mathcal{K}_{3}=&-\frac{ c}{4\pi^2}\oint_{\cC_1}\oint_{\cC_2}\frac{z_1z_2}{g_{11}g_{12}}\int h_{21}h_{22}dH_1(x)\int \frac{h_{11}^2h_{12}^2}{y^2}dH_2(y){d g_{11}}{d g_{12}}\\
=&\frac{ c^2}{4\pi^2}\oint_{\cC_1}\oint_{\cC_2}\frac{z_1\int h_{22}dH_1(x)}{g_{11}g_{12}^2}\int h_{21}h_{22}dH_1(x)\int \frac{h_{11}^2h_{12}^2}{y^2}dH_2(y){d g_{11}}{d g_{12}}\\
=&\frac{ c^2}{\pi i}\oint_{\cC_1}\frac{z_1\int xdH_1(x)}{g_{11}}\int xh_{21}dH_1(x)\int {y h_{11}^2}dH_2(y){d g_{11}}\\
=&4{ c^3}{\left(\int xdH_1(x)\right)^2}\int x^2dH_1(x)\int {y^4}dH_2(y)\\
=&\frac{3c^3}2{\left(\int xdH_1(x)\right)^2}\int x^2dH_1(x).
\end{align*}
Hence, we conclude that
\begin{align*}
\sigma^2={\rm Cov}\bigg( X_{f},X_f\bigg)=&\frac{ c^2(1+\alpha_x^2)} {2}\left(\int x^2dH_1(x)\right)^2
+\frac32{c^3}(\kappa_x+2)\left(\int xdH_1(x)\right)^2\int x^2dH_1(x).
\end{align*}

We are now in position to calculating $\mu=\re X_{f}$. Using (\ref{aa5}), it follows that
\begin{align*}
\re X_{f}=&\frac1{2\pi i}\oint \frac{c d_3(z)d_{4}(z)\int\frac x{1+g_2(z)x}dH_1(x)}{g_1(z)\int \frac{x}{(1+g_2(z)x)^2}dH_1(x)}\left(\frac{\alpha_x}{1-\frac{{{\alpha_x}}d_{3}(z)d_{4}(z)g_1^2(z)}{c\int x/(1+g_2(z)x)dH_1(x)}}+{\kappa_x}\right)dg_1(z)\\
=& {c }\left({\alpha_x}+{\kappa_x}\right)\int x^2dH_1(x)\int y^2dH_2(y)=\frac{{c }\left({\alpha_x}+{\kappa_x}\right)}{2}\int x^2dH_1(x).
\end{align*}

Furthermore, we see that
\begin{align*}
  &p\int x^2dF^{c_n,H_{1n},H_{2n}}\\
  =&-\frac{p}{2\pi i}\oint_{\mathcal{C}}z^2m_{n}^0(z)dz=\frac{p}{2\pi i}\oint_{\mathcal{C}}z\int\frac1{1+g_{2n}^0(z)x}dH_{1n}(x)dz\\
  =&\frac{p c_n^2}{2\pi i}\oint_{\mathcal{C}}\frac{\left(\int\frac{x}{1+g_{2n}^0(z)x}dH_{1n}(x)\right)^3\int\frac1{1+g_{2n}^0(z)x}dH_{1n}(x)}{g_{1n}^0(z)^3}
  \frac{1- \frac{(g_{1n}^0(z))^2}{c_n(\int\frac x{1+g_{2n}^0(z)}dH_{1n}(x))^2} d_{n3}(z)d_{n4}(z)}{\int\frac x{(1+xg_{2n}^0(z))^2}dH_{1n}(x)}dg_{1n}^0(z)\\
    =&\frac{p c_n^2}2\left.\left[{\left(\int\frac{x}{1+g_{2n}^0(z)x}dH_{1n}(x)\right)^3\int\frac1{1+g_{2n}^0(z)x}dH_{1n}(x)}
  \frac{1- \frac{(g_{1n}^0(z))^2}{c_n(\int\frac x{1+g_{2n}^0(z)}dH_{1n}(x))^2} d_{n3}(z)d_{n4}(z)}{\int\frac x{(1+xg_{2n}^0(z))^2}dH_{1n}(x)}\right]^{(2)}\right|_{g_{1n}^0=0}\\
  =&{pc_n}\left(\int x dH_{1n}(x)\right)^2\int y^2dH_{2n}(y)={pc_n}\left(\int x dH_{1n}(x)\right)^2\(\frac{1}{n}\sum_{k=1}^{n-\tau+1}\(\cos\frac{k}{n-\tau+1}\)^2\)\\\notag
  =&\frac{n-\tau}{2n}{pc_n}\left(\int x dH_{1n}(x)\right)^2.
\end{align*}
This lemma is thus proved.


\section{Appendix}

\subsection{Estimation of $\mathcal{I}_1$ and $\mathcal{I}_2$}\label{appes}

\subsubsection{Estimation of $\mathcal{I}_1$}

To begin with, we see that from Lemma \ref{le3}
\begin{align*}
  \varphi_{jk}'(w_{jk})=&\left[\bt_{1n}'\frac{\partial \widetilde f({\bf G}_{n})}{\partial w_{jk}}\bt_{1n}\bw_n\bt_{2n}\right]_{jk}e^{ixS^0(\theta)}+\left[\bt_{1n}'\widetilde f({\bf G}_{n})\bt_{1n}\right]_{jj}\left[\bt_{2n}\right]_{kk}e^{ixS^0(\theta)}\\
  &+\frac{2xi}{n}\left[\bt_{1n}'\widetilde f({\bf G}_{n})\bt_{1n}\bw_n\bt_{2n}\right]_{jk}^2e^{ixS^0(\theta)}.
\end{align*}
This yields
\begin{align*}
\varphi_{jk}^{(2)}(w_{jk}) =&\left[\bt_{1n}'\frac{\partial^2 \widetilde f({\bf G}_{n})}{\partial w_{jk}^2}\bt_{1n}\bw_n\bt_{2n}\right]_{jk}e^{ixS^0(\theta)}
+2\left[\bt_{1n}'\frac{\partial \widetilde f({\bf G}_{n})}{\partial w_{jk}}\bt_{1n}\right]_{jj}\left[\bt_{2n}\right]_{kk}e^{ixS^0(\theta)}\\
&+\frac{6xi}{n}\left[\bt_{1n}'\widetilde f({\bf G}_{n})\bt_{1n}\bw_n\bt_{2n}\right]_{jk}\left[\bt_{1n}'\frac{\partial \widetilde f({\bf G}_{n})}{\partial w_{jk}}\bt_{1n}\bw_n\bt_{2n}\right]_{jk}e^{ixS^0(\theta)}\\
&+\frac{6xi}{n}\left[\bt_{1n}'\widetilde f({\bf G}_{n})\bt_{1n}\bw_n\bt_{2n}\right]_{jk}\left[\bt_{1n}'\widetilde f({\bf G}_{n})\bt_{1n}\right]_{jj}\left[\bt_{2n}\right]_{kk}e^{ixS^0(\theta)}\\
    &-\frac{4x^2}{n^2}\left[\bt_{1n}'\widetilde f({\bf G}_{n})\bt_{1n}\bw_n\bt_{2n}\right]_{jk}^3e^{ixS^0(\theta)}\\
\triangleq&\mathcal{J}_{jk}^1+\mathcal{J}_{jk}^2+\mathcal{J}_{jk}^3+\mathcal{J}_{jk}^4+\mathcal{J}_{jk}^5.
\end{align*}

Applying (\ref{eq30}) and Lemma \ref{le5}, one finds
\begin{small}
\begin{align*}
  &\left[\bt_{1n}'\frac{\partial^2 \widetilde f({\bf G}_{n})}{\partial w_{jk}^2}\bt_{1n}\bw_n\bt_{2n}\right]_{jk}
  =i\int_{-\infty}^{\infty}u\widehat f(u)\left[\bt_{1n}'\frac{\partial^2 \bh_n(u)}{\partial w_{jk}^2}\bt_{1n}\bw_n\bt_{2n}\right]_{jk}du\\
=&-\frac 2 n\int_{-\infty}^{\infty}u\widehat f(u)\left[\bt_{2n}\right]_{kk}\left[\bt_{1n}'{\bf H}_n\bt_{1n}\right]_{jj}*\left[\bt_{1n}'{\bf H}_n\bt_{1n}\bw_n\bt_{2n}\right]_{jk}(u)du\\
&-\frac {6i}{ n^2}\int_{-\infty}^{\infty}u\widehat f(u)\left[\bt_{1n}'{\bf H}_n\bt_{1n}\right]_{jj}*\left[\bt_{2n}\bw_n'\bt_{1n}'{\bf H}_n\bt_{1n}\bw_n\bt_{2n}\right]_{kk}*\left[\bt_{1n}'{\bf H}_n\bt_{1n}\bw_n\bt_{2n}\right]_{jk}(u)du\\
&-\frac {2i}{ n^2}\int_{-\infty}^{\infty}u\widehat f(u)\left[\bt_{1n}'{\bf H}_n\bt_{1n}\bw_n\bt_{2n}\right]_{jk}*\left[\bt_{1n}'{\bf H}_n\bt_{1n}\bw_n\bt_{2n}\right]_{jk}*\left[\bt_{1n}'{\bf H}_n\bt_{1n}\bw_n\bt_{2n}\right]_{jk}(u)du.
\end{align*}
\end{small}

From \cite{yin2018no}, it can be seen that the moments of $\left\|\bt_{1n}'{\bf H}_n\bt_{1n}\right\|_2$, $\frac1{\sqrt n}\left\|\bt_{1n}'{\bf H}_n\bt_{1n}\bw_n\bt_{2n}\right\|_2$, and
\begin{align}\label{ee}
\frac1n\left\|\bt_{2n}\bw_n'\bt_{1n}'{\bf H}_n\bt_{1n}\bw_n\bt_{2n}\right\|_2
\end{align}
are bounded. Using Lemma \ref{le4} and (\ref{ee}), it is obvious that
\begin{align*}
&\left|\frac1{n}\sum_{j=1}^{m_1}\sum_{k=1}^{m_2}\re\mathcal{J}_{jk}^1\right|\\
\le&\frac 2{n^2}\int_{-\infty}^{\infty}|u\widehat f(u)|\int_0^u\bigg|\sum_{j,k}\re\left[\bt_{2n}\right]_{kk}\left[\bt_{1n}'{\bf H}_n(s)\bt_{1n}\right]_{jj}\left[\bt_{1n}'{\bf H}_n(u-s)\bt_{1n}\bw_n\bt_{2n}\right]_{jk}\bigg|dsdu\\
&+\frac {6}{ n^3}\int_{-\infty}^{\infty}|u\widehat f(u)|\int_0^u\int_0^s\bigg|\sum_{j,k}\re\left[\bt_{1n}'{\bf H}_n(s)\bt_{1n}\right]_{jj}\left[\bt_{2n}\bw_n'\bt_{1n}'{\bf H}_n(q)\bt_{1n}\bw_n\bt_{2n}\right]_{kk}\\
&\times\left[\bt_{1n}'{\bf H}_n(u-s-q)\bt_{1n}\bw_n\bt_{2n}\right]_{jk}\bigg|dqdsdu+\frac {2}{ n^3}\int_{-\infty}^{\infty}|u\widehat f(u)|\int_0^u\int_0^s\bigg|\sum_{j,k}\\
&\re\left[\bt_{1n}'{\bf H}_n(s)\bt_{1n}\bw_n\bt_{2n}\right]_{jk}\left[\bt_{1n}'{\bf H}_n(q)\bt_{1n}\bw_n\bt_{2n}\right]_{jk}\left[\bt_{1n}'{\bf H}_n(u-s-q)\bt_{1n}\bw_n\bt_{2n}\right]_{jk}\bigg|dqdsdu\\
\le&\frac C{n^{1/4}}\int_{-\infty}^{\infty}|u|^2|\widehat{f}(u)|du+\frac C{n^{1/4}}\int_{-\infty}^{\infty}|u|^3|\widehat{f}(u)|du+\frac C{n^{1/2}}\int_{-\infty}^{\infty}|u|^3|\widehat{f}(u)|du\le C n^{-1/4}.
\end{align*}

For $\mathcal{J}_{jk}^2$,
\begin{align*}
  &\left[\bt_{1n}'\frac{\partial \widetilde f({\bf G}_{n})}{\partial w_{jk}}\bt_{1n}\right]_{jj}=i\int_{-\infty}^{\infty}u\widehat f(u)\left[\bt_{1n}'\frac{\partial {\bf H}_{n}(u)}{\partial w_{jk}}\bt_{1n}\right]_{jj}du\\
  =&i\int_{-\infty}^{\infty}u\widehat f(u)\sum_{d,l}\left[\bt_{1n}'\right]_{jd}\left[\frac{\partial {\bf H}_{n}(u)}{\partial w_{jk}}\right]_{dl}\left[\bt_{1n}\right]_{lj}du\\
=&-\frac2 n\int_{-\infty}^{\infty}u\widehat f(u)\left[\bt_{1n}'{\bf H}_n\bt_{1n}\bw_n\bt_{2n}\right]_{jk}*\left[\bt_{1n}'{\bf H}_n\bt_{1n}\right]_{jj}(u)du.
\end{align*}
Due to Lemma \ref{le4} and (\ref{ee}), we have
\begin{align*}
\left|\frac1{n}\sum_{j=1}^{m_1}\sum_{k=1}^{m_2}\re\mathcal{J}_{jk}^2\right|\le&\frac4{n^2}\int_{-\infty}^{\infty}\int_0^u|u\widehat f(u)|\bigg|\sum_{j,k}\re\left[\bt_{1n}'{\bf H}_n(s)\bt_{1n}\right]_{jj}\left[\bt_{2n}\right]_{kk}\\
&\times\left[\bt_{1n}'{\bf H}_n(u-s)\bt_{1n}\bw_n\bt_{2n}\right]_{jk}\bigg|dsdu\\
\le&\frac C{n^{1/4}}\int_{-\infty}^{\infty}u^2|\widehat{f}(u)|du\le C n^{-1/4}.
\end{align*}

For $\mathcal{J}_{jk}^3$,
\begin{align*}
&\left[\bt_{1n}'\frac{\partial \widetilde f({\bf G}_{n})}{\partial w_{jk}}\bt_{1n}\bw_n\bt_{2n}\right]_{jk}=i\int_{-\infty}^{\infty}u\widehat f(u)\left[\bt_{1n}'\frac{\partial {\bf H}_{n}(u)}{\partial w_{jk}}\bt_{1n}\bw_n\bt_{2n}\right]_{jk}du\\
=&i\int_{-\infty}^{\infty}u\widehat f(u)\sum_{d,l}\left[\bt_{1n}'\right]_{jd}\left[\frac{\partial {\bf H}_{n}(u)}{\partial w_{jk}}\right]_{dl}\left[\bt_{1n}\bw_n\bt_{2n}\right]_{lk}du\\
=&-\frac1n\int_{-\infty}^{\infty}u\widehat f(u)\left[\bt_{1n}'{\bf H}_n\bt_{1n}\right]_{jj}*\left[\bt_{2n}\bw_n'\bt_{1n}'{\bf H}_n\bt_{1n}\bw_n\bt_{2n}\right]_{kk}(u)du\\
&-\frac1n\int_{-\infty}^{\infty}u\widehat f(u)\left[\bt_{1n}'{\bf H}_n\bt_{1n}\bw_n\bt_{2n}\right]_{jk}*\left[\bt_{1n}'{\bf H}_n\bt_{1n}\bw_n\bt_{2n}\right]_{jk}(u)du.
\end{align*}
Applying Lemma \ref{le4}, it follows that
\begin{align*}
\left|\frac1{n}\sum_{j=1}^{m_1}\sum_{k=1}^{m_2}\re\mathcal{J}_{jk}^3\right|\le&\frac{6|x|}{n^3}\int_{-\infty}^{\infty}\int_{-\infty}^{\infty}|u\widehat f(u)||s\widehat f(s)|\bigg|\sum_{j,k}\re\left[\bt_{1n}'{\bf H}_n(s)\bt_{1n}\bw_n\bt_{2n}\right]_{jk}\\
&\times\left[\bt_{1n}'{\bf H}_n\bt_{1n}\right]_{jj}*\left[\bt_{2n}\bw_n'\bt_{1n}'{\bf H}_n\bt_{1n}\bw_n\bt_{2n}\right]_{kk}(u)\bigg|dsdu\\
&+\frac{6|x|}{n^3}\int_{-\infty}^{\infty}\int_{-\infty}^{\infty}|u\widehat f(u)||s\widehat f(s)|\bigg|\sum_{j,k}\re\left[\bt_{1n}'{\bf H}_n(s)\bt_{1n}\bw_n\bt_{2n}\right]_{jk}\\
&\times\left[\bt_{1n}'{\bf H}_n\bt_{1n}\bw_n\bt_{2n}\right]_{jk}*\left[\bt_{1n}'{\bf H}_n\bt_{1n}\bw_n\bt_{2n}\right]_{jk}(u)\bigg|dsdu\\
\le&\frac C{n^{1/4}}\int_{-\infty}^{\infty}|s\widehat{f}(s)|ds\int_{-\infty}^{\infty}u^2|\widehat{f}(u)|du\le C n^{-1/4}.
\end{align*}

From Lemma \ref{le4}, one gets
\begin{align*}
\left|\frac1{n}\sum_{j=1}^{m_1}\sum_{k=1}^{m_2}\re\mathcal{J}_{jk}^4\right|\le&\frac{6|x|}{n^2}\int_{-\infty}^{\infty}\int_{-\infty}^{\infty}|u\widehat f(u)||s\widehat f(s)|\bigg|\sum_{j,k}\re\left[\bt_{1n}'{\bf H}_{n}(s)\bt_{1n}\bw_n\bt_{2n}\right]_{jk}\\
&\times\left[\bt_{1n}'{\bf H}_{n}(u)\bt_{1n}\right]_{jj}\left[\bt_{2n}\right]_{kk}\bigg|dtdsdu\\
\le&\frac C{n^{1/4}}\int_{-\infty}^{\infty}|u\widehat{f}(u)|du\int_{-\infty}^{\infty}|s\widehat{f}(s)|ds\le C n^{-1/4}
\end{align*}
and
\begin{align*}
\left|\frac1{n}\sum_{j=1}^{m_1}\sum_{k=1}^{m_2}\re\mathcal{J}_{jk}^5\right|\le&\frac{4x^2}{n^3}\int_{-\infty}^{\infty}\int_{-\infty}^{\infty}\int_{-\infty}^{\infty}|u\widehat f(u)||s\widehat f(s)||t\widehat f(t)|\bigg|\sum_{j,k}\re\left[\bt_{1n}'{\bf H}_{n}(t)\bt_{1n}\bw_n\bt_{2n}\right]_{jk}\\
&\times\left[\bt_{1n}'{\bf H}_{n}(s)\bt_{1n}\bw_n\bt_{2n}\right]_{jk}\left[\bt_{1n}'{\bf H}_{n}(u)\bt_{1n}\bw_n\bt_{2n}\right]_{jk}\bigg|dsdu\\\le&\frac C{n^{1/2}}\int_{-\infty}^{\infty}|u\widehat{f}(u)|du\int_{-\infty}^{\infty}|s\widehat{f}(s)|ds\int_{-\infty}^{\infty}|t\widehat{f}(t)|dt\le C n^{-1/2}.
\end{align*}

Hence,
\begin{align*}
&\left|\mathcal{I}_1\right|\to0\quad{\rm as}\ n\to\infty.
\end{align*}

\subsubsection{Estimation of $\mathcal{I}_2$}

Let $\tilde w$ be a random variable which has the same distribution as $w_{jk}$. Then
\begin{align*}
 \left|\mathcal{I}_2\right|\le&\frac {|x|}{12n} \sum_{j=1}^{m_1}\sum_{k=1}^{m_2} \re \left|w_{jk}' w_{jk}^4\right|\re\sup_{\hat w}\left|\varphi_{jk}^{(4)}(\varrhoup \tilde w)\right|\\
 \le&\frac {M|x|}{12 n } \sum_{j=1}^{m_1}\sum_{k=1}^{m_2}\re\sup_{\tilde w}\left|\varphi_{jk}^{(4)}(\varrhoup \tilde w)\right|.
\end{align*}
Let $w=\varrhoup \tilde w$. It can be verified that from Lemma \ref{le3}
\begin{small}
\begin{align*}
&\varphi_{jk}^{(3)}(w) =\left[\bt_{1n}'\frac{\partial^3 \widetilde f({\bf G}_{njkw})}{\partial w^3}\bt_{1n}\bw_{njkw}\bt_{2n}\right]_{jk}e^{ixS^0(w,\theta)}\\
&+3\left[\bt_{1n}'\frac{\partial^2 \widetilde f({\bf G}_{njkw})}{\partial w^2}\bt_{1n}\right]_{jj}\left[\bt_{2n}\right]_{kk}e^{ixS^0(w,\theta)}\\
&+\frac {8xi} n\left[\bt_{1n}'\frac{\partial^2 \widetilde f({\bf G}_{njkw})}{\partial w^2}\bt_{1n}\bw_{njkw}\bt_{2n}\right]_{jk}\left[\bt_{1n}'\widetilde f({\bf G}_{njkw})\bt_{1n}\bw_{njkw}\bt_{2n}\right]_{jk}e^{ixS^0(w,\theta)}\\
&+\frac {16xi} n\left[\bt_{1n}'\frac{\partial \widetilde f({\bf G}_{njkw})}{\partial w}\bt_{1n}\right]_{jj}\left[\bt_{2n}\right]_{kk}\left[\bt_{1n}'\widetilde f({\bf G}_{njkw})\bt_{1n}\bw_{njkw}\bt_{2n}\right]_{jk}e^{ixS^0(w,\theta)}\\
&+\frac{6xi}{n}\left[\bt_{1n}'\frac{\partial \widetilde f({\bf G}_{njkw})}{\partial w}\bt_{1n}\bw_{njkw}\bt_{2n}\right]_{jk}^2e^{ixS^0(w,\theta)}\\
&+\frac{12xi}{n}\left[\bt_{1n}'\widetilde f({\bf G}_{njkw})\bt_{1n}\right]_{jj}\left[\bt_{2n}\right]_{kk}\left[\bt_{1n}'\frac{\partial \widetilde f({\bf G}_{njkw})}{\partial w}\bt_{1n}\bw_{njkw}\bt_{2n}\right]_{jk}e^{ixS^0(w,\theta)}\\
&-\frac{24x^2}{n^2}\left[\bt_{1n}'\widetilde f({\bf G}_{njkw})\bt_{1n}\bw_{njkw}\bt_{2n}\right]_{jk}^2\left[\bt_{1n}'\frac{\partial \widetilde f({\bf G}_{njkw})}{\partial w}\bt_{1n}\bw_{njkw}\bt_{2n}\right]_{jk}e^{ixS^0(w,\theta)}\\
&+\frac{6xi}{n}\left[\bt_{1n}'\widetilde f({\bf G}_{njkw})\bt_{1n}\right]_{jj}^2\left[\bt_{2n}\right]_{kk}^2e^{ixS^0(w,\theta)}\\
&-\frac{24x^2}{n^2}\left[\bt_{1n}'\widetilde f({\bf G}_{njkw})\bt_{1n}\bw_{njkw}\bt_{2n}\right]_{jk}^2\left[\bt_{1n}'\widetilde f({\bf G}_{njkw})\bt_{1n}\right]_{jj}\left[\bt_{2n}\right]_{kk}e^{ixS^0(w,\theta)}\\
    &-\frac{8x^3i}{n^3}\left[\bt_{1n}'\widetilde f({\bf G}_{njkw})\bt_{1n}\bw_{njkw}\bt_{2n}\right]_{jk}^4e^{ixS^0(w,\theta)}.
\end{align*}
\end{small}
It is easy to obtain from the above equality
\begin{small}
\begin{align*}
&\varphi_{jk}^{(4)}(w) =\left[\bt_{1n}'\frac{\partial^4 \widetilde f({\bf G}_{njkw})}{\partial w^4}\bt_{1n}\bw_{njkw}\bt_{2n}\right]_{jk}e^{ixS^0(w,\theta)}\\
&+4\left[\bt_{1n}'\frac{\partial^3 \widetilde f({\bf G}_{njkw})}{\partial w^3}\bt_{1n}\right]_{jj}\left[\bt_{2n}\right]_{kk}e^{ixS^0(w,\theta)}\\
&+\frac{10xi}{n}\left[\bt_{1n}'\frac{\partial^3 \widetilde f({\bf G}_{njkw})}{\partial w^3}\bt_{1n}\bw_{njkw}\bt_{2n}\right]_{jk}\left[\bt_{1n}'\widetilde f({\bf G}_{njkw})\bt_{1n}\bw_{njkw}\bt_{2n}\right]_{jk}e^{ixS^0(w,\theta)}\\
&+\frac{30xi}{n}\left[\bt_{1n}'\frac{\partial^2 \widetilde f({\bf G}_{njkw})}{\partial w^2}\bt_{1n}\right]_{jj}\left[\bt_{2n}\right]_{kk}\left[\bt_{1n}'\widetilde f({\bf G}_{njkw})\bt_{1n}\bw_{njkw}\bt_{2n}\right]_{jk}e^{ixS^0(w,\theta)}\\
&+\frac {20xi} n\left[\bt_{1n}'\frac{\partial^2 \widetilde f({\bf G}_{njkw})}{\partial w^2}\bt_{1n}\bw_{njkw}\bt_{2n}\right]_{jk}\left[\bt_{1n}'\frac{\partial\widetilde f({\bf G}_{njkw})}{\partial w}\bt_{1n}\bw_{njkw}\bt_{2n}\right]_{jk}e^{ixS^0(w,\theta)}\\
&+\frac {20xi} n\left[\bt_{1n}'\frac{\partial^2 \widetilde f({\bf G}_{njkw})}{\partial w^2}\bt_{1n}\bw_{njkw}\bt_{2n}\right]_{jk}\left[\bt_{1n}'\widetilde f({\bf G}_{njkw})\bt_{1n}\right]_{jj}\left[\bt_{2n}\right]_{kk}e^{ixS^0(w,\theta)}\\
&+\frac {40xi} n\left[\bt_{1n}'\frac{\partial \widetilde f({\bf G}_{njkw})}{\partial w}\bt_{1n}\right]_{jj}\left[\bt_{2n}\right]_{kk}\left[\bt_{1n}'\frac{\partial\widetilde f({\bf G}_{njkw})}{\partial w}\bt_{1n}\bw_{njkw}\bt_{2n}\right]_{jk}e^{ixS^0(w,\theta)}\\
&+\frac {40xi} n\left[\bt_{1n}'\frac{\partial \widetilde f({\bf G}_{njkw})}{\partial w}\bt_{1n}\right]_{jj}\left[\bt_{1n}'\widetilde f({\bf G}_{njkw})\bt_{1n}\right]_{jj}\left[\bt_{2n}\right]_{kk}^2e^{ixS^0(w,\theta)}\\
&-\frac {40x^2} {n^2}\left[\bt_{1n}'\frac{\partial^2 \widetilde f({\bf G}_{njkw})}{\partial w^2}\bt_{1n}\bw_{njkw}\bt_{2n}\right]_{jk}\left[\bt_{1n}'\widetilde f({\bf G}_{njkw})\bt_{1n}\bw_{njkw}\bt_{2n}\right]_{jk}^2e^{ixS^0(w,\theta)}\\
&-\frac {80x^2} {n^2}\left[\bt_{1n}'\frac{\partial \widetilde f({\bf G}_{njkw})}{\partial w}\bt_{1n}\right]_{jj}\left[\bt_{2n}\right]_{kk}\left[\bt_{1n}'\widetilde f({\bf G}_{njkw})\bt_{1n}\bw_{njkw}\bt_{2n}\right]_{jk}^2e^{ixS^0(w,\theta)}\\
&-\frac{60x^2}{n^2}\left[\bt_{1n}'\frac{\partial \widetilde f({\bf G}_{njkw})}{\partial w}\bt_{1n}\bw_{njkw}\bt_{2n}\right]_{jk}^2\left[\bt_{1n}'\widetilde f({\bf G}_{njkw})\bt_{1n}\bw_{njkw}\bt_{2n}\right]_{jk}e^{ixS^0(w,\theta)}\\
&-\frac{60x^2}{n^2}\left[\bt_{1n}'\widetilde f({\bf G}_{njkw})\bt_{1n}\right]_{jj}^2\left[\bt_{2n}\right]_{kk}^2\left[\bt_{1n}'\widetilde f({\bf G}_{njkw})\bt_{1n}\bw_{njkw}\bt_{2n}\right]_{jk}e^{ixS^0(w,\theta)}\\
&-\frac{120x^2}{n^2}\left[\bt_{1n}'\widetilde f({\bf G}_{njkw})\bt_{1n}\right]_{jj}\left[\bt_{2n}\right]_{kk}\left[\bt_{1n}'\frac{\partial \widetilde f({\bf G}_{njkw})}{\partial w}\bt_{1n}\bw_{njkw}\bt_{2n}\right]_{jk}\\
&\quad\times\left[\bt_{1n}'\widetilde f({\bf G}_{njkw})\bt_{1n}\bw_{njkw}\bt_{2n}\right]_{jk}e^{ixS^0(w,\theta)}\\
&-\frac{80x^3i}{n^3}\left[\bt_{1n}'\widetilde f({\bf G}_{njkw})\bt_{1n}\bw_{njkw}\bt_{2n}\right]_{jk}^3\left[\bt_{1n}'\frac{\partial \widetilde f({\bf G}_{njkw})}{\partial w}\bt_{1n}\bw_{njkw}\bt_{2n}\right]_{jk}e^{ixS^0(w,\theta)}\\
&-\frac{80x^3i}{n^3}\left[\bt_{1n}'\widetilde f({\bf G}_{njkw})\bt_{1n}\bw_{njkw}\bt_{2n}\right]_{jk}^3\left[\bt_{1n}'\widetilde f({\bf G}_{njkw})\bt_{1n}\right]_{jj}\left[\bt_{2n}\right]_{kk}e^{ixS^0(w,\theta)}\\
&+\frac{16x^4}{n^4}\left[\bt_{1n}'\widetilde f({\bf G}_{njkw})\bt_{1n}\bw_{njkw}\bt_{2n}\right]_{jk}^5e^{ixS^0(w,\theta)}\\
\triangleq&\bigg(\mathcal{Q}_{jk}^1+\mathcal{Q}_{jk}^2+\mathcal{Q}_{jk}^3+\mathcal{Q}_{jk}^4+\mathcal{Q}_{jk}^5+\mathcal{Q}_{jk}^6+\mathcal{Q}_{jk}^7+\mathcal{Q}_{jk}^8
+\mathcal{Q}_{jk}^9+\mathcal{Q}_{jk}^{10}+\mathcal{Q}_{jk}^{11}+\mathcal{Q}_{jk}^{12}\\
&+\mathcal{Q}_{jk}^{13}+\mathcal{Q}_{jk}^{14}+\mathcal{Q}_{jk}^{15}+\mathcal{Q}_{jk}^{16}\bigg) \ e^{ixS^0(w,\theta)}.
\end{align*}
\end{small}

Using Lemma \ref{le7}, it follows that
\begin{footnotesize}
\begin{align*}
&\mathcal{Q}_{jk}^1
=i\int_{-\infty}^{\infty}u\widehat f(u)\left[\bt_{1n}'\frac{\partial^4 {\bf H}_{njkw}(u)}{\partial w_{jk}^4}\bt_{1n}\bw_{njkw}\bt_{2n}\right]_{jk}e^{ixS^0(w,\theta)}du\\
=&-\frac{24i}{n^2}\int_{-\infty}^{\infty}u\widehat f(u)\left[\bt_{2n}\right]_{kk}^2\left[\bt_{1n}'{\bf H}_{njkw}\bt_{1n}\right]_{jj}*\left[\bt_{1n}'{\bf H}_{njkw}\bt_{1n}\right]_{jj}*\left[\bt_{1n}'{\bf H}_{njkw}\bt_{1n}\bw_{njkw}\bt_{2n}\right]_{jk}(u)du\\
&+\frac {144}{n^3}\int_{-\infty}^{\infty}u\widehat f(u)\left[\bt_{2n}\right]_{kk}\left[\bt_{1n}'{\bf H}_{njkw}\bt_{1n}\right]_{jj}*\left[\bt_{2n}\bw_{njkw}'\bt_{1n}'{\bf H}_{njkw}\bt_{1n}\right]_{kj}*\left[\bt_{2n}\bw_{njkw}'\bt_{1n}'{\bf H}_{njkw}\bt_{1n}\right]_{kj}\\
&\qquad\qquad*\left[\bt_{1n}'{\bf H}_{njkw}\bt_{1n}\bw_{njkw}\bt_{2n}\right]_{jk}(u)du\\
&+\frac {144}{n^3}\int_{-\infty}^{\infty}u\widehat f(u)\left[\bt_{2n}\right]_{kk}\left[\bt_{1n}'{\bf H}_{njkw}\bt_{1n}\right]_{jj}*\left[\bt_{2n}\bw_{njkw}'\bt_{1n}'{\bf H}_{njkw}\bt_{1n}\bw_{njkw}\bt_{2n}\right]_{kk}*\left[\bt_{1n}'{\bf H}_{njkw}\bt_{1n}\right]_{jj}\\
&\qquad\qquad*\left[\bt_{1n}'{\bf H}_{njkw}\bt_{1n}\bw_{njkw}\bt_{2n}\right]_{jk}(u)du\\
&+\frac {240i}{n^4}\int_{-\infty}^{\infty}u\widehat f(u)\left[\bt_{1n}'{\bf H}_{njkw}\bt_{1n}\right]_{jj}*\left[\bt_{2n}\bw_{njkw}'\bt_{1n}'{\bf H}_{njkw}\bt_{1n}\right]_{kj}*\left[\bt_{2n}\bw_{njkw}'\bt_{1n}'{\bf H}_{njkw}\bt_{1n}\right]_{kj}\\
&\qquad\qquad*\left[\bt_{2n}\bw_{njkw}'\bt_{1n}'{\bf H}_{njkw}\bt_{1n}\right]_{kj}*\left[\bt_{2n}\bw_{njkw}'\bt_{1n}'{\bf H}_{njkw}\bt_{1n}\bw_{njkw}\bt_{2n}\right]_{kk}(u)du\\
&+\frac {120i}{n^4}\int_{-\infty}^{\infty}u\widehat f(u)\left[\bt_{1n}'{\bf H}_{njkw}\bt_{1n}\right]_{jj}*\left[\bt_{2n}\bw_{njkw}'\bt_{1n}'{\bf H}_{njkw}\bt_{1n}\bw_{njkw}\bt_{2n}\right]_{kk}*\left[\bt_{1n}'{\bf H}_{njkw}\bt_{1n}\right]_{jj}\\
&\qquad\qquad*\left[\bt_{2n}\bw_{njkw}'\bt_{1n}'{\bf H}_{njkw}\bt_{1n}\right]_{kj}*\left[\bt_{2n}\bw_{njkw}'\bt_{1n}'{\bf H}_{njkw}\bt_{1n}\bw_{njkw}\bt_{2n}\right]_{kk}(u)du\\
&+\frac {24i}{n^4}\int_{-\infty}^{\infty}u\widehat f(u)\left[\bt_{1n}'{\bf H}_{njkw}\bt_{1n}\bw_{njkw}\bt_{2n}\right]_{jk}*\left[\bt_{1n}'{\bf H}_{njkw}\bt_{1n}\bw_{njkw}\bt_{2n}\right]_{jk}*\left[\bt_{1n}'{\bf H}_{njkw}\bt_{1n}\bw_{njkw}\bt_{2n}\right]_{jk}\\
&\qquad\qquad*\left[\bt_{1n}'{\bf H}_{njkw}\bt_{1n}\bw_{njkw}\bt_{2n}\right]_{jk}*\left[\bt_{1n}'{\bf H}_{njkw}\bt_{1n}\bw_{njkw}\bt_{2n}\right]_{jk}(u)du.
\end{align*}
\end{footnotesize}
From Lemma \ref{le8}, we see that
\begin{align*}
  \frac1n\sum_{j,k}\re\left|\mathcal{Q}_{jk}^1e^{ixS^0(w,\theta)}\right|\le\frac Cn\int_{-\infty}^{\infty}\left(|u|^3+u^4+|u|^5\right)\widehat f(u)du\to0.
\end{align*}

For $\mathcal{Q}_{jk}^2$, by Lemma \ref{le6}, one gets
\begin{align*}
 &\mathcal{Q}_{jk}^2=4i\int_{-\infty}^{\infty}u\widehat f(u)\left[\bt_{1n}'\frac{\partial^3 {\bf H}_{njkw}(u)}{\partial w^3}\bt_{1n}\right]_{jj}\left[\bt_{2n}\right]_{kk}du\\
=&-\frac{96i}{n^2}\int_{-\infty}^{\infty}u\widehat f(u)\left[\bt_{2n}\right]_{kk}^2\left[\bt_{1n}'{\bf H}_n\bt_{1n}\right]_{jj}*\left[\bt_{2n}\bw_n'\bt_{1n}'{\bf H}_n\bt_{1n}\right]_{kj}*\left[\bt_{1n}'{\bf H}_n\bt_{1n}\right]_{jj}(u)du\\
&+\frac{96}{n^3} \int_{-\infty}^{\infty}u\widehat f(u)\left[\bt_{2n}\right]_{kk}\left[\bt_{1n}'{\bf H}_n\bt_{1n}\right]_{jj}*\left[\bt_{2n}\bw_n'\bt_{1n}'{\bf H}_n\bt_{1n}\right]_{kj}*\left[\bt_{2n}\bw_n'\bt_{1n}'{\bf H}_n\bt_{1n}\right]_{kj}\\
&\qquad*\left[\bt_{2n}\bw_n'\bt_{1n}'{\bf H}_n\bt_{1n}\right]_{kj}(u)du\\
&+\frac{96}{n^3} \int_{-\infty}^{\infty}u\widehat f(u)\left[\bt_{2n}\right]_{kk}\left[\bt_{1n}'{\bf H}_n\bt_{1n}\right]_{jj}*\left[\bt_{2n}\bw_n'\bt_{1n}'{\bf H}_n\bt_{1n}\bw_n\bt_{2n}\right]_{kk}*\left[\bt_{1n}'{\bf H}_n\bt_{1n}\right]_{jj}\\
&\qquad*\left[\bt_{2n}\bw_n'\bt_{1n}'{\bf H}_n\bt_{1n}\right]_{kj}(u)du .
\end{align*}
Thus, applying Lemma \ref{le8} again, it yields that
\begin{align*}
  \frac1n\sum_{j,k}\re\left|\mathcal{Q}_{jk}^2e^{ixS^0(w,\theta)}\right|\le\frac Cn\int_{-\infty}^{\infty}\left( u^2+|u|^3\right)\widehat f(u)du\to0.
\end{align*}
The remaining terms are similar. Consequently, we conclude that
\begin{align*}
&\left|\mathcal{I}_2\right|\to0\quad{\rm as}\ n\to\infty.
\end{align*}

\subsection{proof of Lemma \ref{th2}}
 We start with two probability inequalities for extreme eigenvalues of $\bs_n$. It is well known (see \cite{bai2004clt,Yin1988}) that for any $l$, $\eta_1>(1+\sqrt c)^2$ and $\eta_2<(1-\sqrt c)^2$
\begin{align*}
{\rm P}\left(\lambda_{\max}\left(\frac1n\bx_n\bx_n^*\right)\ge\eta_1\right)=o(n^{-l})
\end{align*}
and
\begin{align*}
{\rm P}\left(\lambda_{\min}\left(\frac1n\bx_n\bx_n^*\right)\le\eta_2\right)=o(n^{-l}).
\end{align*}
Thus, letting
\begin{align*}
\eta_r\in
\begin{cases}
(0,x_r),&c\ge1,\\
(\limsup_{n}s_1\lambda_{\max}^{\bS_1}\left(1+\sqrt c\right)^2,x_r),& {\rm otherwise},
\end{cases}
\end{align*}
we have for any $l>0$
\begin{align*}
{\rm P}\left(\lambda_{\max}\left(\bs_n\right)\ge\eta_r\right)=o(n^{-l}).
\end{align*}
Likewise, we have
\begin{align*}
{\rm P}\left(\lambda_{\min}\left(\bs_n\right)\le\eta_l\right)=o(n^{-l}).
\end{align*}
where
\begin{align*}
\eta_l\in
\begin{cases}
(x_l ,0),&c\ge1,\\
(x_l,\liminf_ns_n\lambda_{\min}^{\bS_1}I_{(0,1)}(c)\left(1-\sqrt c\right)^2),&{\rm if} \ \liminf_ns_n\lambda_{\min}^{\bS_1}I_{(0,1)}(c)>0,\\
(x_l,\liminf_ns_n\lambda_{\max}^{\bS_1}\left(1+\sqrt c\right)^2),&{\rm if} \ \liminf_ns_n\lambda_{\min}^{\bS_1}I_{(0,1)}(c)\le0.
\end{cases}
\end{align*}
Here $\eta_l,\eta_r,x_l,x_r$ can be chosen such that
\begin{align}\label{eq18}
x_r-\eta_r>2\tau^2\quad {\rm and} \quad\eta_l-x_l>2\tau^2.
\end{align}

\subsection{The limiting distribution of $M_{n1}(z)$}

The aim of this part is to find the limiting distribution of $M_{n1}(z)$. That is to say, we show for any positive integer $r$, the sum
  $$\sum_{j=1}^r\alpha_jM_{n1}(z_j) \qquad\Im z_j\neq0$$
converges in distribution to a Gaussian random variable. Since
\begin{align*}
\lim_{v_0\downarrow 0}\limsup_{n\to\infty}\re\left|\int_{\mathcal{C}_l\cup\mathcal{C}_r}f(z)M_{n1}(z)dz\right|^2\to0,
\end{align*}
 it suffices to consider $z=u+iv_0\in \mathcal{C}_u$.

Note that
\begin{align*}
m_n(z)=\frac1p\rtr\left(\bs_n-z\bi_p\right)^{-1}\triangleq \frac1p\rtr \bd^{-1}(z).
\end{align*}
Let ${\rm E}_{0}(\cdot)$ denote mathematical expectation and ${\rm E}_{k}(\cdot)$ denote conditional expectation with respect to the $\sigma$-field given by ${\bf x}_1,\cdots,{\bf x}_k.$ By the formula
\begin{align}\label{bal2}
\left(\boldsymbol\Sigma+q\boldsymbol{\alpha\beta}^*\right)^{-1}=\boldsymbol\Sigma^{-1}-\frac{q\boldsymbol\Sigma^{-1}\boldsymbol{\alpha\beta}^*\boldsymbol\Sigma^{-1}}
{1+q\boldsymbol{\beta}^*
\boldsymbol\Sigma^{-1}\boldsymbol\alpha},
\end{align} we have
\begin{align}
M_{n1}(z)
=&\sum_{k=1}^n\rtr\left\{\re_{k}\bd^{-1}(z)-\re_{k-1}\bd^{-1}(z)\right\}
=\sum_{k=1}^n\left({\rm E}_{k}-{\rm E}_{k-1}\right){\rm tr}\left[\bd(z)^{-1}-\bd_k^{-1}(z)\right]\notag\\
=&-\frac1n\sum_{k=1}^n\left({\rm E}_{k}-{\rm E}_{k-1}\right)s_k\beta_k(z)\bq_k^*\bd_k^{-2}(z)\bq_k\notag\\
=&-\frac1n\sum_{k=1}^n\left({\rm E}_{k}-{\rm E}_{k-1}\right)s_k\beta_k(z)\gamma_k(z)-\frac1n\sum_{k=1}^n\left({\rm E}_{k}-{\rm E}_{k-1}\right)s_k\beta_k(z)\rtr(\bd_k^{-2}(z)\bS_1)\notag\\
\triangleq&\mathcal{I}_1+\mathcal{I}_2.\label{beq12}
\end{align}
From the identity
\begin{align}\label{bal10}
\beta_k(z)-\widetilde{\beta}_k(z)=-\frac1n s_k \widetilde{\beta}_k(z)\beta_k(z)\varepsilon_k(z),
\end{align}
we have
\begin{align*}
\mathcal{I}_1=&-\frac1n\sum_{k=1}^n{\rm E}_{k}s_k\widetilde{\beta}_k(z)\gamma_k(z)+\frac1{n^2}\sum_{k=1}^n\left({\rm E}_{k}-{\rm E}_{k-1}\right)s_k^2 \widetilde{\beta}_k(z)\beta_k(z) \varepsilon_k(z)\gamma_k(z).
\end{align*}
It is obvious from Lemma 0.1 for $l\ge1$
\begin{align}\label{cal12}
  \re\left|\beta_k(z)\right|^l\le C,\quad\re\left|\widetilde\beta_k(z)\right|^l\le C.
\end{align}
By Lemma \ref{ble3} and (\ref{cal12}), it yields
\begin{align*}
&\frac1{n^4}\sum_{k=1}^n\re|\left({\rm E}_{k}-{\rm E}_{k-1}\right)s_k^2 \widetilde{\beta}_k(z)\beta_k(z) \varepsilon_k(z)\gamma_k(z)|^2\\
\le&\frac C{n^4}\sum_{k=1}^n\re^{1/2}|\widetilde{\beta}_k(z)\beta_k(z)|^4\re^{1/2}|\varepsilon_k(z)\gamma_k(z)|^4\\
\le&\frac C{n^4}\sum_{k=1}^n\re^{1/4}|\varepsilon_k(z)|^8\re^{1/4}|\gamma_k(z)|^8\le C\eta_n^5\to0.
\end{align*}
This implies
\begin{align}\label{beq10}
\mathcal{I}_1=-\frac1n\sum_{k=1}^n{\rm E}_{k}s_k\widetilde{\beta}_k(z)\gamma_k(z)+o_p(1).
\end{align}
 Using the same argument and
\begin{align}\label{bal11}
\beta_k(z)-\widetilde{\beta}_k(z)=-\frac1n s_k \widetilde{\beta}_k^2(z)\varepsilon_k(z)+\frac1{n^2}s_k^2\beta_k(z)\widetilde{\beta}_k^2(z)\varepsilon_k^2(z),
\end{align}
 one gets
\begin{align}\label{beq11}
\mathcal{I}_2=\frac1{n^2}\sum_{k=1}^n{\rm E}_{k}s_k^2\widetilde{\beta}_k^2(z)\varepsilon_k(z)\rtr(\bd_k^{-2}(z)\bS_1)+o_p(1).
\end{align}
From (\ref{beq12}), (\ref{beq10}), and (\ref{beq11}), we conclude that
\begin{align*}
M_{n1}(z)=&-\frac1n\sum_{k=1}^n{\rm E}_{k}s_k\widetilde{\beta}_k(z)\gamma_k(z)+\frac1{n^2}\sum_{k=1}^n{\rm E}_{k}s_k^2\widetilde{\beta}_k^2(z)\varepsilon_k(z)\rtr(\bd_k^{-2}(z)\bS_1)+o_p(1).
\end{align*}
Define
\begin{align*}
h_k(z)=&-\frac1n{\rm E}_{k}s_k\widetilde{\beta}_k(z)\gamma_k(z)+\frac1{n^2}{\rm E}_{k}s_k^2\widetilde{\beta}_k^2(z)\varepsilon_k(z)\rtr(\bd_k^{-2}(z)\bS_1)\\
=&-n^{-1}\frac{d}{dz}{\rm E}_{k}s_k\widetilde{\beta}_k(z)\varepsilon_k(z).
\end{align*}
Thus we only need to prove that $\sum_{j=1}^r\alpha_j\sum_{k=1}^nh_k(z_j)=\sum_{k=1}^n\sum_{j=1}^r\alpha_jh_k(z_j)$ converges in distribution to a Gaussian random variable. By Lemma \ref{lep2}, it suffices to verify condition $(i)$ and $(ii)$. It follows from Lemma \ref{ble3} that
\begin{align*}
\sum_{k=1}^n\re\left|\sum_{j=1}^r\alpha_jh_k(z_j)\right|^4\le&\frac C{n^4}\sum_{k=1}^n\sum_{j=1}^r\alpha_j^4\bigg[\re|\gamma_k(z_j)|^4+\re|\varepsilon_k(z_j)|^4\bigg]\le C \eta_n^2\to0
\end{align*}
which implies that conditions $(ii)$ of Lemma \ref{lep2}. The goal turns into finding a limit in probability of
\begin{align}\label{ben1}
\Phi(z_1,z_2)\triangleq\sum_{k=1}^n\re_{k-1}\left[h_k(z_1)h_k(z_2)\right]
\end{align}
 for $z_1,z_2$ with nonzero fixed imaginary parts.

 It is obvious that
\begin{align*}
\Phi(z_1,z_2)=n^{-2}\frac{\partial^2}{\partial z_2\partial z_1}\sum_{k=1}^n\re_{k-1}\left[\re_k\left(s_k\widetilde{\beta}_k(z_1)\varepsilon_k(z_1)\right)\re_k
\left(s_k\widetilde{\beta}_k(z_2)\varepsilon_k(z_2)\right)\right].
\end{align*}
Due to the analysis on page 571 in \cite{bai2004clt}, it is enough to prove that
$$n^{-2}\sum_{k=1}^n s_k^2\re_{k-1}\left[\re_k\left(\widetilde{\beta}_k(z_1)\varepsilon_k(z_1)\right)\re_k\left(\widetilde{\beta}_k(z_2)\varepsilon_k(z_2)\right)\right]$$
converges in probability to a constant. Using (\ref{cl3}), we get
\begin{align*}
  \widetilde{\beta}_k(z)=\frac1{1+s_kg_1(z)}+o_{\rm a.s.}(1).
\end{align*}
Therefore, our goal is to find the limit in probability of
\begin{align*}
&n^{-2}\sum_{k=1}^n \frac{s_k^2}{(1+s_kg_1(z_1))(1+s_kg_1(z_2))} \re_{k-1}\left[\re_k\left(\varepsilon_k(z_1)\right)\re_k\left(\varepsilon_k(z_2)\right)\right].
\end{align*}
Using the moments of the random variables, we have
\begin{align*}
\re_{k-1}\big[\re_k\left(\varepsilon_k(z_1)\right)&\re_k\left(\varepsilon_k(z_2)\right)\big]
=\rtr\left(\bS_1\re_k\bd_k^{-1}(z_1)\bS_1\re_k\bd_k^{-1}(z_2)\right)\\
&+\alpha_x\rtr\left(\bS_1\re_k\bd_k^{-1}(z_1)\bS_1\re_k\left(\bd_k'(z_2)\right)^{-1}\right)\\
&+\kappa_x\sum_{j\neq k=1}^n\be_j'\bt_{1n}^*\re_k\bd_{k}^{-1}(z_1)\bt_{1n}\be_j\be_j'\bt_{1n}^*\re_k\bd_{k}^{-1}(z_2)\bt_{1n}\be_j+o(n)\\
=&\rtr\re_k\left(\bS_1\bd_k^{-1}(z_1)\bS_1\breve\bd_k^{-1}(z_2)\right)+\alpha_x\rtr\re_k\left(\bS_1\bd_k^{-1}(z_1)\bS_1\left(\breve\bd_k'(z_2)\right)^{-1}\right)\\
&+\kappa_x\sum_{j\neq k=1}^n\re_k\be_j'\bt_{1n}^*\bd_{k}^{-1}(z_1)\bt_{1n}\be_j\be_j'\bt_{1n}^*\breve\bd_{k}^{-1}(z_2)\bt_{1n}\be_j+o(n).
\end{align*}
Here $\bt_{1n}$ is real and $\breve\bd_k(z)=\frac1n\sum_{j<k}s_j\bq_j\bq_j^*+\frac1n\sum_{j>k}s_j\breve\bq_j\breve\bq_j^*-z\bi_p$  where $\breve\bq_j$ are an i.i.d. copy of $\bq_j, j = 1,\cdots,n.$

To begin with, we calculate the limit of the first term. From (\ref{cl3}),  one can find that
\begin{align}\label{cl1}
  \frac1n\re_k\left[z_1\rtr\left(\bS_1\bd_k^{-1}(z_1)\right)-z_2\rtr\left(\bS_1\breve\bd_k^{-1}(z_2)\right)\right]\to z_1g_1(z_1)-z_2g_1(z_2)\quad{\rm a.s.}
\end{align}
and
\begin{align*}
  \beta_{jk}(z)=\frac1{1+s_jg_1(z)}+o_{\rm a.s.}(1).
\end{align*}
On the other hand,
\begin{align*}
 &\frac1n\re_k\left[z_1\rtr\left(\bS_1\bd_k^{-1}(z_1)\right)-z_2\rtr\left(\bS_1\breve\bd_k^{-1}(z_2)\right)\right]\\
 =&\frac1n\re_k\rtr\left[\bS_1\bd_k^{-1}(z_1)\left(z_1\breve\bd_k(z_2)-z_2\bd_k(z_1)\right)\breve\bd_k^{-1}(z_2)\right]\\
 =&\frac1{n^2}\re_k\rtr\left[\bS_1\bd_k^{-1}(z_1)\left((z_1-z_2)\sum_{j=1}^{k-1}s_j\bq_j\bq_j^*+z_1\sum_{j=k+1}^ns_j\breve\bq_j\breve\bq_j^*
 -z_2\sum_{j=k+1}^ns_j\bq_j\bq_j^*\right)\breve\bd_k^{-1}(z_2)\right]\\
 =&\frac{(z_1-z_2)}{n^2}\sum_{j=1}^{k-1}s_j\re_k\left(\beta_{jk}(z_1)\breve\beta_{jk}(z_2)\bq_j^*\breve\bd_{jk}^{-1}(z_2)\bS_1\bd_{jk}^{-1}(z_1)\bq_j\right)\\
 &+\frac1{n^2}\sum_{j=k+1}^ns_j\re_k\left(z_1\breve\beta_{jk}(z_2)\breve\bq_j^*\breve\bd_{jk}^{-1}(z_2)\bS_1\bd_k^{-1}(z_1)\breve\bq_j
 -z_2\beta_{jk}(z_1)\bq_j^*\breve\bd_k^{-1}(z_2)\bS_1\bd_{jk}^{-1}(z_1)\bq_j\right)\\
  =&\frac{(z_1-z_2)}{n^2}\sum_{j=1}^{k-1}\frac{s_j\re_k\rtr\left(\breve\bd_{jk}^{-1}(z_2)\bS_1\bd_{jk}^{-1}(z_1)\bS_1\right)}{(1+s_jg_1(z_1))(1+s_jg_1(z_2))}\\
 &+\frac1{n^2}\sum_{j=k+1}^n\re_k\left(\frac{z_1s_j\rtr\left(\breve\bd_{jk}^{-1}(z_2)\bS_1\bd_k^{-1}(z_1)\bS_1\right)}{1+s_jg_1(z_2)}
 -\frac{z_2s_j\rtr\left(\breve\bd_k^{-1}(z_2)\bS_1\bd_{jk}^{-1}(z_1)\bS_1\right)}{1+s_jg_1(z_1)}\right)+o_{\rm a.s.}(1)\\
   =&\left[\frac{(z_1-z_2)}{n^2}\sum_{j=1}^{k-1}\frac{s_j}{(1+s_jg_1(z_1))(1+s_jg_1(z_2))}+\frac1{n^2}\sum_{j=k+1}^n\left(\frac{z_1s_j}{1+s_jg_1(z_2)}
 -\frac{z_2s_j}{1+s_jg_1(z_1)}\right)\right]\\
 &\times\re_k\rtr\left(\breve\bd_{k}^{-1}(z_2)\bS_1\bd_k^{-1}(z_1)\bS_1\right)+o_{\rm a.s.}(1)
\end{align*}
where $\breve\beta_{jk}(z)=\frac1{1+n^{-1}s_j\bq_j^*\breve\bd_{jk}^{-1}(z)\bq_j}$ when $j<k$ and $\breve\beta_{jk}(z)=\frac1{1+n^{-1}s_j\breve\bq_j^*\breve\bd_{jk}^{-1}(z)\breve\bq_j}$ when $j>k$ . From (\ref{cl2}), it is obvious that
\begin{align}\label{cl4}
 \frac1n\sum_{k=1}^n\frac{s_k}{1+s_kg_1(z)}\to-zg_2(z)
\end{align}
and
\begin{align*}
 \frac1n\sum_{k=1}^n\frac{s_k^2}{\left(1+s_kg_1(z_1)\right)\left(1+s_kg_1(z_2)\right)}\to\frac{z_1g_2(z_1)-z_2g_2(z_2)}{g_1(z_1)-g_1(z_2)}.
\end{align*}
Similar to the analysis of Page 20 in \cite{liclt}, we get
\begin{align*}
 &\frac1n\re_k\left[z_1\rtr\left(\bS_1\bd_k^{-1}(z_1)\right)-z_2\rtr\left(\bS_1\breve\bd_k^{-1}(z_2)\right)\right]\\
   =&\left[\frac{(k-1)(z_1-z_2)}{n^2}\left(z_2g_2(z_2)-g_1(z_1)\frac{z_1g_2(z_1)-z_2g_2(z_2)}{g_1(z_1)-g_1(z_2)}\right)+\frac{n-k}{n^2}z_1z_2\left(
 g_2(z_1)-g_2(z_2)\right)\right]\\
 &\times\re_k\rtr\left(\breve\bd_{k}^{-1}(z_2)\bS_1\bd_k^{-1}(z_1)\bS_1\right)+o_{\rm a.s.}(1).
\end{align*}
Together with (\ref{cl1}), one has
\begin{align*}
 &\frac1n\re_k\left(\breve\bd_{k}^{-1}(z_2)\bS_1\bd_k^{-1}(z_1)\bS_1\right)\\
   =&\frac{z_1g_1(z_1)-z_2g_1(z_2)}{\left[\frac{(k-1)(z_1-z_2)}{n}\left(z_2g_2(z_2)-g_1(z_1)\frac{z_1g_2(z_1)-z_2g_2(z_2)}{g_1(z_1)-g_1(z_2)}\right)+\frac{n-k}{n}z_1z_2
   \left(g_2(z_1)
 -g_2(z_2)\right)\right]}+o_{\rm a.s.}(1)\\
    =&\frac{\left[z_1g_1(z_1)-z_2g_1(z_2)\right]/\left[z_1z_2\left(
 g_2(z_1)-g_2(z_2)\right)\right]}{\left[\frac{k(z_1-z_2)}{nz_1z_2\left(
 g_2(z_1)-g_2(z_2)\right)}\left(z_2g_2(z_2)-g_1(z_1)\frac{z_1g_2(z_1)-z_2g_2(z_2)}{g_1(z_1)-g_1(z_2)}\right)
    +\frac{n-k}{n}\right]}+o_{\rm a.s.}(1)\\
    =&\frac{\left[z_1g_1(z_1)-z_2g_1(z_2)\right]/\left[z_1z_2\left(
 g_2(z_1)- g_2(z_2)\right)\right]}{1-\frac knd(z_1,z_2)}+o_{\rm a.s.}(1).
\end{align*}
Consequently, we get the first probability
\begin{align*}
&n^{-2}\sum_{k=1}^n \frac{s_k^2 \re_k\left(\breve\bd_{k}^{-1}(z_2)\bS_1\bd_k^{-1}(z_1)\bS_1\right)}{(1+s_kg_1(z_1))(1+s_kg_1(z_2))}\to\int_0^{d(z_1,z_2)}\frac1{1-z}dz.
\end{align*}

Now, we are to compute the second term. It is also from Section \ref{see1} that
\begin{align*}
\frac1n\re_k\left[z_1\rtr\left(\bS_1\bd_k^{-1}(z_1)\right)-z_2\rtr\left(\bS_1\left(\breve\bd_k'(z_2)\right)^{-1}\right)\right]\to z_1g_1(z_1)-z_2g_1(z_2)\quad{\rm a.s.}
\end{align*}
Rewrite $\frac1n\re_k\left[z_1\rtr\left(\bS_1\bd_k^{-1}(z_1)\right)-z_2\rtr\left(\bS_1\left(\breve\bd_k'(z_2)\right)^{-1}\right)\right]$ as
\begin{align*}
 &\frac1n\re_k\left[z_1\rtr\left(\bS_1\bd_k^{-1}(z_1)\right)-z_2\rtr\left(\bS_1\left(\breve\bd_k'(z_2)\right)^{-1}\right)\right]\\
 =&\frac1n\re_k\rtr\left[\bS_1\bd_k^{-1}(z_1)\left(z_1\breve\bd_k'(z_2)-z_2\bd_k(z_1)\right)\left(\breve\bd_k'(z_2)\right)^{-1}\right]\\
 =&\frac1{n^2}\re_k\rtr\left[\bS_1\bd_k^{-1}(z_1)\left(\sum_{j=1}^{k-1}s_j\left(z_1\bar\bq_j\bq_j'-z_2\bq_j\bq_j^*\right)+\sum_{j=k+1}^ns_j\left(z_1\bar{\breve\bq}_j\breve\bq_j'
 -z_2\bq_j\bq_j^*\right)\right)\left(\breve\bd_k'(z_2)\right)^{-1}\right]\\
 =&\sum_{j=1}^{k-1}\frac{ s_j}{n^2}\re_k\left(z_1\breve\beta_{jk}'(z_2)\bq_j'\left(\breve\bd_{jk}'(z_2)\right)^{-1}\bS_1\bd_{k}^{-1}(z_1)\bar\bq_j
 -z_2\beta_{jk}(z_1)\bq_j^*\left(\breve\bd_k'(z_2)\right)^{-1}\bS_1\bd_{jk}^{-1}(z_1)\bq_j\right)\\
 &+\sum_{j=k+1}^n\frac{s_j}{n^2}\re_k\left(z_1\breve\beta_{jk}'(z_2)\breve\bq_j'\left(\breve\bd_{jk}'(z_2)\right)^{-1}\bS_1\bd_k^{-1}(z_1)\bar{\breve\bq}_j
 -z_2\beta_{jk}(z_1)\bq_j^*\left(\breve\bd_k'(z_2)\right)^{-1}\bS_1\bd_{jk}^{-1}(z_1)\bq_j\right)\\
=&\sum_{j=1}^{k-1}\frac{ s_j}{n^2}\re_k\left(z_1\breve\beta_{jk}'(z_2)\bq_j'\left(\breve\bd_{jk}'(z_2)\right)^{-1}\bS_1\bd_{jk}^{-1}(z_1)\bar\bq_j
 -z_2\beta_{jk}(z_1)\bq_j^*\left(\breve\bd_{jk}'(z_2)\right)^{-1}\bS_1\bd_{jk}^{-1}(z_1)\bq_j\right)\\
 &-\sum_{j=1}^{k-1}\frac{ s_j^2}{n^3}\re_k\bigg(z_1\beta_{jk}(z_1)\breve\beta_{jk}'(z_2)\bq_j'\left(\breve\bd_{jk}'(z_2)\right)^{-1}\bS_1\bd_{jk}^{-1}(z_1)\bq_j\bq_j^*\bd_{jk}^{-1}(z_1)\bar\bq_j\\
 &-z_2\beta_{jk}(z_1)\breve\beta_{jk}'(z_2)\bq_j^*\left(\breve\bd_{jk}'(z_2)\right)^{-1}\bar\bq_j\bq_j'\left(\breve\bd_{jk}'(z_2)\right)^{-1}\bS_1\bd_{jk}^{-1}(z_1)\bq_j\bigg)\\
 &+\sum_{j=k+1}^n\frac{s_j}{n^2}\re_k\left(z_1\breve\beta_{jk}'(z_2)\breve\bq_j'\left(\breve\bd_{jk}'(z_2)\right)^{-1}\bS_1\bd_k^{-1}(z_1)\bar{\breve\bq}_j
 -z_2\beta_{jk}(z_1)\bq_j^*\left(\breve\bd_k'(z_2)\right)^{-1}\bS_1\bd_{jk}^{-1}(z_1)\bq_j\right)\\
 =&\sum_{j\neq k}\frac{ s_j}{n^2}\left(\frac{z_1}{1+s_jg_1(z_2)}
 -\frac{z_2}{1+s_jg_1(z_1)}\right)\re_k\rtr\left(\bS_1\bd_k^{-1}(z_1)\bS_1\left(\breve\bd_k'(z_2)\right)^{-1}\right)\\
  &-\alpha_x\sum_{j=1}^{k-1}\frac{ s_j^2\left(z_1g_1(z_1)-z_2g_1(z_2)\right)}{n^2({1+s_jg_1(z_1)})({1+s_jg_1(z_2)})}\re_k\rtr\left(\bS_1\bd_k^{-1}(z_1)\bS_1\left(\breve\bd_k'(z_2)\right)^{-1}\right)+o_{\rm a.s.}(1)\\
   =&\left[\frac{ z_1z_2\left({g_2(z_1)-g_2(z_2)}\right)}{n}-\alpha_x\frac{ (k-1)\left(z_1g_2(z_1)-z_2g_2(z_2)\right)\left(z_1g_1(z_1)-z_2g_1(z_2)\right)}{n^2({g_1(z_1)-g_1(z_2)})}\right]\\
  &\times\re_k\rtr\left(\bS_1\bd_k^{-1}(z_1)\bS_1\left(\breve\bd_k'(z_2)\right)^{-1}\right)+o_{\rm a.s.}(1)
\end{align*}
where $\breve\beta_{jk}'(z)=\frac1{1+n^{-1}s_j\bq_j'\left(\breve\bd_{jk}(z)\right)^{-1}\bar\bq_j}$ when $j<k$ and $\breve\beta_{jk}'(z)=\frac1{1+n^{-1}s_j\breve\bq_j'\left(\breve\bd_{jk}(z)\right)^{-1}\bar{\breve\bq}_j}$ when $j>k$. This implies that
\begin{align*}
 \frac1n\rtr\re_k\left(\bS_1\bd_k^{-1}(z_1)\bS_1\left(\breve\bd_k'(z_2)\right)^{-1}\right)=\frac{ \left(z_1g_1(z_1)-z_2g_1(z_2)\right)/\left(z_1z_2({g_2(z_1)-g_2(z_2)})\right)}{1-\frac{k-1}{n}\widetilde d(z_1,z_2)}+o_{\rm a.s.}(1).
\end{align*}
Hence,
\begin{align*}
&\alpha_x n^{-2}\sum_{k=1}^n \frac{s_k^2\rtr\re_k\left(\bS_1\bd_k^{-1}(z_1)\bS_1\left(\breve\bd_k'(z_2)\right)^{-1}\right)}{(1+s_kg_1(z_1))(1+s_kg_1(z_2))}\to\int_0^{\widetilde d({z_1,z_2})}\frac1{1-z}dz.
\end{align*}

At last, we compute the limit of the third term. Applying the formula
\begin{align*}
  \left(\ba+{\bf U}\bb{\bf V}\right)^{-1}=\ba^{-1}-\ba^{-1}{\bf U}\left(\bi+\bb{\bf V}\ba^{-1}{\bf U}\right)^{-1}\bb{\bf V}\ba^{-1}
\end{align*}
and $\bt_{1n}^*\bt_{1n}$ is diagonal, it is obvious that
\begin{align}\label{cal7}
 \bt_{1n}^*\left(\bi_p+g_2(z_1)\bS_1\right)^{-1}\bt_{1n}
\end{align}
is also diagonal. By (\ref{beq15}), (\ref{cl3}), and (\ref{cl4}), it follows that
\begin{align*}
&\frac{\kappa_x}{n^{2}}\sum_{k=1}^n \frac{s_k^2\sum_{j\neq k=1}^n\re_k\be_j'\bt_{1n}^*\bd_{k}^{-1}(z_1)\bt_{1n}\be_j\be_j'\bt_{1n}^*\breve\bd_{k}^{-1}(z_2)\bt_{1n}\be_j }{(1+s_kg_1(z_1))(1+s_kg_1(z_2))}\\
=&\frac{\kappa_x}{n^{2}}\sum_{k=1}^n \frac{s_k^2\sum_{j\neq k=1}^n\re_k\be_j'\bt_{1n}^*\br_{k}^{-1}(z_1)\bt_{1n}\be_j\be_j'\bt_{1n}^*\breve\br_{k}^{-1}(z_2)\bt_{1n}\be_j }{(1+s_kg_1(z_1))(1+s_kg_1(z_2))}+o_{\rm a.s.}(1)\\
=&\frac{\kappa_x}{n^{2}}\sum_{k=1}^n \frac{s_k^2\sum_{j\neq k=1}^n\re_k\be_j'\bt_{1n}^*\left(\bi_p+g_2(z_1)\bS_1\right)^{-1}\bt_{1n}\be_j\be_j'\bt_{1n}^*\left(\bi_p+g_2(z_2)\bS_1\right)^{-1}\bt_{1n}\be_j }{z_1z_2(1+s_kg_1(z_1))(1+s_kg_1(z_2))}+o_{\rm a.s.}(1)\\
=&\frac{\kappa_x}{n^{2}}\sum_{k=1}^n \frac{s_k^2\re_k\rtr\left[\left(\bi_p+g_2(z_1)\bS_1\right)^{-1}\bS_1\left(\bi_p+g_2(z_2)\bS_1\right)^{-1}\bS_1\right] }{z_1z_2(1+s_kg_1(z_1))(1+s_kg_1(z_2))}+o_{\rm a.s.}(1)\\
=&\frac{\kappa_x(z_1g_2(z_1)-z_2g_2(z_2))}{z_1z_2(g_1(z_1)-g_1(z_2))} \frac1{n}{\rtr\left[\left(\bi_p+g_2(z_1)\bS_1\right)^{-1}\bS_1\left(\bi_p+g_2(z_2)\bS_1\right)^{-1}\bS_1\right] }+o_{\rm a.s.}(1)\\
\to&\frac{c\kappa_x(z_1g_2(z_1)-z_2g_2(z_2))}{z_1z_2(g_1(z_1)-g_1(z_2))} \int\frac{x^2}{(1+g_2(z_1)x)(1+g_2(z_2)x)}dH_1(x)=\kappa_xd(z_1,z_2).
\end{align*}

Therefore, we conclude that $M_{n1}(z)$ converges in distribution to a Gaussian random variable $M_1(z)$ with zero mean and
\begin{align*}
  {\rm Cov}\bigg(M_1(z_1),&M_1(z_2)\bigg)=\frac{\partial^2}{\partial z_2\partial z_1}\Bigg\{\int_0^{d(z_1,z_2)}\frac{1}{1-z}dz+\int_0^{\alpha_xd(z_1,z_2)}\frac{1}{1-z}dz\notag+\kappa_xd(z_1,z_2)\Bigg\}.
\end{align*}

\subsection{Tightness of $M_{n1}(z)$}

Similar to Section 3.2 of \cite{liclt}, one can show  that
\begin{align*}
\sup_{n;z_1,z_2\in\mathcal{C}_n}\frac{\re\left|M_{n1}(z_1)-M_{n1}(z_2)\right|^2}{|z_1-z_2|^2}
\end{align*}
is finite.

\subsection{Convergence of $M_{n2}(z)$}

Denote $\underline{\bs}_n=\frac1n{\bx}_{n}^*\bt_{1n}^*\bt_{1n}{\bx}_{n}{\bf T}_{2n}$. Since $m_2$ is replaced by $n$, we have the following relationship between the empirical distributions of $\bs_n$ and $\underline{\bs}_n$
\begin{align*}
F^{\underline{\bs}_n}(x)=c_nF^{\bs_n}(x)+(1-c_n)I_{[0,\infty)}(x),
\end{align*}
and hence
\begin{align}\label{qal1}
\underline m_n(z)=c_n m_n(z)+z^{-1}(c_n-1).
\end{align}
where $c_n=p/n$ and $\underline m_n(z)=m_{F^{\underline{\bs}_n}}(z)$. Denote by $\underline F$ the limiting distribution of $F^{\underline{\bs}_n}$. Then $F$ and $\underline F$ must satisfy
\begin{align*}
\underline F(x)=c F(x)+(1-c)I_{[0,\infty)}(x),
\end{align*}
and
\begin{align}\label{qal3}
\underline m(z)=c m(z)-z^{-1}(1-c)
\end{align}
where $\underline m(z)=m_{\underline F}(z)$.

{Let ${\bf W}(z)=\frac1n\sum_{j=1}^ns_j\varphi_j(z)\bS_1-z\bi_p$ and
\begin{align*}
&\varphi_j(z)=\frac1{1+s_j\re g_n(z)},\quad b_j(z)=\frac1{1+s_j\re\rtr\left(\bd_j^{-1}(z)\bS_1\right)}\\
&\rho_j(z)=\bq_j^*\bd_j^{-1}(z)\bq_j-\re\rtr\left(\bd_j^{-1}(z)\bS_1\right),\quad g_n(z)=\frac1n\rtr\left(\bd^{-1}(z)\bS_1\right).
\end{align*}
It can be verified that $\left\|{\bf W}^{-1}(z)\right\|$ is uniformly bounded on $\mathcal{C}_n$ (see \cite{liclt}). From \cite{liclt}, we also know that
\begin{align}\label{bal1}
   \re\|\bd_j^{-1}(z)\|^l\le C_l, \quad\re|\beta_j(z)|^l\le C_l,l\ge1, \quad |b_j(z)|\le C.
\end{align}
Let $\ba$ be $p\times p$ matrix whose spectral norm are uniformly bounded for $z\in\mathcal{C}_n$. Applying Lemma \ref{lep1} and (\ref{bal1}), we have for $l\ge4$
\begin{align*}
  &\re\left|\bq_j^*\bd_j^{-1}(z)\ba\bq_j-\re\rtr\left(\bd_j^{-1}(z)\ba\bS_1\right)\right|^l\\
  \le& C_l\re\left|\bq_j^*\bd_j^{-1}(z)\ba\bq_j-\rtr\left(\bd_j^{-1}(z)\ba\bS_1\right)\right|^l
  +C_l\re\left|\rtr\left(\bd_j^{-1}(z)\ba\bS_1\right)-\re\rtr\left(\bd_j^{-1}(z)\ba\bS_1\right)\right|^l\\
    \le& C_l\re\left[\rtr\left(\bd_j^{-1}(z)\ba\bS_1\ba^*\bd_j^{-1}(\bar z)\bS_1\right)\right]^{l/2}+C_l\eta_n^{2l-6}n^{l-3}\re\rtr\left(\bd_j^{-1}(z)\ba\bS_1\ba^*\bd_j^{-1}(\bar z)\bS_1\right)^{l/2}\\
  &+C_l\re\left|\sum_{k\neq j=1}^n\left(\re_k-\re_{k-1}\right)\rtr\left(\bd_{j}^{-1}(z)\ba\bS_1\right)\right|^l\\
      \le& C_ln^{l/2}+C_l\eta_n^{2l-6}n^{l-2}+C_l\re\left(\sum_{k\neq j=1}^n\left|\rtr\left(\bd_{j}^{-1}(z)\ba\bS_1\right)-\rtr\left(\bd_{kj}^{-1}(z)\ba\bS_1\right)\right|^2\right)^{l/2}\\
            \le& C_ln^{l/2}+C_l\eta_n^{2l-6}n^{l-2}+C_l\re\left(\frac1{n^2}\sum_{k\neq j=1}^ns_j^2|\beta_{kj}^2(z)|\left|\bq_j^*\bd_{kj}^{-1}(z)\ba\bS_1\bd_{kj}^{-1}(z)\bq_j\right|^2\right)^{l/2}\\
\le& C_ln^{l/2}+C_l\eta_n^{2l-6}n^{l-2}+\frac{C_ln^{l/2-1}}{n^l}\sum_{k\neq j=1}^n s_j^l\re\left(|\beta_{kj}^l(z)|\left|\bq_j^*\bd_{kj}^{-1}(z)\ba\bS_1\bd_{kj}^{-1}(z)\bq_j\right|^l\right)\\
\le& C_ln^{l/2}+C_l\eta_n^{2l-6}n^{l-2}+\frac{C_ln^{l/2-1}}{n^l}\sum_{k\neq j=1}^n\re^{1/2}\left|\bq_j^*\bd_{kj}^{-1}(z)\ba\bS_1\bd_{kj}^{-1}(z)\bq_j\right|^{2l}\\
\le& C_ln^{l/2}+C_l\eta_n^{2l-6}n^{l-2}+\frac{C_ln^{l/2-1}}{n^l}\sum_{k\neq j=1}^n\re^{1/2}\left|\rtr\left(\bd_{kj}^{-1}(z)\ba\bS_1\bd_{kj}^{-1}(z)\bS_1\right)\right|^{2l}\\
&+\frac{C_ln^{l/2-1}}{n^l}\sum_{k\neq j=1}^n\re^{1/2}\left|\bq_j^*\bd_{kj}^{-1}(z)\ba\bS_1\bd_{kj}^{-1}(z)\bq_j-\rtr\left(\bd_{kj}^{-1}(z)\ba\bS_1\bd_{kj}^{-1}(z)\bS_1\right)\right|^{2l}\\
\le& C_l\eta_n^{l-4}n^{l-2}.
\end{align*}
This implies that for $l\ge4$
\begin{align}\label{cal1}
 \re\left|\rho_j(z)\right|^l\le& C_l\eta_n^{l-4}n^{l-2}
\end{align}
and
\begin{align}\label{cal2}
\re\left|\bq_j^*\bd_j^{-1}(z)\bw^{-1}(z)\bq_j-\re\rtr\left(\bd_j^{-1}(z)\bw^{-1}(z)\bS_1\right)\right|^l
\le& C_l\eta_n^{l-4}n^{l-2}.
\end{align}
It can be obtained from\cite{liclt} that
\begin{align*}
  \sup_{z\in\mathcal{C}_n}\left|\re g_n(z)-g_{1}(z)\right|\to0\quad{\rm as} \ n\to\infty
\end{align*}
which implies that
\begin{align}\label{cal5}
\sup_{z\in\mathcal{C}_n}\left|\varphi_k(z)-\frac1{1+s_kg_{1n}^0(z)}\right|=o(1).
\end{align}

Write
$$\bd(z)-{\bf W}\left(z\right)=\frac1n\sum_{j=1}^ns_j\bq_j\bq_j^*-\frac1n\sum_{j=1}^ns_j\varphi_j(z)\bS_1.$$ Taking inverses and then expected value, we have
\begin{align}\label{sal:6}
&\bw^{-1}(z)-\re\bd^{-1}(z)\\
=&\bw^{-1}(z)\re\left[\frac1n\sum_{j=1}^ns_j\bq_j\bq_j^*\bd^{-1}(z)-\frac1n\sum_{j=1}^ns_j\varphi_j(z)\bS_1\bd^{-1}(z)\right]\notag\\
=&\bw^{-1}(z)\re\left[\frac1n\sum_{j=1}^ns_j\beta_j(z)\bq_j\bq_j^*\bd_j^{-1}(z)-\frac1n\sum_{j=1}^ns_j\varphi_j(z)\bS_1\bd^{-1}(z)\right]\notag.
\end{align}
Taking the trace on both sides and dividing by $-1$, one obtains
\begin{align}\label{beq20}
d_{n1}(z)
=&-\frac1n\sum_{j=1}^ns_j\re\beta_j(z)\left(\bq_j^*\bd_j^{-1}(z)\bw^{-1}(z)\bq_j-\re\rtr\left(\bw^{-1}(z)\bS_1\bd_j^{-1}(z)\right)\right)\\
&-\frac1n\sum_{j=1}^ns_j\re\beta_j(z)\left(\re\rtr\left(\bw^{-1}(z)\bS_1\bd_j^{-1}(z)\right)
-\re\rtr\left(\bw^{-1}(z)\bS_1\bd^{-1}(z)\right)\right)\notag\\
&-\frac1{n}\sum_{j=1}^ns_j\re\left(\beta_j(z)-\psi_j(z)\right)\re\left(\rtr\left(\bw^{-1}(z)\bS_1\bd^{-1}(z)\right)\right)\notag\\
\triangleq&\mathcal{J}_1+\mathcal{J}_2+\mathcal{J}_3\notag,
\end{align}
where $d_{n1}(z)=p\left[\re m_n(z)-\int\frac1{\frac1n\sum_{j=1}^ns_j\varphi_j(z)x -z}dH_{1n}(x)\right]$. Using
\begin{align}
  \beta_j(z)=&b_j(z)-\frac1ns_j\beta_j(z)b_j(z)\rho_j(z)\label{cal3}\\
  =&b_j(z)-\frac1ns_jb_j^2(z)\rho_j(z)+\frac1{n^2}s_j^2\beta_j(z)b_j^2(z)\rho_j^2(z),\notag
\end{align}
it follows that
\begin{align*}
\mathcal{J}_1=&\frac1{n^2}\sum_{j=1}^ns_j^2b_j^2(z)\re\rho_j(z)\left(\bq_j^*\bd_j^{-1}(z)\bw^{-1}(z)\bq_j-\re\rtr\left(\bw^{-1}(z)\bS_1\bd_j^{-1}(z)\right)\right)\\
&-\frac1{n^3}\sum_{j=1}^ns_j^3 b_j^2(z)\re\beta_j(z)\rho_j^2(z)\left(\bq_j^*\bd_j^{-1}(z)\bw^{-1}(z)\bq_j-\re\rtr\left(\bw^{-1}(z)\bS_1\bd_j^{-1}(z)\right)\right)\\
=&\frac1{n^2}\sum_{j=1}^ns_j^2b_j^2(z)\re\varepsilon_j(z)\left(\bq_j^*\bd_j^{-1}(z)\bw^{-1}(z)\bq_j-\rtr\left(\bw^{-1}(z)\bS_1\bd_j^{-1}(z)\right)\right)\\
&+\frac1{n^2}\sum_{j=1}^ns_j^2b_j^2(z){\rm Cov}\left(\rtr\left(\bd_j^{-1}(z)\bS_1\right),\rtr\left(\bw^{-1}(z)\bS_1\bd_j^{-1}(z)\right)\right)\\
&-\frac1{n^3}\sum_{j=1}^ns_j^3 b_j^2(z)\re\beta_j(z)\rho_j^2(z)\left(\bq_j^*\bd_j^{-1}(z)\bw^{-1}(z)\bq_j-\re\rtr\left(\bw^{-1}(z)\bS_1\bd_j^{-1}(z)\right)\right)\\
\triangleq&\mathcal{J}_{11}+\mathcal{J}_{12}+\mathcal{J}_{13}.
\end{align*}
Due to Lemma \ref{lep6} and Cauchy-Schwarz inequality, one finds
\begin{align*}
\left|\mathcal{J}_{12}\right|\le\frac C{n^2}\sum_{j=1}^n{\rm Var}^{1/2}\left(\rtr\left(\bd_j^{-1}(z)\bS_1\right)\right){\rm Var}^{1/2}\left(\rtr\left(\bw^{-1}(z)\bS_1\bd_j^{-1}(z)\right)\right)\le\frac C{n}.
\end{align*}
Applying (\ref{bal1}), (\ref{cal1}), and (\ref{cal2}), we have
\begin{align*}
\left|\mathcal{J}_{13}\right|\le&\frac C{n^3}\sum_{j=1}^n\re^{1/4}
\left|\bq_j^*\bd_j^{-1}(z)\bw^{-1}(z)\bq_j-\re\rtr\left(\bw^{-1}(z)\bS_1\bd_j^{-1}(z)\right)\right|^4\\
&\times\re^{1/2}\left|\beta_j^2(z)\right|\re^{1/4}\left|\rho_j^8(z)\right|\le C\eta_n.
\end{align*}
Hence,
\begin{align}\label{cal4}
\mathcal{J}_1=&\sum_{j=1}^n\frac{s_j^2b_j^2(z)}{n^2}\re\varepsilon_j(z)\left(\bq_j^*\bd_j^{-1}(z)\bw^{-1}(z)\bq_j-\rtr\left(\bw^{-1}(z)\bS_1\bd_j^{-1}(z)\right)\right)+o(1).
\end{align}
Note that from (\ref{cal3})
\begin{align*}
  b_j(z)-\varphi_j(z)=&-\frac1{n^2}s_j^2b_j(z)\varphi_j(z)\re\left(\beta_j(z)\bq_j^*\bd_j^{-1}(z)\bS_1\bd_j^{-1}(z)\bq_j\right)\\
  =&-\frac1{n^2}s_j^2b_j^2(z)\varphi_j(z)\re\left(\bq_j^*\bd_j^{-1}(z)\bS_1\bd_j^{-1}(z)\bq_j\right)\\
  &+\frac1{n^3}s_j^3b_j^2(z)\varphi_j(z)\re\left(\beta_j(z)\rho_j(z)\bq_j^*\bd_j^{-1}(z)\bS_1\bd_j^{-1}(z)\bq_j\right).
\end{align*}
From (\ref{cal1}) and Lemma \ref{lep1}, we have
\begin{align*}
&\left|\frac1{n^3}s_j^3b_j^2(z)\varphi_j(z)\re\left(\beta_j(z)\rho_j(z)\bq_j^*\bd_j^{-1}(z)\bS_1\bd_j^{-1}(z)\bq_j\right)\right|\\
\le&\frac C{n^3}\re^{1/4}|\beta_j(z)|^4\re^{1/4}|\rho_j(z)|^4\re^{1/2}\left|\bq_j^*\bd_j^{-1}(z)\bS_1\bd_j^{-1}(z)\bq_j\right|^2\le \frac C{n^{3/2}}.
\end{align*}
This yields that
\begin{align}\label{bal15}
b_j(z)-\varphi_j(z)
=&-\frac1{n^2}s_j^2b_j^2(z)\varphi_j(z)\re\rtr\left(\bd_j^{-1}(z)\bS_1\bd_j^{-1}(z)\bS_1\right)+o(n^{-1}).
\end{align}
Therefore,
\begin{align}\label{bal:9}
|b_j(z)-\varphi_j(z)|\le &\frac C{n}\re\left\|\bd_j^{-1}(z)\bS_1\bd_j^{-1}(z)\bS_1\right\|+o(n^{-1})
\le \frac C{n}.
\end{align}
Together with (\ref{cal4}), we conclude that
\begin{align*}
\mathcal{J}_1=&\frac1{n^2}\sum_{j=1}^ns_j^2\varphi_j^2(z)\re\varepsilon_j(z)\left(\bq_j^*\bd_j^{-1}(z)\bw^{-1}(z)\bq_j-\rtr\left(\bw^{-1}(z)\bS_1\bd_j^{-1}(z)\right)\right)+o(1).
\end{align*}

By (\ref{cal3}), it is obvious that
\begin{align*}
\mathcal{J}_2=&-\frac1{n^2}\sum_{j=1}^ns_j^2\re\beta_j(z)\re\left(\beta_j(z)\bq_j^*\bd_j^{-1}(z)\bw^{-1}(z)\bS_1\bd_j^{-1}(z)\bq_j\right)\\
=&-\frac1{n^2}\sum_{j=1}^ns_j^2b_j(z)\re\left(\beta_j(z)\bq_j^*\bd_j^{-1}(z)\bw^{-1}(z)\bS_1\bd_j^{-1}(z)\bq_j\right)\\
&+\frac1{n^3}\sum_{j=1}^ns_j^3b_j(z)\re\left(\beta_j(z)\rho_j(z)\right)\re\left(\beta_j(z)\bq_j^*\bd_j^{-1}(z)\bw^{-1}(z)\bS_1\bd_j^{-1}(z)\bq_j\right)\\
=&-\frac1{n^2}\sum_{j=1}^ns_j^2b_j^2(z)\re\rtr\left(\bd_j^{-1}(z)\bw^{-1}(z)\bS_1\bd_j^{-1}(z)\bS_1\right)\\
&+\frac1{n^3}\sum_{j=1}^ns_j^3b_j^2(z)\re\left(\beta_j(z)\rho_j(z)\bq_j^*\bd_j^{-1}(z)\bw^{-1}(z)\bS_1\bd_j^{-1}(z)\bq_j\right)\\
&+\frac1{n^3}\sum_{j=1}^ns_j^3b_j(z)\re\left(\beta_j(z)\rho_j(z)\right)\re\left(\beta_j(z)\bq_j^*\bd_j^{-1}(z)\bw^{-1}(z)\bS_1\bd_j^{-1}(z)\bq_j\right)\\
\triangleq&\mathcal{J}_{21}+\mathcal{J}_{22}+\mathcal{J}_{23}.
\end{align*}
Using (\ref{bal1}), (\ref{cal3}), and Lemma \ref{lep1}, we get
\begin{align*}
  \left|\mathcal{J}_{22}\right|\le&\frac C{n^3}\sum_{j=1}^n\re^{1/4}\left|\beta_j^4(z)\right|\re^{1/4}\left|\rho_j^4(z)\right|\re^{1/2}\left|\bq_j^*\bd_j^{-1}(z)\bw^{-1}(z)\bS_1\bd_j^{-1}(z)\bq_j\right|^2\\
  \le&\frac C{n^{3/2}}\re^{1/2}\Bigg[\rtr\left(\bd_j^{-1}(z)\bw^{-1}(z)\bS_1\bd_j^{-1}(z)\bS_1\bd_j^{-1}(\bar z)\bS_1\bw^{-1}(\bar z)\bd_j^{-1}(\bar z)\bS_1\right)\\
  &\quad+\left|\rtr\left(\bd_j^{-1}(z)\bw^{-1}(z)\bS_1\bd_j^{-1}(z)\bS_1\right)\right|^2\Bigg]\le\frac C{\sqrt n}
\end{align*}
and
\begin{align*}
  \left|\mathcal{J}_{23}\right|\le&\frac C{n^3}\sum_{j=1}^n\re\left|\beta_j^2(z)\right|\re^{1/4}\left|\rho_j^4(z)\right|\re^{1/2}\left|\bq_j^*\bd_j^{-1}(z)\bw^{-1}(z)\bS_1\bd_j^{-1}(z)\bq_j\right|^2\le\frac C{\sqrt n}
\end{align*}
Combining with the above two inequalities and (\ref{bal:9}), it yields that
\begin{align*}
\mathcal{J}_2
=&-\frac1{n^2}\sum_{j=1}^ns_j^2\varphi_j^2(z)\re\rtr\left(\bd_j^{-1}(z)\bw^{-1}(z)\bS_1\bd_j^{-1}(z)\bS_1\right)+o(1).
\end{align*}

From (\ref{cal3}), it follows that
\begin{align*}
  \re\left(\beta_j(z)-b_j(z)\right)=&\frac1{n^2}s_j^2b_j^3(z)\re\rho_j^2(z)-\frac1{n^3}s_j^3 b_j^3(z)\re\left(\beta_j(z)\rho_j^3(z)\right),\\
=&\frac1{n^2}s_j^2b_j^3(z)\re\varepsilon_j^2(z)-\frac1{n^3}s_j^3 b_j^3(z)\re\left(\beta_j(z)\rho_j^3(z)\right),\\
  &+\frac1{n^2}s_j^2b_j^3(z)\re\left[\rtr\left(\bd_j^{-1}(z)\bS_1\right)-\re\rtr\left(\bd_j^{-1}(z)\bS_1\right)\right]^2\\
\triangleq&\mathcal{J}_{31}+\mathcal{J}_{32}+\mathcal{J}_{33}.
\end{align*}
Using (\ref{bal1}) and (\ref{cal1}), one gets
\begin{align*}
  \left|\mathcal{J}_{32}\right|\le\frac C{n^3}\re^{1/2}\left|\beta_j(z)\right|^2\re^{1/2}\left|\rho_j^6(z)\right|\le\frac{C\eta_n}{n}.
\end{align*}
It can be obtained from Lemma \ref{lep6} that
\begin{align*}
  \left|\mathcal{J}_{33}\right|\le\frac{C}{n^2}.
\end{align*}
Together with the above equality, (\ref{bal15}), and (\ref{bal:9}), we see that
\begin{align}\label{cal6}
  \re\left(\beta_j(z)-\varphi_j(z)\right)=&\frac1{n^2}s_j^2\varphi_j^3(z)\re\varepsilon_j^2(z)
  \\
  &-\frac1{n^2}s_j^2\varphi_j^3(z)\re\rtr\left(\bd_j^{-1}(z)\bS_1\bd_j^{-1}(z)\bS_1\right)+o(n^{-1}).\notag
\end{align}
Consequently, we have
\begin{align*}
\mathcal{J}_3=&-\frac1{n^3}\sum_{j=1}^ns_j^3\varphi_j^3(z)\re\varepsilon_j^2(z)\re\left(\rtr\left(\bw^{-1}(z)\bS_1\bd^{-1}(z)\right)\right)\\
&+\frac1{n^3}\sum_{j=1}^ns_j^3\varphi_j^3(z)\re\rtr\left(\bd_j^{-1}(z)\bS_1\bd_j^{-1}(z)\bS_1\right)\re\left(\rtr\left(\bw^{-1}(z)\bS_1\bd^{-1}(z)\right)\right)+o(1).
\end{align*}

Because of the above analysis and (\ref{cal5}), we know that
\begin{align*}
  &d_{n1}(z)=\frac1{n^2}\sum_{j=1}^ns_j^2\varphi_j^2(z)\re\varepsilon_j(z)\left(\bq_j^*\bd_j^{-1}(z)\bw^{-1}(z)\bq_j-\rtr\left(\bw^{-1}(z)\bS_1\bd_j^{-1}(z)\right)\right)\\
  &-\frac1{n^2}\sum_{j=1}^ns_j^2\varphi_j^2(z)\re\rtr\left(\bd_j^{-1}(z)\bw^{-1}(z)\bS_1\bd_j^{-1}(z)\bS_1\right)\\
  &-\frac1{n^3}\sum_{j=1}^ns_j^3\varphi_j^3(z)\re\varepsilon_j^2(z)\re\left(\rtr\left(\bw^{-1}(z)\bS_1\bd^{-1}(z)\right)\right)\\
&+\frac1{n^3}\sum_{j=1}^ns_j^3\varphi_j^3(z)\re\rtr\left(\bd_j^{-1}(z)\bS_1\bd_j^{-1}(z)\bS_1\right)\re\left(\rtr\left(\bw^{-1}(z)\bS_1\bd^{-1}(z)\right)\right)+o(1)\\
=&\frac{\alpha_x}{n^2}\sum_{j=1}^ns_j^2\varphi_j^2(z)\re\rtr\left(\bd_j^{-1}(z)\bS_1\bw^{-1}(z)\left(\bd_j'(z)\right)^{-1}\bS_1\right)
\\
  &+\frac{\kappa_x}{n^2}\sum_{j=1}^ns_j^2\varphi_j^2(z)\sum_{k\neq j=1}^n\re\left(\be_k'\bt_{1n}^*\bd_j^{-1}(z)\bt_{1n}\be_k\be_k'\bt_{1n}^*\bd_j^{-1}(z)\bw^{-1}(z)\bt_{1n}\be_k\right)\\
&-\frac{\alpha_x}{n^3}\sum_{j=1}^ns_j^3\varphi_j^3(z)\re\rtr\left(\bd_j^{-1}(z)\bS_1\left(\bd_j'(z)\right)^{-1}\bS_1\right)\re\left(\rtr\left(\bw^{-1}(z)\bS_1\bd^{-1}(z)\right)\right)\\
  &-\frac{\kappa_x}{n^3}\sum_{j=1}^ns_j^3\varphi_j^3(z)\sum_{k\neq j=1}^n\re\left(\be_k'\bt_{1n}^*\bd_j^{-1}(z)\bt_{1n}\be_k\right)^2\re\left(\rtr\left(\bw^{-1}(z)\bS_1\bd^{-1}(z)\right)\right)+o(1)\\
  =&\frac{\alpha_x}{n^2}\sum_{j=1}^n\frac{s_j^2}{(1+s_jg_{1n}^0(z))^2}\re\rtr\left(\bd^{-1}(z)\bS_1\bw^{-1}(z)\left(\bd'(z)\right)^{-1}\bS_1\right)
\\
  &+\frac{\kappa_x}{n^2}\sum_{j=1}^n\frac{s_j^2}{(1+s_jg_{1n}^0(z))^2}\sum_{k\neq j=1}^n\re\left(\be_k'\bt_{1n}^*\bd^{-1}(z)\bt_{1n}\be_k\be_k'\bt_{1n}^*\bd^{-1}(z)\bw^{-1}(z)\bt_{1n}\be_k\right)\\
&-\frac{\alpha_x}{n^3}\sum_{j=1}^n\frac{s_j^3}{(1+s_jg_{1n}^0(z))^3}\re\rtr\left(\bd^{-1}(z)\bS_1\left(\bd'(z)\right)^{-1}\bS_1\right)\re\left(\rtr\left(\bw^{-1}(z)\bS_1\bd^{-1}(z)\right)\right)\\
  &-\frac{\kappa_x}{n^3}\sum_{j=1}^n\frac{s_j^3}{(1+s_jg_{1n}^0(z))^3}\sum_{k\neq j=1}^n\re\left(\be_k'\bt_{1n}^*\bd^{-1}(z)\bt_{1n}\be_k\right)^2\re\left(\rtr\left(\bw^{-1}(z)\bS_1\bd^{-1}(z)\right)\right)+o(1).
\end{align*}

Write $M_{n2}(z)$ as
\begin{align*}
&p\left[\re m_n(z)-m_n^0(z)\right]\\
=&d_{n1}(z)+p\left[\int\frac1{\frac1n\sum_{j=1}^ns_j\varphi_j(z)x- z}dH_{1n}(x)+z^{-1}\int\frac1{1+g_{2n}^0(z)x}dH_{1n}(x)\right]\\
=&d_{n1}(z)-p\left(\re g_{n}(z)-g_{1n}^0(z)\right)\frac1n\sum_{j=1}^n\frac{s_j^2\varphi_j(z)}{1+g_{1n}^0(z)s_j}\\
&\times\int\frac{x}
{\left(\frac1n\sum_{j=1}^ns_j\varphi_j(z)x- z\right)\left({z+zg_{2n}^0(z)x}\right)}dH_{1n}(x).
\end{align*}
Below we first find the relation between $\left(\re m_n(z)-m_n^0(z)\right)$ and $\left(\re g_{n}(z)-g_{1n}^0(z)\right)$. Write $\bd(z)+z\bi_p=\frac1n\sum_{k=1}^ns_k\bq_k\bq_k^*$. Multiplying by $\bd^{-1}(z)$ on the right-hand side and using (\ref{bal2}),
 we obtain
\begin{align*}
\bi_p+z\bd^{-1}(z)=&\frac1n\sum_{k=1}^n s_k\beta_k(z)\bq_k\bq_k^*\bd_k^{-1}(z).
\end{align*}
Taking the trace on both side and dividing by $p$, one gets
\begin{align*}
1+zm_n(z)=c_n-\frac{c_n}n\sum_{k=1}^n\beta_k(z).
\end{align*}
Together with (\ref{qal1}) , we have
\begin{align}\label{bal:8}
\underline m_n(z)=-\frac1{zn}\sum_{k=1}^n\beta_k(z).
\end{align}
From (\ref{cal6}) and (\ref{bal:8}), it implies
\begin{align*}
 \re\underline m_n(z)=&-\frac1{zn}\sum_{k=1}^n\varphi_k(z)-\frac{\alpha_x}{zn^3}\sum_{k=1}^ns_k^2\varphi_k^3(z)\re\rtr\left(\bd_k^{-1}(z)\bS_1\left(\bd_k'(z)\right)^{-1}\bS_1\right)\\
 &-\frac{\kappa_x}{zn^3}\sum_{k=1}^ns_k^2\varphi_k^3(z)\sum_{j\neq k=1}^n\re\left(\be_j'\bt_{1n}^*\bd_k^{-1}(z)\bt_{1n}\be_j\right)^2+o(n^{-1})\\
 =&-\frac1{zn}\sum_{k=1}^n\varphi_k(z)-\frac{\alpha_x}{zn^3}\sum_{k=1}^ns_k^2\varphi_k^3(z)\re\rtr\left(\bd^{-1}(z)\bS_1\left(\bd'(z)\right)^{-1}\bS_1\right)\\
 &-\frac{\kappa_x}{zn^3}\sum_{k=1}^ns_k^2\varphi_k^3(z)\sum_{j\neq k=1}^n\re\left(\be_j'\bt_{1n}^*\bd^{-1}(z)\bt_{1n}\be_j\right)^2+o(n^{-1}).\\
\end{align*}
Hence, we see from (\ref{cal5})
\begin{align*}
\re\underline m_n(z)-\underline m_n^0(z)=&\left(\re g_{n}(z)-g_{1n}^0(z)\right)\frac1{zn}\sum_{k=1}^n\frac{s_k}{(1+g_{1n}^0(z)s_k)^2}\\
&-\frac{\alpha_x}{zn^3}\sum_{k=1}^n\frac{s_k^2}{(1+g_{1n}^0(z)s_k)^3}\re\rtr\left(\bd^{-1}(z)\bS_1\left(\bd'(z)\right)^{-1}\bS_1\right)\\
 &-\frac{\kappa_x}{zn^3}\sum_{k=1}^n\frac{s_k^2}{(1+g_{1n}^0(z)s_k)^3}\sum_{j\neq k=1}^n\re\left(\be_j'\bt_{1n}^*\bd^{-1}(z)\bt_{1n}\be_j\right)^2+o(n^{-1})
\end{align*}
which yields that
\begin{align*}
\re g_{n}(z)-g_{1n}^0(z)=&\left[\frac1{zn}\sum_{k=1}^n\frac{s_k}{(1+g_{1n}^0(z)s_k)^2}\right]^{-1}\Bigg[\re\underline m_n(z)-\underline m_n^0(z)\\
&+\frac{\alpha_x}{zn^3}\sum_{k=1}^n\frac{s_k^2}{(1+g_{1n}^0(z)s_k)^3}\re\rtr\left(\bd^{-1}(z)\bS_1\left(\bd'(z)\right)^{-1}\bS_1\right)\\
&+\frac{\kappa_x}{zn^3}\sum_{k=1}^n\frac{s_k^2}{(1+g_{1n}^0(z)s_k)^3}\sum_{j\neq k=1}^n\re\left(\be_j'\bt_{1n}^*\bd^{-1}(z)\bt_{1n}\be_j\right)^2\Bigg]+o(n^{-1}).
\end{align*}
Combining the above equalities with $M_{n2}(z)=p\left[\re m_n(z)-m_n^0(z)\right]=n\left[\re\underline m_n(z)-\underline m_n^0(z)\right]$, we conclude that from (\ref{cal5})
\begin{align}\label{cal8}
&p\left[\re m_n(z)-m_n^0(z)\right]\\
=&d_{n1}(z)-c_nzn\sum_{k=1}^n\frac{s_k^2\varphi_j(z)}{1+g_{1n}^0(z)s_k}\int\frac{x}
{\left(\frac1n\sum_{k=1}^ns_k\varphi_k(z)x- z\right)\left({z+zg_{2n}^0(z)x}\right)}dH_{1n}(x)\notag\\
&\times\left[\sum_{k=1}^n\frac{s_k}{(1+g_{1n}^0(z)s_k)^2}\right]^{-1}\Bigg[\re\underline m_n(z)-\underline m_n^0(z)\notag\\
&+\frac{\alpha_x}{zn^3}\sum_{k=1}^n\frac{s_k^2}{(1+g_{1n}^0(z)s_k)^3}\re\rtr\left(\bd^{-1}(z)\bS_1\left(\bd'(z)\right)^{-1}\bS_1\right)\notag\\
&+\frac{\kappa_x}{zn^3}\sum_{k=1}^n\frac{s_k^2}{(1+g_{1n}^0(z)s_k)^3}\sum_{j\neq k=1}^n\re\left(\be_j'\bt_{1n}^*\bd^{-1}(z)\bt_{1n}\be_j\right)^2\Bigg]\notag\\
=&\left(d_{n1}(z)+d_{n2}(z)\right)\Bigg[1-\frac{c_n}z\sum_{k=1}^n\frac{s_k^2}{(1+g_{1n}^0(z)s_k)^2}\int\frac{x}
{\left({1+g_{2n}^0(z)x}\right)^2}dH_{1n}(x)\notag\\
&\quad\times\left[\sum_{k=1}^n\frac{s_k}{(1+g_{1n}^0(z)s_k)^2}\right]^{-1}\Bigg]^{-1}+o(1)\notag
\end{align}
where
\begin{align*}
  d_{n2}(z)=&\frac{c_n}z\sum_{k=1}^n\frac{s_k^2}{(1+g_{1n}^0(z)s_k)^2}\int\frac{x}
{\left({1+g_{2n}^0(z)x}\right)^2}dH_{1n}(x)\left[\sum_{k=1}^n\frac{s_k}{(1+g_{1n}^0(z)s_k)^2}\right]^{-1}\\
&\times\Bigg[\frac{\alpha_x}{zn^2}\sum_{k=1}^n\frac{s_k^2}{(1+g_{1n}^0(z)s_k)^3}\re\rtr\left(\bd^{-1}(z)\bS_1\left(\bd'(z)\right)^{-1}\bS_1\right)\\
&+\frac{\kappa_x}{zn^2}\sum_{k=1}^n\frac{s_k^2}{(1+g_{1n}^0(z)s_k)^3}\sum_{j\neq k=1}^n\re\left(\be_j'\bt_{1n}^*\bd^{-1}(z)\bt_{1n}\be_j\right)^2\Bigg].
\end{align*}
In the following, we show
\begin{align*}
1-\frac{c_n}z\sum_{k=1}^n\frac{s_k^2}{(1+g_{1n}^0(z)s_k)^2}\int\frac{x}
{\left({1+g_{2n}^0(z)x}\right)^2}dH_{1n}(x)\left[\sum_{k=1}^n\frac{s_k}{(1+g_{1n}^0(z)s_k)^2}\right]^{-1}\neq0.
\end{align*}
Write
\begin{align*}
&\frac1{n}\sum_{k=1}^n\frac{s_k}{\left(1+g_{1}(z)s_k\right)^2}-\frac c{zn}\sum_{k=1}^n\frac{s_k^2}{\left(1+g_{1}(z)s_k\right)^2}\int\frac{x}
{\left({1+g_2(z)x}\right)^2}dH_{1}(x)\\
=&-zg_2(z)-\frac 1n\sum_{k=1}^n\frac{s_k^2}{\left(1+g_{1}(z)s_k\right)^2}\left[g_1(z)+\frac cz\int\frac{x}
{\left({1+g_{2}(z)x}\right)^2}dH_{1}(x)\right]\\
=&-zg_2(z)\left\{1-\frac c{z^2n}\sum_{k=1}^n\frac{s_k^2}{\left(1+g_{1}(z)s_k\right)^2}\int\frac{x^2}
{\left({1+g_{2}(z)x}\right)^2}dH_{1}(x)\right\}.
\end{align*}
Note that for all $z=u+iv\in\mathcal{C}$
\begin{align*}
&\left|\frac c{z^2}\int\frac{y^2}{\left(1+g_{1}(z)y\right)^2}dH_2(y)\int\frac{x^2}
{\left({1+g_{2}(z)x}\right)^2}dH_{1}(x)\right|\\
\le&c\int\frac{y^2}{\left|1+g_{1}(z)y\right|^2}dH_2(y)\int\frac{x^2}
{\left|{z+zg_{2}(z)x}\right|^2}dH_{1}(x)\\
=&\frac{\Im \left(zg_2(z)\right)}{\Im g_1(z)}\frac{\Im g_1(z)-cv\int\frac{x}
{\left|{z+zg_{1}(z)x}\right|^2}dH_{2}(x)}{\Im \left(zg_2(z)\right)}<1.
\end{align*}

We are now in position to find the limit of $d_{n1}(z)$ and $d_{n2}(z)$. From (\ref{sal:6}), we have
\begin{align}\label{bal18}
{\bf W}^{-1}(z)-\bd^{-1}(z)
=&\frac1n\sum_{j=1}^ns_j\varphi_j(z){\bf W}^{-1}(z)\left(\bq_j\bq_j^*-\bS_1\right)\bd_{j}^{-1}(z)\\
&+\frac1n\sum_{j=1}^ns_j\left(\beta_{j}(z)-\varphi_j(z)\right){\bf W}^{-1}(z)\bq_j\bq_j^*\bd_{j}^{-1}(z)\notag\\
&+\frac1n\sum_{j=1}^ns_j\varphi_j(z){\bf W}^{-1}(z)\bS_1\left(\bd_{j}^{-1}(z)-\bd^{-1}(z)\right)\notag\\
\triangleq&{\bf G_1}(z)+{\bf G_2}(z)+{\bf G_3}(z)\notag.
\end{align}
Let ${\bf M}$ be $p\times p$ matrix with a nonrandom bound on the spectral norm of ${\bf M}$ for all parameters governing ${\bf M}$ and under all realizations of ${\bf M}$. Applying (\ref{bal1}) and (\ref{cal1}), we obtain
\begin{align}\label{bal16}
\re|\beta_{j}(z)-&\varphi_j(z)|^2
\le\frac C{n^2}\re|\beta_j(z)\left(\bq_j^*\bd_j^{-1}(z)\bq_j-\re\bd^{-1}(z)\bS_1\right)|^2\\
\le&\frac C{n^2}\re|\beta_j(z)\rho_j(z)|^2+\frac C{n^2}\re|\beta_j(z)|^2\re|\rtr\left((\bd_j^{-1}(z)-\bd^{-1}(z))\bS_1\right)|^2\notag\\
\le&\frac Cn.\notag
\end{align}
which implies that
\begin{align}\label{bal19}
\re\left|\rtr\left({\bf G_2}(z){\bf M}\right)\right|
\le\frac Cn\sum_{j=1}^n\re^{1/2}\left|\beta_{j}(z)-\varphi_j(z)\right|^2\re^{1/2}\left|\bq_j^*\bq_j\right|^2
\le Cn^{1/2}.
\end{align}
Form Lemma \ref{lep1}, we have
\begin{align}\label{bal20}
\re\left|\rtr\left({\bf G_3}(z){\bf M}\right)\right|\le&\frac C{n^2}\sum_{j=1}^n\re\left|\bq_j^*\bq_j\right|
\le C.
\end{align}
Furthermore, write
\begin{align*}
&\re\rtr\left({\bd}^{-1}(z){\bf M}\left(\bd'(z)\right)^{-1}\bS_1\right)\\
=&\re\rtr\left({\bw}^{-1}(z){\bf M}\left(\bd'(z)\right)^{-1}\bS_1\right)-\re\rtr\left({\bf G}_1(z){\bf M}\left(\bd'(z)\right)^{-1}\bS_1\right)+O(n^{1/2})\\
=&\re\rtr\left({\bw}^{-1}(z){\bf M}\left(\bd'(z)\right)^{-1}\bS_1\right)+O(n^{1/2})\\
&-\frac1n\sum_{j=1}^ns_j\varphi_j(z)\re\bq_j^*\bd_{j}^{-1}(z){\bf M}\left(\left(\bd'(z)\right)^{-1}-\left(\bd'_j(z)\right)^{-1}\right)\bS_1{\bf W}^{-1}(z)\bq_j\\
&-\frac1n\sum_{j=1}^ns_j\varphi_j(z)\re\rtr\left[{\bf W}^{-1}(z)\bS_1\bd_{j}^{-1}(z){\bf M}\left(\left(\bd_j'(z)\right)^{-1}-\left(\bd'(z)\right)^{-1}\right)\bS_1\right]\\
=&\re\rtr\left({\bw}^{-1}(z){\bf M}\left(\bd'(z)\right)^{-1}\bS_1\right)+O(n^{1/2})\\
&+\frac1{n^2}\sum_{j=1}^ns_j^2\varphi_j(z)\re\beta_j(z)\bq_j^*\bd_{j}^{-1}(z){\bf M}\left(\bd_j'(z)\right)^{-1}\bar\bq_j\bq_j'\left(\bd'_j(z)\right)^{-1}\bS_1{\bf W}^{-1}(z)\bq_j\\
&-\frac1{n^2}\sum_{j=1}^ns_j^2\varphi_j(z)\re\beta_j(z)\bq_j'\left(\bd'_j(z)\right)^{-1}\bS_1{\bf W}^{-1}(z)\bS_1\bd_{j}^{-1}(z){\bf M}\left(\bd_j'(z)\right)^{-1}\bar\bq_j\\
\triangleq&\re\rtr\left({\bw}^{-1}(z){\bf M}\left(\bd'(z)\right)^{-1}\bS_1\right)+O(n^{1/2})+r_1(z)+r_2(z).
\end{align*}
 Using Lemma \ref{lep1} and (\ref{bal1}), we have
\begin{align*}
\re|r_2(z)|\le C.
\end{align*}
Together with the above inequality, (\ref{cal3}), and (\ref{bal:9}), one gets
\begin{align*}
&\re\rtr\left({\bd}^{-1}(z){\bf M}\left(\bd'(z)\right)^{-1}\bS_1\right)=\re\rtr\left({\bw}^{-1}(z){\bf M}\left(\bd'(z)\right)^{-1}\bS_1\right)+O(n^{1/2})\\
&+\frac1{n^2}\sum_{j=1}^ns_j^2\varphi_j^2(z)\re\bq_j^*\bd_{j}^{-1}(z){\bf M}\left(\bd_j'(z)\right)^{-1}\bar\bq_j\bq_j'\left(\bd'_j(z)\right)^{-1}\bS_1{\bf W}^{-1}(z)\bq_j\\
=&\re\rtr\left({\bw}^{-1}(z){\bf M}\left(\bd'(z)\right)^{-1}\bS_1\right)+O(n^{1/2})\\
&+\frac{\alpha_x}{n^2}\sum_{j=1}^ns_j^2\varphi_j^2(z)\re\rtr\left(\bd_{j}^{-1}(z){\bf M}\left(\bd_j'(z)\right)^{-1}\bS_1\right)\rtr\left(\left(\bd'_j(z)\right)^{-1}\bS_1{\bf W}^{-1}(z)\bS_1\right)\\
=&\re\rtr\left({\bw}^{-1}(z){\bf M}\left(\bd'(z)\right)^{-1}\bS_1\right)+O(n^{1/2})\\
&+\frac{\alpha_x}{n^2}\sum_{j=1}^ns_j^2\varphi_j^2(z)\re\rtr\left(\bd^{-1}(z){\bf M}\left(\bd'(z)\right)^{-1}\bS_1\right)\re\rtr\left(\left(\bd'(z)\right)^{-1}\bS_1{\bf W}^{-1}(z)\bS_1\right).
\end{align*}
By (\ref{bal18}), (\ref{bal19}), and (\ref{bal20}), it follow that
\begin{align*}
\re\rtr\left({\bd}^{-1}(z){\bf M}\left(\bd'(z)\right)^{-1}\bS_1\right)
=&\frac{\re\rtr\left({\bw}^{-1}(z){\bf M}\bw^{-1}(z)\bS_1\right)}{1-\frac{\alpha_x}{n^2}\sum_{j=1}^ns_j^2\varphi_j^2(z)\re\rtr\left(\bw^{-1}(z)\bS_1{\bf W}^{-1}(z)\bS_1\right)}+O(n^{1/2}).
\end{align*}
Due to (\ref{cl4}) and (\ref{cal5}), we find
\begin{align*}
\frac1n\re\rtr\left({\bd}^{-1}(z)\bS_1\left(\bd'(z)\right)^{-1}\bS_1\right)
=&\frac{c_nz^{-2}d_{n3}(z)}{1-{{\alpha_x}c_n}{z^{-2}}d_{n3}(z)d_{n4}(z)}+o(1).
\end{align*}
and
\begin{align*}
&\frac1n\re\rtr\left(\bd^{-1}(z)\bS_1\bw^{-1}(z)\left(\bd'(z)\right)^{-1}\bS_1\right)
=\frac{-c_nz^{-3}\int x^2/\left(1+g_{2n}^0(z)x\right)^3dH_{1n}(x)}{1-{{\alpha_x}c_n}{z^{-2}}d_{n3}(z)d_{n4}(z)}+o(1)
\end{align*}
where $d_{n3}(z)=\int \frac{x^2}{\left(1+g_{2n}^0(z)x\right)^2}dH_{1n}(x)$ and $d_{n4}(z)=\int\frac{y^2}{(1+g_{1n}^0(z)y)^2}dH_{2n}(y)$.

Therefore, it follow from (\ref{cal7}) and (\ref{cal5})
\begin{align*}
  &d_{n1}(z)\\
  =&\frac{\alpha_x}{n^2}\sum_{j=1}^n\frac{s_j^2}{(1+s_jg_{1n}^0(z))^2}\re\rtr\left(\bd^{-1}(z)\bS_1\bw^{-1}(z)\left(\bd'(z)\right)^{-1}\bS_1\right)\\
  &+\frac{\kappa_x}{n^2}\sum_{j=1}^n\frac{s_j^2}{(1+s_jg_{1n}^0(z))^2}\re\rtr\left(\bw^{-1}(z)\bS_1\bw^{-2}(z)\bS_1\right)\\
&-\frac{\alpha_x}{n^3}\sum_{j=1}^n\frac{s_j^3}{(1+s_jg_{1n}^0(z))^3}\re\rtr\left(\bd^{-1}(z)\bS_1\left(\bd'(z)\right)^{-1}\bS_1\right)
\re\rtr\left(\bw^{-2}(z)\bS_1\right)\\
  &-\frac{\kappa_x}{n^3}\sum_{j=1}^n\frac{s_j^3}{(1+s_jg_{1n}^0(z))^3}\re\rtr\left(\bw^{-1}(z)\bS_1\bw^{-1}(z)\bS_1\right)\re\rtr\left(\bw^{-2}(z)\bS_1\right)+o(1)\\
=&-\frac{\alpha_x c_n d_{n4}(z)}{z^{3}}\frac{\int x^2/\left(1+g_{2n}^0(z)x\right)^3dH_{1n}(x)}{1-{{\alpha_x}c_n}{z^{-2}}d_{n3}(z)d_{n4}(z)}-\frac{\kappa_xc_n d_{n4}(z)}{z^3}\int \frac{x^2}{\left(1+g_{2n}^0(z)x\right)^3}dH_{1n}(x)\\
&-\frac{\alpha_xc_n^2 d_{n3}(z)}{z^4({1-{{\alpha_x}c_n}z^{-2}{}d_{n3}(z)d_{n4}(z)})}\int\frac{y^3}{(1+g_{1n}^0(z)y)^3}dH_{2n}(y)
\int \frac{x}{\left(1+g_{2n}^0(z)x\right)^2}dH_{1n}(x)\\
  &-\frac{\kappa_xc_n^2 d_{n3}(z)}{z^4}\int\frac{y^3}{(1+g_{1n}^0(z)y)^3}dH_{2n}(y)\int \frac{x}{\left(1+g_{2n}^0(z)x\right)^2}dH_{1n}(x)+o(1)  .
\end{align*}
and
\begin{align*}
  d_{n2}(z)=&\frac{c_nd_{n4}(z)}z\int\frac{x}
{\left({1+g_{2n}^0(z)x}\right)^2}dH_{1n}(x)\left[\int\frac{y}{(1+g_{1n}^0(z)y)^2}dH_{2n}(y)\right]^{-1}\\
&\times\Bigg[\frac{\alpha_x}{zn^2}\sum_{k=1}^n\frac{s_k^2}{(1+g_{1n}^0(z)s_k)^3}\re\rtr\left(\bd^{-1}(z)\bS_1\left(\bd'(z)\right)^{-1}\bS_1\right)\\
&+\frac{\kappa_x}{zn^2}\sum_{k=1}^n\frac{s_k^2}{(1+g_{1n}^0(z)s_k)^3}\re\rtr\left(\bw^{-1}(z)\bS_1\bw^{-1}(z)\bS_1\right)\Bigg]+o(1)\\
=&\frac{\alpha_x c_n^2 d_{n3}(z)d_{n4}(z)}{z^4({1-{{\alpha_x}c_n}{z^{-2}}d_{n3}(z)d_{n4}(z)})}\int\frac{x}
{\left({1+g_{2n}^0(z)x}\right)^2}dH_{1n}(x)\\
&\times\int\frac{y^2}{(1+g_{1n}^0(z)y)^3}dH_{2n}(y)\left[\int\frac{y}{(1+g_{1n}^0(z)y)^2}dH_{2n}(y)\right]^{-1}\\
+&\frac{{\kappa_x}c_n^2 d_{n3}(z)d_{n4}(z)}{z^4}\int\frac{x}
{\left({1+g_{2n}^0(z)x}\right)^2}dH_{1n}(x)\\
&\times\int\frac{y^2}{(1+g_{1n}^0(z)y)^3}dH_{2n}(y)\left[\int\frac{y}{(1+g_{1n}^0(z)y)^2}dH_{2n}(y)\right]^{-1}+o(1).
\end{align*}

From (\ref{cal8}) and the above two equalities, we conclude that
\begin{align*}
M_{n2}(z)
\to&\left(d_{1}(z)+d_{2}(z)\right)\Bigg[1-\frac{cd_{4}(z)}z\frac{\int{x}/
{\left({1+g_{2}(z)x}\right)^2}dH_{1}(x)}{\int{y}/{(1+g_{1}(z)y)^2}dH_{2}(y)}\Bigg]^{-1}\\
=&-\frac{d_{1}(z)+d_{2}(z)}{zg_2(z)(1-cz^{-2}d_3(z)d_4(z))}{\int\frac{y}{(1+g_{1}(z)y)^2}dH_{2}(y)}
\end{align*}
where $d_1(z)$ and $d_2(z)$ are the limits of $d_{n1}(z)$ and $d_{n2}(z)$ respectively. By calculating, we find
\begin{align*}
&\frac{d_{1}(z)+ d_{2}(z)}{zg_2(z)}{\int\frac{y}{(1+g_{1}(z)y)^2}dH_{2}(y)}\\
=&\left(\frac{\alpha_x}{1-{{\alpha_x}c}{z^{-2}}d_{3}(z)d_{4}(z)}+{\kappa_x}\right)\left[\frac{c d_3(z)d_{4}(z)}{z^3}-\frac{c^2 d_{3}^2(z)d_{4}^2(z)}{z^5}+\frac{c d_{5}(z)}{z^4}+\frac{c^2d_{6}(z)}{z^4}\right].
\end{align*}

\subsection{List of necessary lemmas}

\subsubsection{Lemmas that need to prove}

\begin{lemma}\label{le1}
Let $1\le j\le m_1,1\le k\le m_2$. Recall the definition of ${\bf G}_n$ in (\ref{eq31}). Then, for any $1\le a,b \le p$, we have
\begin{align*}
\frac{\partial g_{ab}}{\partial w_{jk}}=\left[\frac{\partial {\bf G}_n}{\partial w_{jk}}\right]_{ab}=\frac1n\left[\bt_{1n}\right]_{aj}\left[\bt_{1n}\bw_n\bt_{2n}\right]_{b k}
+\frac1n\left[\bt_{1n}\bw_n\bt_{2n}\right]_{ak}\left[\bt_{1n}\right]_{b j}.
\end{align*}
\end{lemma}
\begin{proof}
It is obvious that
\begin{align*}
\frac{\partial {\bf G}_n}{\partial w_{jk}}=&\frac1n\bt_{1n}\frac{\partial{\bf W}_n}{\partial w_{jk}}\bt_{2n}\bw_n'\bt_{1n}'
+\frac1n\bt_{1n}\bw_n\bt_{2n}\frac{\partial{\bf W}_n'}{\partial w_{jk}}\bt_{1n}'\\
=&\frac1n\bt_{1n}{\bf e}_j{\bf e}_k'\bt_{2n}\bw_n'\bt_{1n}'
+\frac1n\bt_{1n}\bw_n\bt_{2n}{\bf e}_k{\bf e}_j'\bt_{1n}'.
\end{align*}
This yields
\begin{align*}
\left[\frac{\partial {\bf G}_n}{\partial w_{jk}}\right]_{ab}=&\frac1n\left[\bt_{1n}\right]_{aj}\left[\bt_{1n}\bw_n\bt_{2n}\right]_{b k}
+\frac1n\left[\bt_{1n}\bw_n\bt_{2n}\right]_{ak}\left[\bt_{1n}\right]_{b j}.
\end{align*}
\end{proof}

\begin{lemma}\label{le2}
Let $1\le j\le m_1,1\le k\le m_2$. Recall the definition of ${\bf H}_n(t)$ in (\ref{eq32}). Then for any $1\le d,l \le p$
\begin{align*}
\frac{\partial h_{dl}}{\partial w_{jk}}=\frac i n\left[{\bf H}_n\bt_{1n}\right]_{dj}*\left[\bt_{2n}\bw_n'\bt_{1n}'{\bf H}_n\right]_{kl}(t)+\frac i n\left[{\bf H}_n\bt_{1n}\bw_n\bt_{2n}\right]_{dk}*\left[\bt_{1n}'{\bf H}_n\right]_{jl}(t)
\end{align*}
where $h_{dl}=\left({\bf H}_n(t)\right)_{dl}$ and $f*g(t)=\int_0^tf(s)g(t-s)ds$.
\end{lemma}
\begin{proof}
Applying Lemma \ref{le1} and Lemma \ref{lea1}, we get
\begin{align*}
\frac{\partial {\bf H}_n(t)}{\partial w_{jk}}=&\sum_{a,b=1}^p\frac{\partial {\bf H}_n(t)}{\partial g_{ab}}\frac{\partial g_{ab}}{\partial w_{jk}}
=\frac i n\sum_{a,b=1}^p\int_0^te^{is{\bf G}_n}{\bf e}_a{\bf e}_b'e^{i(t-s){\bf G}_n}ds\\
&\times\left\{\left[\bt_{1n}\right]_{aj}\left[\bt_{1n}\bw_n\bt_{2n}\right]_{b k}
+\left[\bt_{1n}\bw_n\bt_{2n}\right]_{ak}\left[\bt_{1n}\right]_{b j}\right\}.
\end{align*}
Hence, one has
\begin{align*}
\frac{\partial h_{dl}}{\partial w_{jk}}
=&\frac i n\sum_{a,b=1}^ph_{da}*h_{bl}(t)
\left\{\left[\bt_{1n}\right]_{aj}\left[\bt_{1n}\bw_n\bt_{2n}\right]_{b k}
+\left[\bt_{1n}\bw_n\bt_{2n}\right]_{ak}\left[\bt_{1n}\right]_{b j}\right\}\\
=&\frac i n\left[{\bf H}_n\bt_{1n}\right]_{dj}*\left[\bt_{2n}\bw_n'\bt_{1n}'{\bf H}_n\right]_{kl}(t)+\frac i n\left[{\bf H}_n\bt_{1n}\bw_n\bt_{2n}\right]_{dk}*\left[\bt_{1n}'{\bf H}_n\right]_{jl}(t).
\end{align*}
\end{proof}

\begin{lemma}\label{le3}
Let $1\le j\le m_1,1\le k\le m_2$. Recall the definitions of $S(\theta)$ in (\ref{eq32}) and $\widetilde f({\bf G}_n)$ in (\ref{eq30}). Then
\begin{align*}
\frac{\partial S(\theta)}{\partial w_{jk}}=\frac {2} n\left[\bt_{1n}'\widetilde f({\bf G}_n)\bt_{1n}\bw_n\bt_{2n}\right]_{jk}.
\end{align*}
\end{lemma}
\begin{proof}
By the inverse Fourier transform, we obtain
\begin{align*}
\frac{\partial S(\theta)}{\partial w_{jk}}=\int_{-\infty}^{\infty}\widehat f(u)\rtr\frac{\partial {\bf H}_n(u)}{\partial w_{jk}}du.
\end{align*}
It follows from Lemma \ref{le2} that
\begin{align*}
\frac{\partial S(\theta)}{\partial w_{jk}}=&\frac {2i} n\int_{-\infty}^{\infty}\widehat f(u)\sum_{d=1}^p\left[{\bf H}_n\bt_{1n}\right]_{dj}*\left[\bt_{2n}\bw_n'\bt_{1n}'{\bf H}_n\right]_{kd}(u)du\\
=&\frac {2i} n\int_{-\infty}^{\infty}u\widehat f(u)\left[\bt_{1n}'{\bf H}_n(u)\bt_{1n}\bw_n\bt_{2n}\right]_{jk}du\\
=&\frac {2} n\left[\bt_{1n}'\widetilde f({\bf G}_n)\bt_{1n}\bw_n\bt_{2n}\right]_{jk}.
\end{align*}
\end{proof}

\begin{lemma}\label{le5}
Let $1\le j\le m_1,1\le k\le m_2,1\le d,l\le p$. Recall the definitions of  ${\bf H}_{n}(u)$ in (\ref{eq32}). Then
\begin{align*}
\left[\frac{\partial^2 {\bf H}_n(u)}{\partial w_{jk}^2}\right]_{dl}
=&\frac{2i}n\left[\bt_{2n}\right]_{kk}\left[{\bf H}_n\bt_{1n}\right]_{dj}*\left[\bt_{1n}'{\bf H}_n\right]_{jl}(u)\\
&-\frac {2}{n^2}\left[{\bf H}_n\bt_{1n}\right]_{dj}*\left[\bt_{2n}\bw_n'\bt_{1n}'{\bf H}_n\bt_{1n}\right]_{kj}*\left[\bt_{2n}\bw_n'\bt_{1n}'{\bf H}_n\right]_{kl}(u)\\
&-\frac {2}{n^2}\left[{\bf H}_n\bt_{1n}\bw_n\bt_{2n}\right]_{dk}*\left[\bt_{1n}'{\bf H}_n\bt_{1n}\right]_{jj}*\left[\bt_{2n}\bw_n'\bt_{1n}'{\bf H}_n\right]_{kl}(u)\\
&-\frac {2}{n^2}\left[{\bf H}_n\bt_{1n}\right]_{dj}*\left[\bt_{2n}\bw_n'\bt_{1n}'{\bf H}_n\bt_{1n}\bw_n\bt_{2n}\right]_{kk}*\left[\bt_{1n}'{\bf H}_n\right]_{jl}(u)\\
&-\frac {2}{n^2}\left[{\bf H}_n\bt_{1n}\bw_n\bt_{2n}\right]_{dk}*\left[\bt_{1n}'{\bf H}_n\bt_{1n}\bw_n\bt_{2n}\right]_{jk}*\left[\bt_{1n}'{\bf H}_n\right]_{jl}(u).
\end{align*}
\end{lemma}
\begin{proof}
Applying Lemma \ref{le2}, we have
\begin{align*}
&\left[\frac{\partial^2 {\bf H}_n(u)}{\partial w_{jk}^2}\right]_{dl}\\
  =&\frac in\frac{\partial \left[{\bf H}_n\bt_{1n}\right]_{dj}*\left[\bt_{2n}\bw_n'\bt_{1n}'{\bf H}_n\right]_{kl}(u)}{\partial w_{jk}}+\frac in\frac{\partial\left[{\bf H}_n\bt_{1n}\bw_n\bt_{2n}\right]_{dk}*\left[\bt_{1n}'{\bf H}_n\right]_{jl}(u)}{\partial w_{jk}}\\
=&\frac{2i}n\left[\bt_{2n}\right]_{kk}\left[{\bf H}_n\bt_{1n}\right]_{dj}*\left[\bt_{1n}'{\bf H}_n\right]_{jl}(u)\\
&+\frac in\left[\frac{\partial {\bf H}_n}{\partial w_{jk}}\bt_{1n}\right]_{dj}*\left[\bt_{2n}\bw_n'\bt_{1n}'{\bf H}_n\right]_{kl}(u)
+\frac in\left[{\bf H}_n\bt_{1n}\right]_{dj}*\left[\bt_{2n}\bw_n'\bt_{1n}'\frac{\partial {\bf H}_n}{\partial w_{jk}}\right]_{kl}(u)\\
&+\frac in\left[\frac{\partial{\bf H}_n}{\partial w_{jk}}\bt_{1n}\bw_n\bt_{2n}\right]_{dk}*\left[\bt_{1n}'{\bf H}_n\right]_{jl}(u)
+\frac in\left[{\bf H}_n\bt_{1n}\bw_n\bt_{2n}\right]_{dk}*\left[\bt_{1n}'\frac{\partial{\bf H}_n}{\partial w_{jk}}\right]_{jl}(u)\\
=&\frac{2i}n\left[\bt_{2n}\right]_{kk}\left[{\bf H}_n\bt_{1n}\right]_{dj}*\left[\bt_{1n}'{\bf H}_n\right]_{jl}(u)\\
&-\frac {2}{n^2}\left[{\bf H}_n\bt_{1n}\right]_{dj}*\left[\bt_{2n}\bw_n'\bt_{1n}'{\bf H}_n\bt_{1n}\right]_{kj}*\left[\bt_{2n}\bw_n'\bt_{1n}'{\bf H}_n\right]_{kl}(u)\\
&-\frac {2}{n^2}\left[{\bf H}_n\bt_{1n}\bw_n\bt_{2n}\right]_{dk}*\left[\bt_{1n}'{\bf H}_n\bt_{1n}\right]_{jj}*\left[\bt_{2n}\bw_n'\bt_{1n}'{\bf H}_n\right]_{kl}(u)\\
&-\frac {2}{n^2}\left[{\bf H}_n\bt_{1n}\right]_{dj}*\left[\bt_{2n}\bw_n'\bt_{1n}'{\bf H}_n\bt_{1n}\bw_n\bt_{2n}\right]_{kk}*\left[\bt_{1n}'{\bf H}_n\right]_{jl}(u)\\
&-\frac {2}{n^2}\left[{\bf H}_n\bt_{1n}\bw_n\bt_{2n}\right]_{dk}*\left[\bt_{1n}'{\bf H}_n\bt_{1n}\bw_n\bt_{2n}\right]_{jk}*\left[\bt_{1n}'{\bf H}_n\right]_{jl}(u).
\end{align*}
\end{proof}

\begin{lemma}\label{le6}
Let $1\le j\le m_1,1\le k\le m_2,1\le d,l\le p$. Recall the definitions of  ${\bf H}_{n}(u)$ in (\ref{eq32}). Then
\begin{align*}
  &\left[\frac{\partial^3 {\bf H}_n(u)}{\partial w_{jk}^3}\right]_{dl}\\
=&-\frac{6}{n^2}\left[\bt_{2n}\right]_{kk}\left[{\bf H}_n\bt_{1n}\right]_{dj}*\left[\bt_{2n}\bw_n'\bt_{1n}'{\bf H}_n\bt_{1n}\right]_{kj}*\left[\bt_{1n}'{\bf H}_n\right]_{jl}(u)\\
&-\frac{6}{n^2}\left[\bt_{2n}\right]_{kk}\left[{\bf H}_n\bt_{1n}\bw_n\bt_{2n}\right]_{dk}*\left[\bt_{1n}'{\bf H}_n\bt_{1n}\right]_{jj}*\left[\bt_{1n}'{\bf H}_n\right]_{jl}(u)\\
&-\frac{6}{n^2}\left[\bt_{2n}\right]_{kk}\left[{\bf H}_n\bt_{1n}\right]_{dj}*\left[\bt_{1n}'{\bf H}_n\bt_{1n}\right]_{jj}*\left[\bt_{2n}\bw_n'\bt_{1n}'{\bf H}_n\right]_{kl}(u)\\
&-\frac{6}{n^2}\left[\bt_{2n}\right]_{kk}\left[{\bf H}_n\bt_{1n}\right]_{dj}*\left[\bt_{1n}'{\bf H}_n\bt_{1n}\bw_n\bt_{2n}\right]_{jk}*\left[\bt_{1n}'{\bf H}_n\right]_{jl}(u)\\
&-\frac{6i}{n^3} \left[{\bf H}_n\bt_{1n}\right]_{dj}*\left[\bt_{2n}\bw_n'\bt_{1n}'{\bf H}_n\bt_{1n}\right]_{kj}*\left[\bt_{2n}\bw_n'\bt_{1n}'{\bf H}_n\bt_{1n}\right]_{kj}*\left[\bt_{2n}\bw_n'\bt_{1n}'{\bf H}_n\right]_{kl}(u)\\
&-\frac{6i}{n^3} \left[{\bf H}_n\bt_{1n}\bw_n\bt_{2n}\right]_{dk}*\left[\bt_{1n}'{\bf H}_n\bt_{1n}\right]_{jj}*\left[\bt_{2n}\bw_n'\bt_{1n}'{\bf H}_n\bt_{1n}\right]_{kj}*\left[\bt_{2n}\bw_n'\bt_{1n}'{\bf H}_n\right]_{kl}(u)\\
&-\frac{6i}{n^3} \left[{\bf H}_n\bt_{1n}\right]_{dj}*\left[\bt_{2n}\bw_n'\bt_{1n}'{\bf H}_n\bt_{1n}\bw_n\bt_{2n}\right]_{kk}*\left[\bt_{1n}'{\bf H}_n\bt_{1n}\right]_{jj}*\left[\bt_{2n}\bw_n'\bt_{1n}'{\bf H}_n\right]_{kl}(u)\\
&-\frac{6i}{n^3} \left[{\bf H}_n\bt_{1n}\bw_n\bt_{2n}\right]_{dk}*\left[\bt_{1n}'{\bf H}_n\bt_{1n}\bw_n\bt_{2n}\right]_{jk}*\left[\bt_{1n}'{\bf H}_n\bt_{1n}\right]_{jj}\left[\bt_{2n}\bw_n'\bt_{1n}'{\bf H}_n\right]_{kl}(u)\\
&-\frac{6i}{n^3} \left[{\bf H}_n\bt_{1n}\right]_{dj}*\left[\bt_{2n}\bw_n'\bt_{1n}'{\bf H}_n\bt_{1n}\right]_{kj}*\left[\bt_{2n}\bw_n'\bt_{1n}'{\bf H}_n\bt_{1n}\bw_n\bt_{2n}\right]_{kk}*\left[\bt_{1n}'{\bf H}_n\right]_{jl}(u)\\
&-\frac{6i}{n^3} \left[{\bf H}_n\bt_{1n}\bw_n\bt_{2n}\right]_{dk}*\left[\bt_{1n}'{\bf H}_n\bt_{1n}\right]_{jj}*\left[\bt_{2n}\bw_n'\bt_{1n}'{\bf H}_n\bt_{1n}\bw_n\bt_{2n}\right]_{kk}*\left[\bt_{1n}'{\bf H}_n\right]_{jl}(u)\\
&-\frac{6i}{n^3}\ \left[{\bf H}_n\bt_{1n}\right]_{dj}*\left[\bt_{2n}\bw_n'\bt_{1n}'{\bf H}_n\bt_{1n}\bw_n\bt_{2n}\right]_{kk}*\left[\bt_{1n}'{\bf H}_n\bt_{1n}\bw_n\bt_{2n}\right]_{jk}*\left[\bt_{1n}'{\bf H}_n\right]_{jl}(u) \\
&-\frac{6i}{n^3}\left[{\bf H}_n\bt_{1n}\bw_n\bt_{2n}\right]_{dk}*\left[\bt_{1n}'{\bf H}_n\bt_{1n}\bw_n\bt_{2n}\right]_{jk}*\left[\bt_{1n}'{\bf H}_n\bt_{1n}\bw_n\bt_{2n}\right]_{jk}*\left[\bt_{1n}'{\bf H}_n\right]_{jl}(u).
\end{align*}
\end{lemma}
\begin{proof}
From Lemma \ref{le2} and Lemma \ref{le5}, one gets
\begin{align*}
  &\left[\frac{\partial^3 {\bf H}_n(u)}{\partial w_{jk}^3}\right]_{dl}
=\frac in\frac{\partial^2 \left[{\bf H}_n\bt_{1n}\right]_{dj}*\left[\bt_{2n}\bw_n'\bt_{1n}'{\bf H}_n\right]_{kl}(u)}{\partial w_{jk}^2}\\
&\qquad\qquad+\frac in\frac{\partial^2\left[{\bf H}_n\bt_{1n}\bw_n\bt_{2n}\right]_{dk}*\left[\bt_{1n}'{\bf H}_n\right]_{jl}(u)}{\partial w_{jk}^2}\\
 =&\frac{2i}n\left[\bt_{2n}\right]_{kk}\frac{\partial\left[{\bf H}_n\bt_{1n}\right]_{dj}*\left[\bt_{1n}'{\bf H}_n\right]_{jl}(u)}{\partial w_{jk}}\\
 &+\frac in\frac{\partial \left[\frac{\partial {\bf H}_n}{\partial w_{jk}}\bt_{1n}\right]_{dj}*\left[\bt_{2n}\bw_n'\bt_{1n}'{\bf H}_n\right]_{kl}(u)}{\partial w_{jk}}
+\frac in\frac{\partial \left[{\bf H}_n\bt_{1n}\right]_{dj}*\left[\bt_{2n}\bw_n'\bt_{1n}'\frac{\partial {\bf H}_n}{\partial w_{jk}}\right]_{kl}(u)}{\partial w_{jk}}\\
&+\frac in\frac{\partial\left[\frac{\partial{\bf H}_n}{\partial w_{jk}}\bt_{1n}\bw_n\bt_{2n}\right]_{dk}*\left[\bt_{1n}'{\bf H}_n\right]_{jl}(u)}{\partial w_{jk}}
+\frac in\frac{\partial\left[{\bf H}_n\bt_{1n}\bw_n\bt_{2n}\right]_{dk}*\left[\bt_{1n}'\frac{\partial{\bf H}_n}{\partial w_{jk}}\right]_{jl}(u)}{\partial w_{jk}}\\
=&\frac{4i}n\left[\bt_{2n}\right]_{kk}\left[\frac{\partial{\bf H}_n}{\partial w_{jk}}\bt_{1n}\right]_{dj}*\left[\bt_{1n}'{\bf H}_n\right]_{jl}(u)
+\frac{4i}n\left[\bt_{2n}\right]_{kk}\left[{\bf H}_n\bt_{1n}\right]_{dj}*\left[\bt_{1n}'\frac{\partial{\bf H}_n}{\partial w_{jk}}\right]_{jl}(u)\\
&+\frac{2i}n \left[\frac{\partial {\bf H}_n}{\partial w_{jk}}\bt_{1n}\right]_{dj}*\left[\bt_{2n}\bw_n'\bt_{1n}'\frac{\partial{\bf H}_n}{\partial w_{jk}}\right]_{kl}(u)
+\frac{2i}n\left[\frac{\partial{\bf H}_n}{\partial w_{jk}}\bt_{1n}\bw_n\bt_{2n}\right]_{dk}*\left[\bt_{1n}'\frac{\partial{\bf H}_n}{\partial w_{jk}}\right]_{jl}(u)\\
&+\frac in\left[\frac{\partial^2 {\bf H}_n}{\partial w_{jk}^2}\bt_{1n}\right]_{dj}*\left[\bt_{2n}\bw_n'\bt_{1n}'{\bf H}_n\right]_{kl}(u)
+\frac in \left[{\bf H}_n\bt_{1n}\right]_{dj}*\left[\bt_{2n}\bw_n'\bt_{1n}'\frac{\partial^2 {\bf H}_n}{\partial w_{jk}^2}\right]_{kl}(u)\\
&+\frac in\left[\frac{\partial^2{\bf H}_n}{\partial w_{jk}^2}\bt_{1n}\bw_n\bt_{2n}\right]_{dk}*\left[\bt_{1n}'{\bf H}_n\right]_{jl}(u)
+\frac in\left[{\bf H}_n\bt_{1n}\bw_n\bt_{2n}\right]_{dk}*\left[\bt_{1n}'\frac{\partial^2{\bf H}_n}{\partial w_{jk}^2}\right]_{jl}(u)\\
=&-\frac{6}{n^2}\left[\bt_{2n}\right]_{kk}\left[{\bf H}_n\bt_{1n}\right]_{dj}*\left[\bt_{2n}\bw_n'\bt_{1n}'{\bf H}_n\bt_{1n}\right]_{kj}*\left[\bt_{1n}'{\bf H}_n\right]_{jl}(u)\\
&-\frac{6}{n^2}\left[\bt_{2n}\right]_{kk}\left[{\bf H}_n\bt_{1n}\bw_n\bt_{2n}\right]_{dk}*\left[\bt_{1n}'{\bf H}_n\bt_{1n}\right]_{jj}*\left[\bt_{1n}'{\bf H}_n\right]_{jl}(u)\\
&-\frac{6}{n^2}\left[\bt_{2n}\right]_{kk}\left[{\bf H}_n\bt_{1n}\right]_{dj}*\left[\bt_{1n}'{\bf H}_n\bt_{1n}\right]_{jj}*\left[\bt_{2n}\bw_n'\bt_{1n}'{\bf H}_n\right]_{kl}(u)\\
&-\frac{6}{n^2}\left[\bt_{2n}\right]_{kk}\left[{\bf H}_n\bt_{1n}\right]_{dj}*\left[\bt_{1n}'{\bf H}_n\bt_{1n}\bw_n\bt_{2n}\right]_{jk}*\left[\bt_{1n}'{\bf H}_n\right]_{jl}(u)\\
&-\frac{6i}{n^3} \left[{\bf H}_n\bt_{1n}\right]_{dj}*\left[\bt_{2n}\bw_n'\bt_{1n}'{\bf H}_n\bt_{1n}\right]_{kj}*\left[\bt_{2n}\bw_n'\bt_{1n}'{\bf H}_n\bt_{1n}\right]_{kj}*\left[\bt_{2n}\bw_n'\bt_{1n}'{\bf H}_n\right]_{kl}(u)\\
&-\frac{6i}{n^3} \left[{\bf H}_n\bt_{1n}\bw_n\bt_{2n}\right]_{dk}*\left[\bt_{1n}'{\bf H}_n\bt_{1n}\right]_{jj}*\left[\bt_{2n}\bw_n'\bt_{1n}'{\bf H}_n\bt_{1n}\right]_{kj}*\left[\bt_{2n}\bw_n'\bt_{1n}'{\bf H}_n\right]_{kl}(u)\\
&-\frac{6i}{n^3} \left[{\bf H}_n\bt_{1n}\right]_{dj}*\left[\bt_{2n}\bw_n'\bt_{1n}'{\bf H}_n\bt_{1n}\bw_n\bt_{2n}\right]_{kk}*\left[\bt_{1n}'{\bf H}_n\bt_{1n}\right]_{jj}*\left[\bt_{2n}\bw_n'\bt_{1n}'{\bf H}_n\right]_{kl}(u)\\
&-\frac{6i}{n^3} \left[{\bf H}_n\bt_{1n}\bw_n\bt_{2n}\right]_{dk}*\left[\bt_{1n}'{\bf H}_n\bt_{1n}\bw_n\bt_{2n}\right]_{jk}*\left[\bt_{1n}'{\bf H}_n\bt_{1n}\right]_{jj}\left[\bt_{2n}\bw_n'\bt_{1n}'{\bf H}_n\right]_{kl}(u)\\
&-\frac{6i}{n^3} \left[{\bf H}_n\bt_{1n}\right]_{dj}*\left[\bt_{2n}\bw_n'\bt_{1n}'{\bf H}_n\bt_{1n}\right]_{kj}*\left[\bt_{2n}\bw_n'\bt_{1n}'{\bf H}_n\bt_{1n}\bw_n\bt_{2n}\right]_{kk}*\left[\bt_{1n}'{\bf H}_n\right]_{jl}(u)\\
&-\frac{6i}{n^3} \left[{\bf H}_n\bt_{1n}\bw_n\bt_{2n}\right]_{dk}*\left[\bt_{1n}'{\bf H}_n\bt_{1n}\right]_{jj}*\left[\bt_{2n}\bw_n'\bt_{1n}'{\bf H}_n\bt_{1n}\bw_n\bt_{2n}\right]_{kk}*\left[\bt_{1n}'{\bf H}_n\right]_{jl}(u)\\
&-\frac{6i}{n^3}\ \left[{\bf H}_n\bt_{1n}\right]_{dj}*\left[\bt_{2n}\bw_n'\bt_{1n}'{\bf H}_n\bt_{1n}\bw_n\bt_{2n}\right]_{kk}*\left[\bt_{1n}'{\bf H}_n\bt_{1n}\bw_n\bt_{2n}\right]_{jk}*\left[\bt_{1n}'{\bf H}_n\right]_{jl}(u) \\
&-\frac{6i}{n^3}\left[{\bf H}_n\bt_{1n}\bw_n\bt_{2n}\right]_{dk}*\left[\bt_{1n}'{\bf H}_n\bt_{1n}\bw_n\bt_{2n}\right]_{jk}*\left[\bt_{1n}'{\bf H}_n\bt_{1n}\bw_n\bt_{2n}\right]_{jk}*\left[\bt_{1n}'{\bf H}_n\right]_{jl}(u).
\end{align*}
\end{proof}

\begin{lemma}\label{le7}
Let $1\le j\le m_1,1\le k\le m_2,1\le d,l\le p$. Recall the definitions of  ${\bf H}_{n}(u)$ in (\ref{eq32}). Then
\begin{footnotesize}
\begin{align*}
&\left[\frac{\partial^4 {\bf H}_n(u)}{\partial w_{jk}^4}\right]_{dl}\\
=&-\frac{24}{n^2}\left[\bt_{2n}\right]_{kk}^2\left[{\bf H}_n\bt_{1n}\right]_{dj}*\left[\bt_{1n}'{\bf H}_n\bt_{1n}\right]_{jj}*\left[\bt_{1n}'{\bf H}_n\right]_{jl}(u)\\
&-\frac {72i}{n^3}\left[\bt_{2n}\right]_{kk}\left[{\bf H}_n\bt_{1n}\right]_{dj}*\left[\bt_{2n}\bw_n'\bt_{1n}'{\bf H}_n\bt_{1n}\right]_{kj}*\left[\bt_{2n}\bw_n'\bt_{1n}'{\bf H}_n\bt_{1n}\right]_{kj}*\left[\bt_{1n}'{\bf H}_n\right]_{jl}(u)\\
&-\frac {72i}{n^3}\left[\bt_{2n}\right]_{kk}\left[{\bf H}_n\bt_{1n}\bw_n\bt_{2n}\right]_{dk}*\left[\bt_{1n}'{\bf H}_n\bt_{1n}\right]_{jj}*\left[\bt_{2n}\bw_n'\bt_{1n}'{\bf H}_n\bt_{1n}\right]_{kj}*\left[\bt_{1n}'{\bf H}_n\right]_{jl}(u)\\
&-\frac {48i}{n^3}\left[\bt_{2n}\right]_{kk}\left[{\bf H}_n\bt_{1n}\right]_{dj}*\left[\bt_{2n}\bw_n'\bt_{1n}'{\bf H}_n\bt_{1n}\bw_n\bt_{2n}\right]_{kk}*\left[\bt_{1n}'{\bf H}_n\bt_{1n}\right]_{jj}*\left[\bt_{1n}'{\bf H}_n\right]_{jl}(u)\\
&-\frac{72i}{n^3}\left[\bt_{2n}\right]_{kk} \left[{\bf H}_n\bt_{1n}\right]_{dj}*\left[\bt_{2n}\bw_n'\bt_{1n}'{\bf H}_n\bt_{1n}\right]_{kj}*\left[\bt_{1n}'{\bf H}_n\bt_{1n}\right]_{jj}*\left[\bt_{2n}\bw_n'\bt_{1n}'{\bf H}_n\right]_{kl}(u)\\
&-\frac{24i}{n^3}\left[\bt_{2n}\right]_{kk}\left[{\bf H}_n\bt_{1n}\bw_n\bt_{2n}\right]_{dk}*\left[\bt_{1n}'{\bf H}_n\bt_{1n}\right]_{jj}*\left[\bt_{1n}'{\bf H}_n\bt_{1n}\right]_{jj}*\left[\bt_{2n}\bw_n'\bt_{1n}'{\bf H}_n\right]_{kl}(u)\\
&+\frac {24}{n^4}\left[{\bf H}_n\bt_{1n}\right]_{dj}*\left[\bt_{2n}\bw_n'\bt_{1n}'{\bf H}_n\bt_{1n}\right]_{kj}*\left[\bt_{2n}\bw_n'\bt_{1n}'{\bf H}_n\bt_{1n}\right]_{kj}*\left[\bt_{2n}\bw_n'\bt_{1n}'{\bf H}_n\bt_{1n}\right]_{kj}*\left[\bt_{2n}\bw_n'\bt_{1n}'{\bf H}_n\right]_{kl}(u)\\
&+\frac {72}{n^4}\left[{\bf H}_n\bt_{1n}\right]_{dj}*\left[\bt_{2n}\bw_n'\bt_{1n}'{\bf H}_n\bt_{1n}\right]_{kj}*\left[\bt_{2n}\bw_n'\bt_{1n}'{\bf H}_n\bt_{1n}\right]_{kj}*\left[\bt_{2n}\bw_n'\bt_{1n}'{\bf H}_n\bt_{1n}\bw_n\bt_{2n}\right]_{kk}*\left[\bt_{1n}'{\bf H}_n\right]_{jl}(u)\\
&+\frac {72}{n^4}\left[{\bf H}_n\bt_{1n}\bw_n\bt_{2n}\right]_{dk}*\left[\bt_{1n}'{\bf H}_n\bt_{1n}\right]_{jj}*\left[\bt_{2n}\bw_n'\bt_{1n}'{\bf H}_n\bt_{1n}\right]_{kj}*\left[\bt_{2n}\bw_n'\bt_{1n}'{\bf H}_n\bt_{1n}\right]_{kj}*\left[\bt_{2n}\bw_n'\bt_{1n}'{\bf H}_n\right]_{kl}(u)\\
&+\frac {72}{n^4}\left[{\bf H}_n\bt_{1n}\bw_n\bt_{2n}\right]_{dk}*\left[\bt_{1n}'{\bf H}_n\bt_{1n}\right]_{jj}*\left[\bt_{2n}\bw_n'\bt_{1n}'{\bf H}_n\bt_{1n}\right]_{kj}*\left[\bt_{2n}\bw_n'\bt_{1n}'{\bf H}_n\bt_{1n}\bw_n\bt_{2n}\right]_{kk}*\left[\bt_{1n}'{\bf H}_n\right]_{jl}(u)\\
&+\frac {72}{n^4}\left[{\bf H}_n\bt_{1n}\right]_{dj}*\left[\bt_{2n}\bw_n'\bt_{1n}'{\bf H}_n\bt_{1n}\bw_n\bt_{2n}\right]_{kk}*\left[\bt_{1n}'{\bf H}_n\bt_{1n}\right]_{jj}*\left[\bt_{2n}\bw_n'\bt_{1n}'{\bf H}_n\bt_{1n}\right]_{kj}*\left[\bt_{2n}\bw_n'\bt_{1n}'{\bf H}_n\right]_{kl}(u)\\
&+\frac {24}{n^4}\left[{\bf H}_n\bt_{1n}\right]_{dj}*\left[\bt_{2n}\bw_n'\bt_{1n}'{\bf H}_n\bt_{1n}\bw_n\bt_{2n}\right]_{kk}*\left[\bt_{1n}'{\bf H}_n\bt_{1n}\right]_{jj}*\left[\bt_{2n}\bw_n'\bt_{1n}'{\bf H}_n\bt_{1n}\bw_n\bt_{2n}\right]_{kk}*\left[\bt_{1n}'{\bf H}_n\right]_{jl}(u)\\
&+\frac {24}{n^4}\left[{\bf H}_n\bt_{1n}\bw_n\bt_{2n}\right]_{dk}*\left[\bt_{1n}'{\bf H}_n\bt_{1n}\right]_{jj}*\left[\bt_{2n}\bw_n'\bt_{1n}'{\bf H}_n\bt_{1n}\bw_n\bt_{2n}\right]_{kk}*\left[\bt_{1n}'{\bf H}_n\bt_{1n}\right]_{jj}*\left[\bt_{2n}\bw_n'\bt_{1n}'{\bf H}_n\right]_{kl}(u)\\
&+\frac {24}{n^4}\left[{\bf H}_n\bt_{1n}\bw_n\bt_{2n}\right]_{dk}*\left[\bt_{1n}'{\bf H}_n\bt_{1n}\bw_n\bt_{2n}\right]_{jk}*\left[\bt_{1n}'{\bf H}_n\bt_{1n}\bw_n\bt_{2n}\right]_{jk}*\left[\bt_{1n}'{\bf H}_n\bt_{1n}\bw_n\bt_{2n}\right]_{jk}*\left[\bt_{1n}'{\bf H}_n\right]_{jl}(u).
\end{align*}
\end{footnotesize}
\end{lemma}
\begin{proof}
Using Lemma \ref{le5} and Lemma \ref{le6}, it yields
\begin{footnotesize}
\begin{align*}
&\left[\frac{\partial^4 {\bf H}_n(u)}{\partial w_{jk}^4}\right]_{dl}
=\frac in\frac{\partial^3 \left[{\bf H}_n\bt_{1n}\right]_{dj}*\left[\bt_{2n}\bw_n'\bt_{1n}'{\bf H}_n\right]_{kl}(u)}{\partial w_{jk}^3}+\frac in\frac{\partial^3\left[{\bf H}_n\bt_{1n}\bw_n\bt_{2n}\right]_{dk}*\left[\bt_{1n}'{\bf H}_n\right]_{jl}(u)}{\partial w_{jk}^3}\\
=&\frac{2i}n\left[\bt_{2n}\right]_{kk}\frac{\partial^2 \left[{\bf H}_n\bt_{1n}\right]_{dj}*\left[\bt_{1n}'{\bf H}_n\right]_{jl}(u)}{\partial w_{jk}^2}
+\frac in\frac{\partial^2 \left[\frac{\partial{\bf H}_n}{\partial w}\bt_{1n}\right]_{dj}*\left[\bt_{2n}\bw_n'\bt_{1n}'{\bf H}_n\right]_{kl}(u)}{\partial w^2}\\
&+\frac in\frac{\partial^2 \left[{\bf H}_n\bt_{1n}\right]_{dj}*\left[\bt_{2n}\bw_n'\bt_{1n}'\frac{\partial{\bf H}_n}{\partial w}\right]_{kl}(u)}{\partial w^2}
+\frac in\frac{\partial^2\left[\frac{\partial{\bf H}_n}{\partial w}\bt_{1n}\bw_n\bt_{2n}\right]_{dk}*\left[\bt_{1n}'{\bf H}_n\right]_{jl}(u)}{\partial w^2}\\
&+\frac in\frac{\partial^2\left[{\bf H}_n\bt_{1n}\bw_n\bt_{2n}\right]_{dk}*\left[\bt_{1n}'\frac{\partial{\bf H}_n}{\partial w}\right]_{jl}(u)}{\partial w^2}\\
=&\frac{4i}n\left[\bt_{2n}\right]_{kk}\frac{\partial \left[\frac{\partial{\bf H}_n}{\partial w}\bt_{1n}\right]_{dj}*\left[\bt_{1n}'{\bf H}_n\right]_{jl}(u)}{\partial w}
+\frac{4i}n\left[\bt_{2n}\right]_{kk}\frac{\partial \left[{\bf H}_n\bt_{1n}\right]_{dj}*\left[\bt_{1n}'\frac{\partial{\bf H}_n}{\partial w}\right]_{jl}(u)}{\partial w}\\
&+\frac{2i}n\frac{\partial \left[\frac{\partial{\bf H}_n}{\partial w}\bt_{1n}\right]_{dj}*\left[\bt_{2n}\bw_n'\bt_{1n}'\frac{\partial{\bf H}_n}{\partial w}\right]_{kl}(u)}{\partial w}
+\frac{2i}n\frac{\partial\left[\frac{\partial{\bf H}_n}{\partial w}\bt_{1n}\bw_n\bt_{2n}\right]_{dk}*\left[\bt_{1n}'\frac{\partial{\bf H}_n}{\partial w}\right]_{jl}(u)}{\partial w}\\
&+\frac in\frac{\partial \left[\frac{\partial^2{\bf H}_n}{\partial w^2}\bt_{1n}\right]_{dj}*\left[\bt_{2n}\bw_n'\bt_{1n}'{\bf H}_n\right]_{kl}(u)}{\partial w}
+\frac in\frac{\partial \left[{\bf H}_n\bt_{1n}\right]_{dj}*\left[\bt_{2n}\bw_n'\bt_{1n}'\frac{\partial^2{\bf H}_n}{\partial w^2}\right]_{kl}(u)}{\partial w}\\
&+\frac in\frac{\partial\left[\frac{\partial^2{\bf H}_n}{\partial w^2}\bt_{1n}\bw_n\bt_{2n}\right]_{dk}*\left[\bt_{1n}'{\bf H}_n\right]_{jl}(u)}{\partial w}
+\frac in\frac{\partial\left[{\bf H}_n\bt_{1n}\bw_n\bt_{2n}\right]_{dk}*\left[\bt_{1n}'\frac{\partial^2{\bf H}_n}{\partial w^2}\right]_{jl}(u)}{\partial w}\\
=&\frac{6i}n\left[\bt_{2n}\right]_{kk} \left[\frac{\partial^2{\bf H}_n}{\partial w^2}\bt_{1n}\right]_{dj}*\left[\bt_{1n}'{\bf H}_n\right]_{jl}(u)
+\frac{12i}n\left[\bt_{2n}\right]_{kk} \left[\frac{\partial{\bf H}_n}{\partial w}\bt_{1n}\right]_{dj}*\left[\bt_{1n}'\frac{\partial{\bf H}_n}{\partial w}\right]_{jl}(u)\\
&+\frac{6i}n\left[\bt_{2n}\right]_{kk} \left[{\bf H}_n\bt_{1n}\right]_{dj}*\left[\bt_{1n}'\frac{\partial^2{\bf H}_n}{\partial w^2}\right]_{jl}(u)
+\frac{3i}n \left[\frac{\partial^2{\bf H}_n}{\partial w^2}\bt_{1n}\right]_{dj}*\left[\bt_{2n}\bw_n'\bt_{1n}'\frac{\partial{\bf H}_n}{\partial w}\right]_{kl}(u)\\
&+\frac{3i}n \left[\frac{\partial{\bf H}_n}{\partial w}\bt_{1n}\right]_{dj}*\left[\bt_{2n}\bw_n'\bt_{1n}'\frac{\partial^2{\bf H}_n}{\partial w^2}\right]_{kl}(u)
+\frac{3i}n\left[\frac{\partial^2{\bf H}_n}{\partial w}\bt_{1n}\bw_n\bt_{2n}\right]_{dk}*\left[\bt_{1n}'\frac{\partial{\bf H}_n}{\partial w}\right]_{jl}(u)\\
&+\frac{3i}n\left[\frac{\partial{\bf H}_n}{\partial w}\bt_{1n}\bw_n\bt_{2n}\right]_{dk}*\left[\bt_{1n}'\frac{\partial^2{\bf H}_n}{\partial w^2}\right]_{jl}(u)
+\frac in \left[\frac{\partial^3{\bf H}_n}{\partial w^3}\bt_{1n}\right]_{dj}*\left[\bt_{2n}\bw_n'\bt_{1n}'{\bf H}_n\right]_{kl}(u)\\
&+\frac in \left[{\bf H}_n\bt_{1n}\right]_{dj}*\left[\bt_{2n}\bw_n'\bt_{1n}'\frac{\partial^3{\bf H}_n}{\partial w^3}\right]_{kl}(u)
+\frac in\left[\frac{\partial^3{\bf H}_n}{\partial w^3}\bt_{1n}\bw_n\bt_{2n}\right]_{dk}*\left[\bt_{1n}'{\bf H}_n\right]_{jl}(u)\\
&+\frac in\left[{\bf H}_n\bt_{1n}\bw_n\bt_{2n}\right]_{dk}*\left[\bt_{1n}'\frac{\partial^3{\bf H}_n}{\partial w^3}\right]_{jl}(u)\\
=&-\frac{24}{n^2}\left[\bt_{2n}\right]_{kk}^2\left[{\bf H}_n\bt_{1n}\right]_{dj}*\left[\bt_{1n}'{\bf H}_n\bt_{1n}\right]_{jj}*\left[\bt_{1n}'{\bf H}_n\right]_{jl}(u)\\
&-\frac {72i}{n^3}\left[\bt_{2n}\right]_{kk}\left[{\bf H}_n\bt_{1n}\right]_{dj}*\left[\bt_{2n}\bw_n'\bt_{1n}'{\bf H}_n\bt_{1n}\right]_{kj}*\left[\bt_{2n}\bw_n'\bt_{1n}'{\bf H}_n\bt_{1n}\right]_{kj}*\left[\bt_{1n}'{\bf H}_n\right]_{jl}(u)\\
&-\frac {72i}{n^3}\left[\bt_{2n}\right]_{kk}\left[{\bf H}_n\bt_{1n}\bw_n\bt_{2n}\right]_{dk}*\left[\bt_{1n}'{\bf H}_n\bt_{1n}\right]_{jj}*\left[\bt_{2n}\bw_n'\bt_{1n}'{\bf H}_n\bt_{1n}\right]_{kj}*\left[\bt_{1n}'{\bf H}_n\right]_{jl}(u)\\
&-\frac {48i}{n^3}\left[\bt_{2n}\right]_{kk}\left[{\bf H}_n\bt_{1n}\right]_{dj}*\left[\bt_{2n}\bw_n'\bt_{1n}'{\bf H}_n\bt_{1n}\bw_n\bt_{2n}\right]_{kk}*\left[\bt_{1n}'{\bf H}_n\bt_{1n}\right]_{jj}*\left[\bt_{1n}'{\bf H}_n\right]_{jl}(u)\\
&-\frac{72i}{n^3}\left[\bt_{2n}\right]_{kk} \left[{\bf H}_n\bt_{1n}\right]_{dj}*\left[\bt_{2n}\bw_n'\bt_{1n}'{\bf H}_n\bt_{1n}\right]_{kj}*\left[\bt_{1n}'{\bf H}_n\bt_{1n}\right]_{jj}*\left[\bt_{2n}\bw_n'\bt_{1n}'{\bf H}_n\right]_{kl}(u)\\
&-\frac{24i}{n^3}\left[\bt_{2n}\right]_{kk}\left[{\bf H}_n\bt_{1n}\bw_n\bt_{2n}\right]_{dk}*\left[\bt_{1n}'{\bf H}_n\bt_{1n}\right]_{jj}*\left[\bt_{1n}'{\bf H}_n\bt_{1n}\right]_{jj}*\left[\bt_{2n}\bw_n'\bt_{1n}'{\bf H}_n\right]_{kl}(u)\\
&+\frac {24}{n^4}\left[{\bf H}_n\bt_{1n}\right]_{dj}*\left[\bt_{2n}\bw_n'\bt_{1n}'{\bf H}_n\bt_{1n}\right]_{kj}*\left[\bt_{2n}\bw_n'\bt_{1n}'{\bf H}_n\bt_{1n}\right]_{kj}*\left[\bt_{2n}\bw_n'\bt_{1n}'{\bf H}_n\bt_{1n}\right]_{kj}*\left[\bt_{2n}\bw_n'\bt_{1n}'{\bf H}_n\right]_{kl}(u)\\
&+\frac {72}{n^4}\left[{\bf H}_n\bt_{1n}\right]_{dj}*\left[\bt_{2n}\bw_n'\bt_{1n}'{\bf H}_n\bt_{1n}\right]_{kj}*\left[\bt_{2n}\bw_n'\bt_{1n}'{\bf H}_n\bt_{1n}\right]_{kj}*\left[\bt_{2n}\bw_n'\bt_{1n}'{\bf H}_n\bt_{1n}\bw_n\bt_{2n}\right]_{kk}*\left[\bt_{1n}'{\bf H}_n\right]_{jl}(u)\\
&+\frac {72}{n^4}\left[{\bf H}_n\bt_{1n}\bw_n\bt_{2n}\right]_{dk}*\left[\bt_{1n}'{\bf H}_n\bt_{1n}\right]_{jj}*\left[\bt_{2n}\bw_n'\bt_{1n}'{\bf H}_n\bt_{1n}\right]_{kj}*\left[\bt_{2n}\bw_n'\bt_{1n}'{\bf H}_n\bt_{1n}\right]_{kj}*\left[\bt_{2n}\bw_n'\bt_{1n}'{\bf H}_n\right]_{kl}(u)\\
&+\frac {72}{n^4}\left[{\bf H}_n\bt_{1n}\bw_n\bt_{2n}\right]_{dk}*\left[\bt_{1n}'{\bf H}_n\bt_{1n}\right]_{jj}*\left[\bt_{2n}\bw_n'\bt_{1n}'{\bf H}_n\bt_{1n}\right]_{kj}*\left[\bt_{2n}\bw_n'\bt_{1n}'{\bf H}_n\bt_{1n}\bw_n\bt_{2n}\right]_{kk}*\left[\bt_{1n}'{\bf H}_n\right]_{jl}(u)\\
&+\frac {72}{n^4}\left[{\bf H}_n\bt_{1n}\right]_{dj}*\left[\bt_{2n}\bw_n'\bt_{1n}'{\bf H}_n\bt_{1n}\bw_n\bt_{2n}\right]_{kk}*\left[\bt_{1n}'{\bf H}_n\bt_{1n}\right]_{jj}*\left[\bt_{2n}\bw_n'\bt_{1n}'{\bf H}_n\bt_{1n}\right]_{kj}*\left[\bt_{2n}\bw_n'\bt_{1n}'{\bf H}_n\right]_{kl}(u)\\
&+\frac {24}{n^4}\left[{\bf H}_n\bt_{1n}\right]_{dj}*\left[\bt_{2n}\bw_n'\bt_{1n}'{\bf H}_n\bt_{1n}\bw_n\bt_{2n}\right]_{kk}*\left[\bt_{1n}'{\bf H}_n\bt_{1n}\right]_{jj}*\left[\bt_{2n}\bw_n'\bt_{1n}'{\bf H}_n\bt_{1n}\bw_n\bt_{2n}\right]_{kk}*\left[\bt_{1n}'{\bf H}_n\right]_{jl}(u)\\
&+\frac {24}{n^4}\left[{\bf H}_n\bt_{1n}\bw_n\bt_{2n}\right]_{dk}*\left[\bt_{1n}'{\bf H}_n\bt_{1n}\right]_{jj}*\left[\bt_{2n}\bw_n'\bt_{1n}'{\bf H}_n\bt_{1n}\bw_n\bt_{2n}\right]_{kk}*\left[\bt_{1n}'{\bf H}_n\bt_{1n}\right]_{jj}*\left[\bt_{2n}\bw_n'\bt_{1n}'{\bf H}_n\right]_{kl}(u)\\
&+\frac {24}{n^4}\left[{\bf H}_n\bt_{1n}\bw_n\bt_{2n}\right]_{dk}*\left[\bt_{1n}'{\bf H}_n\bt_{1n}\bw_n\bt_{2n}\right]_{jk}*\left[\bt_{1n}'{\bf H}_n\bt_{1n}\bw_n\bt_{2n}\right]_{jk}*\left[\bt_{1n}'{\bf H}_n\bt_{1n}\bw_n\bt_{2n}\right]_{jk}*\left[\bt_{1n}'{\bf H}_n\right]_{jl}(u).
\end{align*}
\end{footnotesize}
\end{proof}
\begin{lemma}\label{le4}
Suppose ${\bf A},{\bf B},{\bf C},{\bf D},{\bf F}$ are random matrices. Then we get for $j=1,\cdots,m_1,k=1,\cdots,m_2$
\begin{align*}
&\left|\sum_{j,k}\re{\bf A}_{jj}{\bf B}_{jk}{\bf C}_{kk}\right|\le \left(\re\rtr{\bf A}\ba^*\right)^{1/2}\left(\re\rtr{\bf B}'{\bf B}{\bf B}'{\bf B}\right)^{1/4}\left[\re\left(\rtr{\bf C}{\bf C}^*\right)^2\right]^{1/4}
\end{align*}
and
\begin{align*}
&\left|\sum_{j,k}\re{\bf B}_{jk}{\bf D}_{jk}{\bf F}_{jk}\right|\le \re\|{\bf B}\|_2\re\|{\bf D}\|_2\left(\re\rtr{\bf B}{\bf B}^*\right)^{1/4}\left(\re\rtr{\bf D}{\bf D}^*\right)^{1/4}\left(\re\rtr{\bf F}{\bf F}^*\right)^{1/2}.
\end{align*}
\end{lemma}
\begin{proof}
Using the Cauchy-Schwarz inequality, we have
\begin{align*}
\left|\sum_{j,k}\re{\bf A}_{jj}{\bf B}_{jk}{\bf C}_{kk}\right|\le &
\left(\sum_{j}\re\left|{\bf A}_{jj}\right|^2\right)^{1/2}\left(\sum_{j}\re\left|\sum_{k}{\bf B}_{jk}{\bf C}_{kk}\right|^2\right)^{1/2}\\
\le&\left(\re\rtr{\bf A}\ba^*\right)^{1/2}\left(\sum_{k_1,k_2}\re\left({\bf B}'{\bf B}\right)_{k_1k_2}{\bf C}_{k_1k_1}{\bf C}_{k_2k_2}\right)^{1/2}\\
\le&\left(\re\rtr{\bf A}\ba^*\right)^{1/2}\left(\sum_{k_1,k_2}\re\left({\bf B}'{\bf B}\right)^2_{k_1k_2}\right)^{1/4}\left(\sum_{k_1,k_2}\re\left|{\bf C}_{k_1k_1}\right|^2\left|{\bf C}_{k_2k_2}\right|^2\right)^{1/4}\\
\le&\left(\re\rtr{\bf A}\ba^*\right)^{1/2}\left(\re\rtr{\bf B}'{\bf B}{\bf B}'{\bf B}\right)^{1/4}\left[\re\left(\rtr{\bf C}{\bf C}^*\right)^2\right]^{1/4}
\end{align*}
and
\begin{align*}
\left|\sum_{j,k}\re{\bf B}_{jk}{\bf D}_{jk}{\bf F}_{jk}\right|\le &\left(\sum_{j,k}\re\left|{\bf B}_{jk}\right|^2\left|{\bf D}_{jk}\right|^2\right)^{1/2}\left(\sum_{j,k}\re\left|{\bf F}_{jk}\right|^2\right)^{1/2}\\
\le &\left(\sum_{j,k}\re\left|{\bf B}_{jk}\right|^4\right)^{1/4}\left(\sum_{j,k}\re\left|{\bf D}_{jk}\right|^4\right)^{1/4}\left(\re\rtr{\bf F}{\bf F}^*\right)^{1/2}\\
\le &\left[\sum_{j}\left(\sum_k\re\left|{\bf B}_{jk}\right|^2\right)^2\right]^{1/4}\left[\sum_{j}\left(\sum_k\re\left|{\bf D}_{jk}\right|^2\right)^2\right]^{1/4}\left(\re\rtr{\bf F}{\bf F}^*\right)^{1/2}\\
\le &\re\|{\bf B}\|_2\re\|{\bf D}\|_2\left[\sum_{j,k}\re\left|{\bf B}_{jk}\right|^2\right]^{1/4}\left[\sum_{j,k}\re\left|{\bf D}_{jk}\right|^2\right]^{1/4}\left(\re\rtr{\bf F}{\bf F}^*\right)^{1/2}\\
\le&\re\|{\bf B}\|_2\re\|{\bf D}\|_2\left(\re\rtr{\bf B}{\bf B}^*\right)^{1/4}\left(\re\rtr{\bf D}{\bf D}^*\right)^{1/4}\left(\re\rtr{\bf F}{\bf F}^*\right)^{1/2}.
\end{align*}
\end{proof}

\begin{lemma}\label{le8}
Suppose ${\bf A},{\bf B},{\bf C},{\bf D}$ are random Hermitian matrices and ${\bf F},{\bf L},{\bf M}, {\bf Q}, {\bf R}$ are $m_1\times m_2$ random matrices. Then we get for $j=1,\cdots,m_1,k=1,\cdots,m_2$
\begin{align*}
\sum_{j,k}\re\left|{\bf A}_{jj}{\bf B}_{jj}{\bf C}_{kk}{\bf D}_{kk}{\bf F}_{jk}\right|\le&\bigg(\min\left\{{\rm rank}(\ba),{\rm rank (\bb)}\right\}\min\left\{{\rm rank}({\bf C}),{\rm rank ({\bf D})}\right\}\bigg)^{1/2}\\
&\times\left(\re\left\|{\bf A}{\bf B}{\bf C}{\bf D}\right\|_2^2\right)^{1/2}\left(\re\rtr{\bf F}{\bf F}^*\right)^{1/2},\\
\sum_{j,k}\re\left|{\bf A}_{jj}{\bf C}_{kk}{\bf F}_{jk}{\bf L}_{jk}{\bf M}_{jk}\right|
      \le&\left({\rm rank}(\ba){\rm rank}({\bf C})\re\left\|{\bf A}{\bf C}\right\|_2^2\right)^{1/2}\re\|{\bf F}\|_2\re\|{\bf L}\|_2\re\|{\bf M}\|_2\\
&\times\left(\re\rtr{\bf F}{\bf F}^*\right)^{1/4}\left(\re\rtr{\bf L}{\bf L}^*\right)^{1/8}\left(\re\rtr{\bf M}{\bf M}^*\right)^{1/8},
\end{align*}
and
\begin{align*}
\sum_{j,k}\re\left|{\bf F}_{jk}{\bf L}_{jk}{\bf M}_{jk}{\bf Q}_{jk}{\bf R}_{jk}\right|
     \le&\re\left\|{\bf F}\right\|_2\re\|{\bf L}\|_2\re\|{\bf M}\|_2\re\|{\bf Q}\|_2\re\|{\bf R}\|_2\left(\re\rtr{\bf F}{\bf F}^*\right)^{1/4}\left(\re\rtr{\bf L}{\bf L}^*\right)^{1/4}\\
&\qquad\times\left(\re\rtr{\bf M}{\bf M}^*\right)^{1/4}\left(\re\rtr{\bf Q}{\bf Q}^*\right)^{1/8}\left(\re\rtr{\bf R}{\bf R}^*\right)^{1/8}.
\end{align*}
\end{lemma}
\begin{proof}
Applying Cauchy-Schwarz inequality and Lemma \ref{lel2}, it implies that
\begin{align*}
&\sum_{j,k}\re\left|{\bf A}_{jj}{\bf B}_{jj}{\bf C}_{kk}{\bf D}_{kk}{\bf F}_{jk}\right|\le\left(\sum_{j,k}\re\left|{\bf A}_{jj}{\bf B}_{jj}{\bf C}_{kk}{\bf D}_{kk}\right|^2\right)^{1/2}\left(\sum_{j,k}\re\left|{\bf F}_{jk}\right|^2\right)^{1/2}\\
=&\bigg(\re\rtr\left({\bf A}\circ{\bf B}\circ\ba^*\circ\bb^*\right)\rtr\left({\bf C}\circ{\bf D}\circ{\bf C}^*\circ{\bf D}^*\right)\bigg)^{1/2}\left(\re\rtr{\bf F}{\bf F}^*\right)^{1/2}\\
\le&\left(\min\left\{{\rm rank}(\ba),{\rm rank (\bb)}\right\}\min\left\{{\rm rank}({\bf C}),{\rm rank ({\bf D})}\right\}\re\left\|{\bf A}{\bf B}{\bf C}{\bf D}\right\|_2^2\right)^{1/2}\left(\re\rtr{\bf F}{\bf F}^*\right)^{1/2}\\
\le&\bigg(\min\left\{{\rm rank}(\ba),{\rm rank (\bb)}\right\}\min\left\{{\rm rank}({\bf C}),{\rm rank ({\bf D})}\right\}\bigg)^{1/2}\left(\re\left\|{\bf A}{\bf B}{\bf C}{\bf D}\right\|_2^2\right)^{1/2}\left(\re\rtr{\bf F}{\bf F}^*\right)^{1/2}
\end{align*}
and
\begin{align*}
&\sum_{j,k}\re\left|{\bf A}_{jj}{\bf C}_{kk}{\bf F}_{jk}{\bf L}_{jk}{\bf M}_{jk}\right|\le\left(\sum_{j,k}\re\left|{\bf A}_{jj}{\bf C}_{kk}\right|^2\right)^{1/2}\left(\sum_{j,k}\re\left|{\bf F}_{jk}{\bf L}_{jk}{\bf M}_{jk}\right|^2\right)^{1/2}\\
=&\bigg(\re\rtr\left({\bf A}\circ\ba^*\right)\rtr\left({\bf C}\circ{\bf C}^*\right)\bigg)^{1/2}\left(\sum_{j,k}\re\left|{\bf F}_{jk}{\bf L}_{jk}{\bf M}_{jk}\right|^2\right)^{1/2}\\
\le&\left({\rm rank}(\ba){\rm rank}({\bf C})\re\left\|{\bf A}{\bf C}\right\|_2^2\right)^{1/2}\left(\sum_{j,k}\re\left|{\bf F}_{jk}\right|^4\right)^{1/4}\left(\sum_{j,k}\re\left|{\bf L}_{jk}\right|^8\right)^{1/8}\left(\sum_{j,k}\re\left|{\bf M}_{jk}\right|^8\right)^{1/8}\\
\le&\left({\rm rank}(\ba){\rm rank}({\bf C})\re\left\|{\bf A}{\bf C}\right\|_2^2\right)^{1/2}\left[\sum_{j}\left(\sum_{k}\re\left|{\bf F}_{jk}\right|^2\right)^{2}\right]^{1/4}\left[\sum_{j}\left(\sum_{k}\re\left|{\bf L}_{jk}\right|^2\right)^{4}\right]^{1/8}\\
&\qquad\times\left[\sum_{j}\left(\sum_{k}\re\left|{\bf M}_{jk}\right|^2\right)^{4}\right]^{1/8}\\
\le&\left({\rm rank}(\ba){\rm rank}({\bf C})\re\left\|{\bf A}{\bf C}\right\|_2^2\right)^{1/2}\re\|{\bf F}\|_2\re\|{\bf L}\|_2\re\|{\bf M}\|_2\left[\sum_{j,k}\re\left|{\bf F}_{jk}\right|^2\right]^{1/4}\left[\sum_{j,k}\re\left|{\bf L}_{jk}\right|^2\right]^{1/8}\\
&\qquad\times\left[\sum_{j,k}\re\left|{\bf M}_{jk}\right|^2\right]^{1/8}\\
\le&\left({\rm rank}(\ba){\rm rank}({\bf C})\re\left\|{\bf A}{\bf C}\right\|_2^2\right)^{1/2}\re\|{\bf F}\|_2\re\|{\bf L}\|_2\re\|{\bf M}\|_2\left(\re\rtr{\bf F}{\bf F}^*\right)^{1/4}\left(\re\rtr{\bf L}{\bf L}^*\right)^{1/8}\\
&\qquad\times\left(\re\rtr{\bf M}{\bf M}^*\right)^{1/8}.
\end{align*}
Here, $\circ$ means Hadamard product. Furthermore, we have
\begin{align*}
&\sum_{j,k}\re\left|{\bf F}_{jk}{\bf L}_{jk}{\bf M}_{jk}{\bf Q}_{jk}{\bf R}_{jk}\right|\le\left(\sum_{j,k}\re\left|{\bf F}_{jk}{\bf L}_{jk}\right|^2\right)^{1/2}\left(\sum_{j,k}\re\left|{\bf M}_{jk}{\bf Q}_{jk}{\bf R}_{jk}\right|^2\right)^{1/2}\\
\le&\left(\sum_{j,k}\re\left|{\bf F}\right|_{jk}^4\right)^{1/4}\left(\sum_{j,k}\re\left|{\bf L}_{jk}\right|^4\right)^{1/4}\left(\sum_{j,k}\re\left|{\bf M}_{jk}\right|^4\right)^{1/4}\left(\sum_{j,k}\re\left|{\bf Q}_{jk}\right|^8\right)^{1/8}\left(\sum_{j,k}\re\left|{\bf R}_{jk}\right|^8\right)^{1/8}\\
\le&\re\left\|{\bf F}\right\|_2\re\|{\bf L}\|_2\re\|{\bf M}\|_2\re\|{\bf Q}\|_2\re\|{\bf R}\|_2\left(\sum_{j,k}\re\left|{\bf F}\right|_{jk}^2\right)^{1/4}\left(\sum_{j,k}\re\left|{\bf L}_{jk}\right|^2\right)^{1/4}\left(\sum_{j,k}\re\left|{\bf M}_{jk}\right|^2\right)^{1/4}\\
&\qquad\times\left(\sum_{j,k}\re\left|{\bf Q}_{jk}\right|^2\right)^{1/8}\left(\sum_{j,k}\re\left|{\bf R}_{jk}\right|^2\right)^{1/8}\\
\le&\re\left\|{\bf F}\right\|_2\re\|{\bf L}\|_2\re\|{\bf M}\|_2\re\|{\bf Q}\|_2\re\|{\bf R}\|_2\left(\re\rtr{\bf F}{\bf F}^*\right)^{1/4}\left(\re\rtr{\bf L}{\bf L}^*\right)^{1/4}\left(\re\rtr{\bf M}{\bf M}^*\right)^{1/4}\\
&\qquad\times\left(\re\rtr{\bf Q}{\bf Q}^*\right)^{1/8}\left(\re\rtr{\bf R}{\bf R}^*\right)^{1/8}.
\end{align*}
 This completes the proof of the lemma.
\end{proof}

\begin{lemma}\label{ble3}
Let $\ba=(a_{jk})$ be a $p\times p$ nonrandom matrix and $\bbx=(x_1,\cdots,x_{m_1})'$ be a random vector of independent entries. Assume that $\re x_j=0$, $\re|x_j|^2=1$,  $\sup_{j}\re|x_j|^{6}\le M$, and $|x_j|\le \eta_n\sqrt n/t_{\cdot j}$, $p/n\to c\in(0,\infty)$. Here $t_{\cdot j}$ is defined in Section 2. Then for any $l>3$, as $n\to \infty$
\begin{align*}
\re|\bbx^*\bt_{1n}^*\ba\bt_{1n}\bbx-\rtr\ba\bS_1|^l\le C_l \eta_n^{2l-6}n^{l-1}\|\ba\|^l
\end{align*}
where $C_l$ is a constant depending on $l$ only and ${\bS_1}=\bt_{1n}\bt_{1n}^*$.
\end{lemma}
\begin{proof}
Let $\bh=(h_{jk})=\bt_{1n}^*\ba\bt_{1n}$, we have
\begin{align*}
\bbx^*\bh\bbx-\rtr\bh=\sum_{j=1}^{m_1}h_{jj}\left(|x_j|^2-1\right)+\sum_{j=1}^{m_1}\sum_{k=1}^{j-1}\left(h_{kj}\bar x_kx_j+h_{jk}\bar x_jx_k\right).
\end{align*}
At first, we deduce
\begin{align*}
|h_{jk}|=\left|\be_j'\bt_{1n}^*\ba\bt_{1n}\be_k\right|\le\|\ba\|\sqrt{\be_j'\bt_{1n}^*\bt_{1n}\be_j}\sqrt{\be_k'\bt_{1n}^*\bt_{1n}\be_k}=t_{\cdot j}t_{\cdot k}\|\ba\|
\end{align*}
where $\be_j$ is a vector with the $j$-th element $1$ and the remaining elements zero.

Now, assume $1<l\le 3$. By Lemma \ref{mle4} and Lemma \ref{mle5}, we have
\begin{align*}
&\re\left|\sum_{j=1}^{m_1}h_{jj}\left(|x_j|^2-1\right)\right|^l\le C\re\left[\sum_{j=1}^{m_1}|h_{jj}|^2\left(|x_j|^2-1\right)^2\right]^{l/2}\\
\le& C\sum_{j=1}^{m_1}|h_{jj}|^l\re\left||x_j|^2-1\right|^{l}\le C\sum_{j=1}^{m_1}t_{\cdot j}^{2l}\|\ba\|^{l}\le C\sum_{j=1}^{m_1}t_{\cdot j}^{2}\|\ba\|^l\le Cn\|\ba\|^l.
\end{align*}
Furthermore, by the Holder inequality,
\begin{align*}
&\re\left|\sum_{j=1}^{m_1}\sum_{k=1}^{j-1}\left(h_{kj}\bar x_kx_j+h_{jk}\bar x_jx_k\right)\right|^l
\le C\left[\re\left|\sum_{j=1}^{m_1}\sum_{k=1}^{j-1}\left(h_{kj}\bar x_kx_j+h_{jk}\bar x_jx_k\right)\right|^2\right]^{l/2}\\
\le&C\left[\sum_{j=1}^{m_1}\sum_{k=1}^{j-1}\left(|h_{kj}|^2+|h_{jk}|^2\right)\right]^{l/2}\le C\left[\rtr\bh\bh^*\right]^{l/2}\le Cn^{l/2}\|\ba\|^l.
\end{align*}
Combining the two inequalities above, we obtain for $1<l\le 3$
\begin{align}\label{mal10}
\re|\bbx^*\bh\bbx-\rtr\bh\bh^*|^l\le C\left(n+n^{l/2}\right)\|\ba\|^l.
\end{align}

We shall proceed with the proof of the lemma by induction on $l$. And consider the case $3<l\le9$. Using Lemma \ref{mle6} and Lemma \ref{mle5},
\begin{align*}
&\re\left|\sum_{j=1}^{m_1}h_{jj}\left(|x_j|^2-1\right)\right|^l\\
\le &C\left[\left(\sum_{j=1}^{m_1}|h_{jj}|^2\re\left(|x_j|^2-1\right)^2\right)^{{l/2}}+\sum_{j=1}^{m_1}|h_{jj}|^l\re\left(|x_j|^2-1\right)^l\right]\\
\le &C\left[\left(\rtr\bh\bh^*\right)^{{l/2}}+\sum_{j=1}^{m_1}t_{\cdot j}^{2l}\|\ba\|^l\re|x_j|^{2l}\right]\le C\left[n^{l/2}\|\ba\|^{{l}}+\sum_{j=1}^{m_1}t_{\cdot j}^{2l}\|\ba\|^l\frac{\eta_n^{2l-6}n^{l-3}}{t_{\cdot j}^{2l-6}}\re|x_j|^6\right]\\
\le &C\left[n^{l/2}\|\ba\|^{{l}}+{\eta_n^{2l-6}n^{l-3}}\sum_{j=1}^{m_1}t_{\cdot j}^{6}\|\ba\|^l\right]\le C\left[n^{l/2}\|\ba\|^{{l}}+\eta_n^{2l-6}n^{l-2}\|\ba\|^l\right]\le
C\eta_n^{2l-6}n^{l-1}\|\ba\|^l.
\end{align*}
For the same reason, with notation $\re_j(\cdot)$ for the conditional expectation given $\{x_1,\cdots,x_j\}$, we have
\begin{align*}
&\re\left|\sum_{j=1}^{m_1}\sum_{k=1}^{j-1}h_{kj}\bar x_kx_j\right|^l\\
\le &C\left[\re\left(\sum_{j=1}^{m_1}\re_{j-1}\left|\sum_{k=1}^{j-1}h_{kj}\bar x_kx_j\right|^2\right)^{l/2}+\sum_{j=1}^{m_1}\re\left|\sum_{k=1}^{j-1}h_{kj}\bar x_kx_j\right|^l\right]\\
\le &C\left[\re\left(\sum_{j=1}^{m_1}\left|\sum_{k=1}^{j-1}h_{kj}\bar x_k\right|^2\right)^{l/2}+\sum_{j=1}^{m_1}\re\left|\sum_{k=1}^{j-1}h_{kj}\bar x_k\right|^l\right]\\
\le &C\left[\re\left(\sum_{j=1}^{m_1}\left|\re_{j-1}\sum_{k=1}^{m_1}h_{kj}\bar x_k\right|^2\right)^{l/2}+\sum_{j=1}^{m_1}\left(\sum_{k=1}^{j-1}|h_{kj}|^2\right)^{l/2}
+\sum_{j=1}^{m_1}\sum_{k=1}^{j-1}\left|h_{kj}\right|^l\re|x_k|^l\right]\\
\le &C\left[\re\left(\sum_{j=1}^{m_1}\left|\sum_{k=1}^{m_1}h_{kj}\bar x_k\right|^2\right)^{l/2}+\sum_{j=1}^{m_1}\left((\bh^*\bh)_{jj}\right)^{l/2}
+\sum_{j=1}^{m_1}\sum_{k=1}^{j-1}t_{\cdot j}^lt_{\cdot k}^l\|\ba\|^l\frac{\eta_n^{l-3}n^{l/2-3/2}}{t_{\cdot k}^{l-3}}\right]\\
\le &C\left[\re\left(\bbx^*\bh\bh^*\bbx\right)^{l/2}+\rtr\left(\bh^*\bh\right)^{l/2}+{\eta_n^{l-3}n^{l/2+1/2}}\|\ba\|^l\right]\\
\le &C\left[(n+n^{l/4})\|\ba\|^{l}+{\eta_n^{2l-6}n^{l-1}}\|\ba\|^l\right]\le C{\eta_n^{2l-6}n^{l-1}}\|\ba\|^l.
\end{align*}
The last inequality is from (\ref{mal10}) with $\bh$ replaced by $\bh\bh^*$.
Together with the two inequalities above, we conclude for $3<l\le 9$
\begin{align*}
\re|\bbx^*\bh\bbx-\rtr\bh\bh^*|^l\le C{\eta_n^{2l-6}n^{l-1}}\|\ba\|^l.
\end{align*}

In the following, consider the case $3^{\theta}<l\le3^{\theta+1}$ with $\theta\ge2$. Likewise, using Lemma \ref{mle6} and Lemma \ref{mle5}, we deduce
\begin{align*}
&\re\left|\sum_{j=1}^{m_1}h_{jj}\left(|x_j|^2-1\right)\right|^l\\
\le &C_l\left[\left(\sum_{j=1}^{m_1}|h_{jj}|^2\re\left(|x_j|^2-1\right)^2\right)^{{l/2}}+\sum_{j=1}^{m_1}|h_{jj}|^l\re\left(|x_j|^2-1\right)^l\right]\\
\le &C_l\left[\left(\rtr\bh\bh^*\right)^{{l/2}}+\sum_{j=1}^{m_1}t_{\cdot j}^{2l}\|\ba\|^l\re|x_j|^{2l}\right]\le C_l\left[n^{l/2}\|\ba\|^{{l}}+\sum_{j=1}^{m_1}t_{\cdot j}^{2l}\|\ba\|^l\frac{\eta_n^{2l-6}n^{l-3}}{t_{\cdot j}^{2l-6}}\right]\\
\le &C_l\left[n^{l/2}\|\ba\|^{{l}}+\eta_n^{2l-6}n^{l-3}\sum_{j=1}^{m_1}t_{\cdot j}^{6}\|\ba\|^l\right]\le C\left[n^{l/2}\|\ba\|^{{l}}+\eta_n^{2l-6}n^{l-2}\|\ba\|^l\right]
\le C_l\eta_n^{2l-6}n^{l-1}\|\ba\|^l.
\end{align*}
and
\begin{align*}
&\re\left|\sum_{j=1}^{m_1}\sum_{k=1}^{j-1}h_{kj}\bar x_kx_j\right|^l\\
\le &C_l\left[\re\left(\sum_{j=1}^{m_1}\re_{j-1}\left|\sum_{k=1}^{j-1}h_{kj}\bar x_kx_j\right|^2\right)^{l/2}+\sum_{j=1}^{m_1}\re\left|\sum_{k=1}^{j-1}h_{kj}\bar x_kx_j\right|^l\right]\\
\le &C_l\left[\re\left(\sum_{j=1}^{m_1}\left|\sum_{k=1}^{j-1}h_{kj}\bar x_k\right|^2\right)^{l/2}+\sum_{j=1}^{m_1}\re|x_j|^l\re\left|\sum_{k=1}^{j-1}h_{kj}\bar x_k\right|^l\right]\\
\le &C_l\left[\re\left(\sum_{j=1}^{m_1}\left|\re_{j-1}\sum_{k=1}^{m_1}h_{kj}\bar x_k\right|^2\right)^{l/2}+\sum_{j=1}^{m_1}\re|x_j|^l\left(\sum_{k=1}^{j-1}|h_{kj}|^2\right)^{l/2}
+\sum_{j=1}^{m_1}\re|x_j|^l\sum_{k=1}^{j-1}\left|h_{kj}\right|^l\re|x_k|^l\right]\\
\le &C_l\left[\re\left(\sum_{j=1}^{m_1}\left|\sum_{k=1}^{m_1}h_{kj}\bar x_k\right|^2\right)^{l/2}+\sum_{j=1}^{m_1}\re|x_j|^lt_{\cdot j}^l\left(\sum_{k=1}^{m_1}t_{\cdot k}^2\right)^{l/2}\|\ba\|^l
+\sum_{j=1}^{m_1}\re|x_j|^{l}t_{\cdot j}^l\sum_{k=1}^{j-1}\re|x_k|^{l}t_{\cdot k}^l\|\ba\|^l\right]\\
\le &C_l\left[\re\left(\sum_{j=1}^{m_1}\left|\sum_{k=1}^{m_1}h_{kj}\bar x_k\right|^2\right)^{l/2}+n^{l/2}\sum_{j=1}^{m_1}\re|x_j|^lt_{\cdot j}^l\|\ba\|^l
+\left(\sum_{j=1}^{m_1}\re|x_j|^{l}t_{\cdot j}^l\right)^2\|\ba\|^l\right]\\
\le &C_l\left[\re\left(\sum_{j=1}^{m_1}\left|\sum_{k=1}^{m_1}h_{kj}\bar x_k\right|^2\right)^{l/2}+n^{l/2}\sum_{j=1}^{m_1}\frac{\eta_n^{l-6}n^{l/2-3}}{t_{\cdot j}^{l-6}}t_{\cdot j}^l\|\ba\|^l
+\left(\sum_{j=1}^{m_1}\frac{\eta_n^{l-6}n^{l/2-3}}{t_{\cdot j}^{l-6}}t_{\cdot j}^l\right)^2\|\ba\|^l\right]\\
\le &C_l\left[\re\left(\sum_{j=1}^{m_1}\left|\sum_{k=1}^{m_1}h_{kj}\bar x_k\right|^2\right)^{l/2}+{\eta_n^{l-6}n^{l-3}}\sum_{j=1}^{m_1}t_{\cdot j}^6\|\ba\|^l
+\left({\eta_n^{l-6}n^{l/2-3}}\sum_{j=1}^{m_1}t_{\cdot j}^6\right)^2\|\ba\|^l\right]\\
\le &C_l\left[\re\left(\bbx^*\bh\bh^*\bbx\right)^{l/2}+{\eta_n^{l-6}n^{l-2}}\|\ba\|^l
+{\eta_n^{2l-12}n^{l-4}}\|\ba\|^l\right]\\
\le& C_l\left[\re\left(\bbx^*\bh\bh^*\bbx\right)^{l/2}+{\eta_n^{2l-6}n^{l-1}}\|\ba\|^l\right].
\end{align*}
Using the induction hypothesis with $\bh$ replaced by $\bh\bh^*$, it follows that
\begin{align*}
\re\left(\bbx^*\bh\bh^*\bbx\right)^{l/2}\le &C_l\left[\re\left|\bbx^*\bh\bh^*\bbx-\rtr\bh\bh^*\right|^{l/2}+\left(\rtr\bh\bh^*\right)^{l/2}\right]\\
\le&C_l\left[\eta_n^{l-6}n^{l/2-1}\|\ba\|^l+n^{l/2}\|\ba\|^{l}\right]\le C_ln^{l/2}\|\ba\|^l.
\end{align*}
Consequently, we get
\begin{align*}
\re|\bbx^*\bh\bbx-\rtr\bh\bh^*|^l\le C{\eta_n^{2l-6}n^{l-1}}\|\ba\|^l.
\end{align*}
\end{proof}

\subsubsection{Existing Lemmas}

\begin{lemma}[Theorem A.37 in Bai and Silverstein (2010) \cite{bai2010spectral}]\label{lel1}
If $\ba$ and $\bb$ are two $n\times p$ matrices and $\lambda_k$ and $\delta_k,k=1,2,\cdots,n$, denote their singular values, then
\begin{align*}
  \min_{\pi}\sum_{k=1}^n|\lambda_k-\delta_{\pi(k)}|^2\le\rtr\left[(\ba-\bb)(\ba-\bb)^*\right]\le\max_{\pi}\sum_{k=1}^n|\lambda_k-\delta_{\pi(k)}|^2.
\end{align*}
If the singular values are arranged in descending order, then we have
\begin{align*}
    \sum_{k=1}^{\nu}|\lambda_k-\delta_{k}|^2\le\rtr\left[(\ba-\bb)(\ba-\bb)^*\right],
\end{align*}
where $\nu=\min\{p,n\}$.
\end{lemma}

\begin{lemma}[Duhamel formula]\label{lea1}
Let ${\bf M}_1,{\bf M}_2$ be $n\times n$ matrices and $t\in\mathbb{R}$. Then we have
\begin{align*}
e^{({\bf M}_1+{\bf M}_2)t}=e^{{\bf M}_1t}+\int_0^te^{{\bf M}_1(t-s)}{\bf M}_2e^{({\bf M}_1+{\bf M}_2)s}ds.
\end{align*}
Moreover, if ${\bf A}(t)$ is a matrix-valued function of $t\in\mathbb{R}$ that is $C^{\infty}$ in the sense that each matrix element $\left[{\bf A}(t)\right]_{jk}$ is $C^{\infty}$. Then
\begin{align*}
\frac{de^{{\bf A}(t)}}{dt}=\int_0^1e^{s{\bf A}(t)}{\bf A}'(t)e^{(1-s){\bf A}(t)}ds.
\end{align*}
\end{lemma}

\begin{lemma}[Corollary A.22 in Bai and Silverstein (2010) \cite{bai2010spectral}]\label{lel2}
Suppose that $\ba_j,j=1,\cdots,l$ are $l$ $p\times p$ Hermitian matrices whose eigenvalues are bounded by $M_j$. Then
\begin{align*}
\left|\rtr\left(\ba_1\circ\cdots\circ\ba_l\right)\right|\le \min\{{\rm rank}(\ba_1),\cdots,{\rm rank}(\ba_l)\}M_1\cdots M_l.
\end{align*}
\end{lemma}

\begin{lemma}[Burkholder (1973) \cite{burkholder1973distribution}]\label{mle4}
Let $\{X_k\}$ be a complex martingale difference sequence with respect to the increasing $\sigma$-field $\{\mathcal{F}_k\}$. Then, for $l>1$,
\begin{align*}
\re\left|\sum X_k\right|^l\le C_l\re\left(\sum|X_k|^2\right)^{l/2}.
\end{align*}
\end{lemma}

\begin{lemma}[Burkholder (1973) \cite{burkholder1973distribution}]\label{mle6}
Let $\{X_k\}$ be a complex martingale difference sequence with respect to the increasing $\sigma$-field $\{\mathcal{F}_k\}$. Then, for $l\ge2$,
\begin{align*}
\re\left|\sum X_k\right|^l\le C_l\left[\re\left(\re\left(\sum|X_k|^2|\mathcal{F}_{k-1}\right)^{l/2}\right)+\sum\re|X_k|^l\right].
\end{align*}
\end{lemma}

\begin{lemma}[(3.3.41) of Horn and Johnson (1991) \cite{Horn1991Topics}]\label{mle5}
For $n\times n$ Hermitian $\ba=(a_{jk})$ with eigenvalues $\lambda_1,\cdots,\lambda_n$, and convex function $f(\cdot)$, we have
\begin{align*}
\sum_{j=1}^nf(a_{jj})\le\sum_{j=1}^nf(\lambda_j).
\end{align*}
\end{lemma}

\begin{lemma}[Lemma 2.4 in \cite{bai2004clt}]\label{lep2}
Suppose for each $n$ $Y_{n1},Y_{n2},\cdots,Y_{nr_n}$ is a real martingale difference sequence with respect to the increasing $\sigma$-field $\{\mathcal{F}_{nj}\}$ having second moments. If as $n\to\infty$,
\begin{align*}
&(i)\quad\quad\quad\quad\qquad\qquad\sum_{j=1}^{r_n}\re(Y_{nj}^2|\mathcal{F}_{n,j-1})\xrightarrow{i.p.}\sigma^2,\qquad\qquad\qquad\qquad\qquad\qquad\\
&{\rm where}\ \sigma^2\ {\rm is\ a\ positive\ constant}, \ {\rm and\ for\ each}\ \varepsilon\ge0,\\
&(ii)\quad\quad\quad\quad\qquad\qquad\sum_{j=1}^{r_n}\re(Y_{nj}^2I(|Y_{nj}\ge\varepsilon|))\to0,\qquad\qquad\qquad\qquad\qquad\qquad\\
&{\rm then}\\
&\quad\quad\quad\quad\qquad\qquad\qquad\sum_{j=1}^{r_n}Y_{nj}\xrightarrow{D}N(0,\sigma^2).\qquad\qquad\qquad\qquad
\end{align*}
\end{lemma}

\begin{lemma}[Lemma B.26 in \cite{bai2004clt}]\label{lep1}
Let ${\bf A}=(a_{jk})$ be an $n\times n$ nonrandom matrix and $\bx=(x_1,\cdots,x_n)'$ be a random vector of independent entries. Assume that $\re x_j=0$, $\re|x_j|^2=1$ and $\re|x_j|^l\le \nu_l$. Then for $p\ge1$,
\begin{align*}
&\re\left|\bx^*{\bf A}\bx-\rtr{\bf A}\right|^p
\le C_p\left[\left(\nu_4\rtr{\bf A}{\bf A}^*\right)^{p/2}+\nu_{2p}\rtr\left({\bf A}{\bf A}^*\right)^{p/2}\right]
\end{align*}
where $C_p$ is a constant depending on $p$ only.
\end{lemma}

\begin{lemma}[Inequality (4.8) in \cite{bai2004clt}]\label{lep6}
Let ${\bf M}$ be $N\times N$ nonrandom matrix, we find for $j\in\left\{1,2,\cdots,n\right\}$
\begin{align*}
\re\left|\rtr\bd_j^{-1}{\bf M}-\re\rtr\bd_j^{-1}{\bf M}\right|^2\le C\|{\bf M}\|^2
\end{align*}
\end{lemma}


\begin{thebibliography}{34}
\providecommand{\natexlab}[1]{#1}
\providecommand{\url}[1]{\texttt{#1}}
\expandafter\ifx\csname urlstyle\endcsname\relax
  \providecommand{\doi}[1]{doi: #1}\else
  \providecommand{\doi}{doi: \begingroup \urlstyle{rm}\Url}\fi

\bibitem[Anderson and Zeitouni(2006)]{AndersonZ06C}
Greg~W. Anderson and Ofer Zeitouni.
\newblock {A CLT for a band matrix model}.
\newblock \emph{Probability Theory and Related Fields}, 134\penalty0
  (2):\penalty0 283--338, 2006.
\newblock ISSN 01788051.
\newblock \doi{10.1007/s00440-004-0422-3}.
\newblock URL \url{http://www.springerlink.com/index/7EGB49JVCJ3X8FBR.pdf}.

\bibitem[Anderson(1983)]{Anderson1983An}
T.~W. Anderson.
\newblock \emph{An Introduction to Multivariate Statistical Analysis, Second
  Edition, Wiley, New York}.
\newblock 1983.

\bibitem[Bai and Yin(1993)]{BaiYin1993}
Z.~D. Bai and Y.~Q. Yin.
\newblock Limit of the smallest eigenvalue of a large dimensional sample
  covariance matrix.
\newblock \emph{The Annals of Probability}, 21\penalty0 (3):\penalty0 pp.
  1275--1294, 1993.
\newblock ISSN 00911798.
\newblock URL \url{http://www.jstor.org/stable/2244575}.

\bibitem[Bai and Yin(1988)]{bai1988necessary}
ZD~Bai and YQ~Yin.
\newblock Necessary and sufficient conditions for almost sure convergence of
  the largest eigenvalue of a wigner matrix.
\newblock \emph{The Annals of Probability}, 16\penalty0 (4):\penalty0
  1729--1741, 1988.

\bibitem[Bai and Silverstein(2010)]{bai2010spectral}
Zhi~Dong Bai and Jack~William Silverstein.
\newblock \emph{Spectral analysis of large dimensional random matrices}.
\newblock Springer, 2010.

\bibitem[Bai et~al.(2018)Bai, Li, and Pan]{liclt}
Zhi~Dong Bai, Huiqin Li, and Guang~Ming Pan.
\newblock Central limit theorem for linear spectral statistics of large
  dimensional separable sample covariance matrices.
\newblock 2018.

\bibitem[Bai and Silverstein(2004)]{bai2004clt}
Zhidong Bai and Jack~William Silverstein.
\newblock Clt for linear spectral statistics of large-dimensional sample
  covariance matrices.
\newblock \emph{The Annals of Probability}, 32\penalty0 (1A):\penalty0
  553--605, 2004.

\bibitem[Bai and Wang(2015)]{Bai2015}
Zhidong Bai and Chen Wang.
\newblock {A note on the limiting spectral distribution of a symmetrized
  auto-cross covariance matrix}.
\newblock \emph{Statistics and Probability Letters}, 96:\penalty0 333--340,
  2015.
\newblock ISSN 01677152.
\newblock \doi{10.1016/j.spl.2014.10.002}.
\newblock URL \url{http://dx.doi.org/10.1016/j.spl.2014.10.002}.

\bibitem[B¨¹hlmann et~al.(2014)B¨¹hlmann, Kalisch, and Meier]{B2014High}
Peter B¨¹hlmann, Markus Kalisch, and Lukas Meier.
\newblock High-dimensional statistics with a view toward applications in
  biology.
\newblock \emph{Annual Review of Statistics \& Its Application}, 1\penalty0
  (1):\penalty0 255--278, 2014.

\bibitem[Burkholder(1973)]{burkholder1973distribution}
Donald~L Burkholder.
\newblock Distribution function inequalities for martingales.
\newblock \emph{the Annals of Probability}, 1\penalty0 (1):\penalty0 19--42,
  1973.

\bibitem[Cai et~al.(2010)Cai, Zhang, and Zhou]{Cai2010OPTIMAL}
T.~Tony Cai, Cun~Hui Zhang, and Harrison~H. Zhou.
\newblock Optimal rates of convergence for covariance matrix estimation.
\newblock \emph{Annals of Statistics}, 38\penalty0 (4):\penalty0 2118--2144,
  2010.

\bibitem[Chen et~al.(2013)Chen, Xu, and Wu]{Chen2013Covariance}
Xiaohui Chen, Mengyu Xu, and Wei~Biao Wu.
\newblock Covariance and precision matrix estimation for high-dimensional time
  series.
\newblock \emph{Annals of Statistics}, 41\penalty0 (6):\penalty0 2994--3021,
  2013.

\bibitem[Donoho et~al.(2000)]{donoho2000high}
David~L Donoho et~al.
\newblock High-dimensional data analysis: The curses and blessings of
  dimensionality.
\newblock \emph{AMS math challenges lecture}, 1\penalty0 (2000):\penalty0 32,
  2000.

\bibitem[Fan et~al.(2008)Fan, Fan, and Lv]{Fan2008High}
Jianqing Fan, Yingying Fan, and Jinchi Lv.
\newblock High dimensional covariance matrix estimation using a factor model
  ¡î.
\newblock \emph{Journal of Econometrics}, 147\penalty0 (1):\penalty0 186--197,
  2008.

\bibitem[Geman(1980)]{geman1980limit}
Stuart Geman.
\newblock A limit theorem for the norm of random matrices.
\newblock \emph{The Annals of Probability}, 8\penalty0 (2):\penalty0 252--261,
  1980.

\bibitem[Goia and Vieu(2016)]{Goia2016An}
Aldo Goia and Philippe Vieu.
\newblock An introduction to recent advances in high/infinite dimensional
  statistics.
\newblock \emph{Journal of Multivariate Analysis}, 146\penalty0 (2):\penalty0
  1--6, 2016.

\bibitem[Johnstone(2001)]{Johnstone2001On}
Iain~M. Johnstone.
\newblock On the distribution of the largest eigenvalue in principal components
  analysis.
\newblock \emph{Annals of Statistics}, 29\penalty0 (2):\penalty0 295--327,
  2001.

\bibitem[Li et~al.(2008)Li, Genton, and Sherman]{Li2008Testing}
Bo~Li, Marc~G Genton, and Michael Sherman.
\newblock Testing the covariance structure of multivariate random fields.
\newblock \emph{Biometrika}, 95\penalty0 (4):\penalty0 813--829, 2008.

\bibitem[Lytova and Pastur(2009)]{LytovaP09C}
A.~Lytova and L.~Pastur.
\newblock {Central limit theorem for linear eigenvalue statistics of random
  matrices with independent entries}.
\newblock \emph{Annals of Probability}, 37\penalty0 (5):\penalty0 1778--1840,
  2009.
\newblock ISSN 00911798.
\newblock \doi{10.1214/09-AOP452}.
\newblock URL \url{http://projecteuclid.org/euclid.aop/1253539857}.

\bibitem[Marchenko and Pastur(1967)]{marchenko1967distribution}
Vladimir~Alexandrovich Marchenko and Leonid~Andreevich Pastur.
\newblock Distribution of eigenvalues for some sets of random matrices.
\newblock \emph{Matematicheskii Sbornik}, 114\penalty0 (4):\penalty0 507--536,
  1967.

\bibitem[Meinshausen et~al.(2006)Meinshausen, B{\"u}hlmann,
  et~al.]{meinshausen2006high}
Nicolai Meinshausen, Peter B{\"u}hlmann, et~al.
\newblock High-dimensional graphs and variable selection with the lasso.
\newblock \emph{The annals of statistics}, 34\penalty0 (3):\penalty0
  1436--1462, 2006.

\bibitem[Mitchell and Gumpertz(2003)]{Mitchell2003Spatio}
Matthew~W. Mitchell and Marcia~L. Gumpertz.
\newblock Spatio-temporal prediction inside a free-air co2enrichment system.
\newblock \emph{Journal of Agricultural Biological \& Environmental
  Statistics}, 8\penalty0 (3):\penalty0 310--327, 2003.

\bibitem[Patterson et~al.(2006)Patterson, Price, and
  Reich]{Patterson2006Population}
Nick Patterson, Alkes~L. Price, and David Reich.
\newblock Population structure and eigenanalysis.
\newblock In \emph{Plos Genetics}, pages 2074--2093, 2006.

\bibitem[Paul and Silverstein(2009)]{paul2009no}
Debashis Paul and Jack~W Silverstein.
\newblock No eigenvalues outside the support of the limiting empirical spectral
  distribution of a separable covariance matrix.
\newblock \emph{Journal of Multivariate Analysis}, 100\penalty0 (1):\penalty0
  37--57, 2009.

\bibitem[Shcherbina(2011)]{LytovaP09Ca}
M.~Shcherbina.
\newblock {Central limit theorem for linear eigenvalue statistics of the Wigner
  and sample covariance random matrices}.
\newblock \emph{Journal of Mathematical Physics, Analysis, Geometry},
  7\penalty0 (2):\penalty0 176--192, 2011.
\newblock ISSN 18129471.
\newblock \doi{10.1007/s00184-008-0212-5}.
\newblock URL
  \url{http://www.springerlink.com/index/10.1007/s00184-008-0212-5}.

\bibitem[Tikhomirov(2015)]{Tikhomirov2015The}
Konstantin Tikhomirov.
\newblock The limit of the smallest singular value of random matrices with
  i.i.d. entries.
\newblock \emph{Advances in Mathematics}, 284:\penalty0 1--20, 2015.

\bibitem[Tracy and Widom(2002)]{tracy2002distribution}
Craig~A Tracy and Harold Widom.
\newblock Distribution functions for largest eigenvalues and their
  applications.
\newblock \emph{arXiv preprint math-ph/0210034}, 2002.

\bibitem[Verdu(2002)]{Verdu2002Spectral}
S.~Verdu.
\newblock Spectral efficiency in the wideband regime.
\newblock \emph{Information Theory IEEE Transactions on}, 48\penalty0
  (6):\penalty0 1319--1343, 2002.

\bibitem[Yin(2018)]{yin2018no}
Yanqing Yin.
\newblock No eigenvalues outside the limiting support of the spectral
  distribution of general sample covariance matrices.
\newblock \emph{arXiv preprint arXiv:1801.03319}, 2018.

\bibitem[Yin et~al.(1988)Yin, Bai, and Krishnaiah]{Yin1988}
Y.Q. Yin, Z.D. Bai, and P.R. Krishnaiah.
\newblock On the limit of the largest eigenvalue of the large dimensional
  sample covariance matrix.
\newblock \emph{Probability Theory and Related Fields}, 78\penalty0
  (4):\penalty0 pp. 509--521, 1988.
\newblock ISSN 0178-8051.
\newblock \doi{10.1007/BF00353874}.
\newblock URL \url{http://dx.doi.org/10.1007/BF00353874}.

\bibitem[Zhang(2006)]{zhang}
L.~X. Zhang.
\newblock \emph{Spectral analysis of large dimensional random matrices}.
\newblock PhD thesis, National University of Singapore, 2006.

\bibitem[Zheng et~al.(2015)Zheng, Bai, and Yao]{ZhengB15S}
Shurong Zheng, Zhidong Bai, and Jianfeng Yao.
\newblock {Substitution principle for CLT of linear spectral statistics of
  high-dimensional sample covariance matrices with applications to hypothesis
  testing}.
\newblock \emph{Annals of Statistics}, 43\penalty0 (2):\penalty0 546--591,
  2015.
\newblock ISSN 00905364.
\newblock \doi{10.1214/14-AOS1292}.

\bibitem[Zheng et~al.(2017{\natexlab{a}})Zheng, Bai, and Yao]{ZhengB17C}
Shurong Zheng, Zhidong Bai, and Jianfeng Yao.
\newblock {CLT for eigenvalue statistics of large-dimensional general Fisher
  matrices with applications}.
\newblock \emph{Bernoulli}, 23\penalty0 (2):\penalty0 1130--1178,
  2017{\natexlab{a}}.
\newblock ISSN 1350-7265.
\newblock \doi{10.3150/15-BEJ772}.
\newblock URL \url{http://projecteuclid.org/euclid.bj/1486177395}.

\bibitem[Zheng et~al.(2017{\natexlab{b}})Zheng, Bai, Yao, and
  Zhu]{zheng2017clt}
Shurong Zheng, Zhidong Bai, Jianfeng Yao, and Hongtu Zhu.
\newblock Clt for linear spectral statistics of large dimensional sample
  covariance matrices with dependent data.
\newblock \emph{arXiv preprint arXiv:1708.03749}, 2017{\natexlab{b}}.

\end{thebibliography}
\end{document}